\documentclass{amsart}
\usepackage{etex}

\usepackage{fixltx2e}
\usepackage{tikz}
\usepackage{tikz-cd}
\usepackage{adjustbox}
\usepackage{enumitem}

\renewcommand{\mathbb}{\mathbf}
\renewcommand{\ell}{l}

\def\Ad{\mathrm{Ad}}
\newcommand{\Rbar}{\overline{R}}
\DeclareMathOperator{\Lie}{Lie}

\newcommand{\CNL}{\operatorname{CNL}}
\newcommand{\psibar}{\overline{\psi}}

\newcommand{\tv}{{\widetilde{v}}}

\DeclareMathOperator{\SL}{SL}
\def\sl{\mathfrak{sl}}

\def\ad{\mathrm{ad}}

\newcommand\Gder{G^{\mathrm{der}}}
\newcommand\Gab{G^{\mathrm{ab}}}

\newcommand{\mubar}{\overline{\mu}}

\def\wotimes {\widehat{\otimes}}

\def\sl{\mathfrak{sl}}

\def\fz{\mathfrak{z}}
\def\fg{\mathfrak{g}}
\def\fgl{\mathfrak{gl}}

\def\der{\mathrm{der}}

\usepackage{amssymb,amsmath,amsfonts,amsthm,epsfig,amscd}
\usepackage[all,cmtip,poly]{xy}
\usepackage{color}

\newcommand{\univ}{{\operatorname{univ}}}
\newcommand{\To}{\longrightarrow}
\newcommand{\isoto}{\stackrel{\sim}{\To}}

\newcommand{\Rep}{\operatorname{Rep}}
\newcommand{\Isom}{\operatorname{Isom}}

\newcommand{\Kbar}{\overline{K}}

\newtheorem{theorem}[subsubsection]{Theorem}
\newtheorem{thm}[subsubsection]{Theorem}
\newtheorem{lemma}[subsubsection]{Lemma}
 \newtheorem{ithm}{Theorem}

\newtheorem{lem}[subsubsection]{Lemma}

\newtheorem{cor}[subsubsection]{Corollary}

\newtheorem{prop}[subsubsection]{Proposition}

\theoremstyle{definition}

\newtheorem{defn}[subsubsection]{Definition}

\theoremstyle{remark}
\newtheorem{remark}[subsubsection]{Remark}
\newtheorem{rem}[subsubsection]{Remark}

\def\numequation{\addtocounter{subsubsection}{1}\begin{equation}}
\def\anummultline{\addtocounter{subsection}{1}\begin{multline}}

\newif\iffinalrun
\iffinalrun
\else
\fi

\iffinalrun
  \newcommand{\need}[1]{}
  \newcommand{\mar}[1]{}
\else
  \newcommand{\need}[1]{{\tiny *** #1}}
  \newcommand{\mar}[1]{\marginpar{\raggedright\tiny squirrel #1}}
\fi

\newcommand{\A}{\AA}

\newcommand{\F}{\FF}

\newcommand{\Q}{\QQ}

\newcommand{\Z}{\ZZ}

\newcommand{\m}{\frakm}

\renewcommand{\AA}{{\mathbb A}}

\newcommand{\FF}{{\mathbb F}}
\newcommand{\GG}{{\mathbb G}}

\newcommand{\QQ}{{\mathbb Q}}

\newcommand{\ZZ}{{\mathbb Z}}

\newcommand{\bv}{\ensuremath{\mathbf{v}}}

\newcommand{\cC}{{\mathcal C}}

\newcommand{\cG}{{\mathcal G}}

\newcommand{\cO}{{\mathcal O}}

\newcommand{\frakm}{\mathfrak{m}}

\newcommand{\Fbar}{\overline{\F}}
\newcommand{\Qbar}{\overline{\Q}}

\newcommand{\Fpbar}{\Fbar_p}

\newcommand{\Flbar}{\Fbar_{\ell}}

\newcommand{\Zl}{\Z_{\ell}}

\newcommand{\Ql}{\Q_{\ell}}
\newcommand{\Qp}{\Q_p}

\newcommand{\Qlbar}{\Qbar_{\ell}}

\DeclareMathOperator{\Aut}{Aut}

\DeclareMathOperator{\End}{End}

\DeclareMathOperator{\Fil}{Fil}
\DeclareMathOperator{\gr}{gr}
\DeclareMathOperator{\Gal}{Gal}
\DeclareMathOperator{\GL}{GL}

\DeclareMathOperator{\Hom}{Hom}

\DeclareMathOperator{\im}{im}

\DeclareMathOperator{\Mod}{Mod}

\DeclareMathOperator{\Proj}{Proj}

\DeclareMathOperator{\Vect}{Vec}

\def\sl{\mathfrak{sl}}

\DeclareMathOperator{\Spec}{Spec}

\DeclareMathOperator{\Spf}{Spf}

\DeclareMathOperator{\WD}{WD}

\newcommand{\ab}{\mathrm{ab}}

\newcommand{\st}{\mathrm{st}}

\newcommand{\rhobar}{\overline{\rho}}

\newcommand{\into}{\hookrightarrow}

\newcommand{\Gm}{\GG_m}

\newcommand{\rbar}{\overline{r}}
\newcommand{\sbar}{\overline{s}}

\newcommand{\Res}{\operatorname{Res}}

\usepackage[bookmarksopen,bookmarksdepth=2]{hyperref}

\begin{document}
\title[$G$-valued local deformation rings and global lifts]{$G$-valued
  local deformation rings and global lifts}

\author[R. Bellovin]{Rebecca Bellovin} \email{r.bellovin@imperial.ac.uk} \address{Department of
  Mathematics, Imperial College London,
  London SW7 2AZ, UK}

\author[T. Gee]{Toby Gee} \email{toby.gee@imperial.ac.uk} \address{Department of
  Mathematics, Imperial College London,
  London SW7 2AZ, UK}

\thanks{The second author was
  supported in part by a Leverhulme Prize, EPSRC grant EP/L025485/1, Marie Curie Career
  Integration Grant 303605, by
  ERC Starting Grant 306326, and by a Royal Society Wolfson Research Merit Award.}

\maketitle
\begin{abstract}We study $G$-valued Galois deformation rings with
  prescribed properties, where
  $G$ is an arbitrary (not necessarily connected) reductive group over an
  extension of~$\Zl$ for some prime~$l$. In particular, for the Galois
  groups of $p$-adic local fields (with $p$ possibly equal to~$l$) we prove
  that these rings are generically regular, compute their
  dimensions, and show that functorial operations on Galois
  representations give rise to well-defined maps between the sets of irreducible
  components of the corresponding deformation rings. We use these
  local results to prove  lower bounds on the dimension of 
  global deformation rings with
  prescribed local properties. Applying our results to unitary groups, we improve results in
  the literature on
  the existence of lifts of mod~$l$ Galois representations,  and on the weight part of Serre's conjecture.
\end{abstract}
\setcounter{tocdepth}{1}
\tableofcontents
\section{Introduction}\label{sec: introduction}The study of Galois
deformation rings was initiated in~\cite{MR1012172}, and was crucial
to the proof of Fermat's Last Theorem in~\cite{MR1333035}, and in
particular to the modularity lifting theorems proved
in~\cite{MR1333035,MR1333036}. Many generalisations of these
modularity lifting theorems have been proved over the last 25 years,
and it has become increasingly important to consider Galois
representations valued in reductive groups other than~$\GL_n$. From
the point of view of the Langlands program, it is particularly
important to be able to use disconnected groups, as the $L$-groups of
non-split groups are always disconnected. In particular, it is
important to study the structure of local deformation rings for
general reductive groups, and to prove lifting results for global
deformation rings. We briefly review the history of such results in
Section~\ref{subsec: historical overview}, but we firstly explain the
main theorems of this paper.

We begin with a result about local deformation rings. Let $K/\Qp$ be
a finite extension, let~$\cO$ be the ring of integers in a finite
extension $E$ of~$\Ql$ with residue field $\F$, where $l$ is possibly equal to~$p$, and let $G$ be
a (not necessarily connected) reductive group over~$\cO$. Given a
representation $\rhobar:\Gal_K\to G(\F)$, we consider liftings
of~$\rhobar$ of some inertial type~$\tau$, and in the case~$l=p$, some
$p$-adic Hodge type~$\mathbf{v}$. There is a corresponding universal
framed deformation ring~ $R_{\rhobar}^{\square,\tau,\mathbf{v}}$, and
we prove the following result (as well as a variant for ``fixed
determinant $\psi$'' deformations).

\begin{ithm}[Thm.\ \ref{thm: dense set of very smooth points}] \label{thm:
    dimensions etc l equals p intro verison}
  Fix an inertial type~$\tau$, and if~$l=p$ then fix a $p$-adic Hodge
  type~$\mathbf{v}$. Then $R_{\rhobar}^{\square,\tau,\mathbf{v}}[1/l]$
  is generically regular. In addition, $R_{\rhobar}^{\square,\tau,\mathbf{v}}$ is equidimensional of
  dimension~$1+\dim_E G+\delta_{l=p}\dim_E(\Res_{E\otimes K/E}G)/P_{\mathbf{v}}$, and $R_{\rhobar}^{\square,\tau,\mathbf{v},\psi}$ is equidimensional of
  dimension~$1+\dim_E \Gder + \delta_{l=p}\dim_E(\Res_{E\otimes K/E}G)/P_{\mathbf{v}}$.
\end{ithm}
(We are abusing notation here; $P_{\mathbf{v}}$ is a $(\Res_{E\otimes
  K/E}G)_{\overline E}^\circ$-conjugacy class of parabolic subgroups
of $\Res_{E\otimes K/E}G$, and we choose a representative defined over
$E$ to compute the dimension of the quotient.) We are also able to
describe the regular locus
of~$R_{\rhobar}^{\square,\tau,\mathbf{v}}[1/l]$ precisely in terms of
the corresponding Weil--Deligne representations; see
Corollary~\ref{cor: smooth points given by WD rep condition}. In the
case that~$G=\GL_n$ and $l=p$ this is a theorem of
Kisin~\cite{MR2373358}, and results for general groups (but with more
restrictive hypotheses than those of Theorem~\ref{thm:
    dimensions etc l equals p intro verison}) were previously proved
  by Balaji~\cite{MR3152673} and R.B.\ ~\cite{2014arXiv1403.1411B}.

Combining Theorem~\ref{thm:
    dimensions etc l equals p intro verison} with results of Balaji~\cite{MR3152673}, we obtain the following result (see Section~\ref{sec: global deformation
  rings} for any unfamiliar notation or terminology --- in particular,
$\fg_\F^0$ denotes the $\F$-points of the Lie algebra of the derived
subgroup of~$G$); in the case of
potentially crystalline representations, this is the main result of~\cite{MR3152673}.

\begin{ithm}[Prop. \ref{prop: global deformation ring with types is positive
    dimensional}] \label{thm: intro global Krull}
Let $F$ be totally real, assume that~$l>2$, let~$S$ be a finite set of
places of~$F$  containing all places dividing~$l\infty$, and let
$\rhobar:\Gal_{F,S}\to G(\Flbar)$ be a representation admitting a
universal deformation ring. Fix inertial types at all places $v\in S$,
and Hodge types at all places $v|l$, in such a way that the
corresponding local deformation rings are nonzero, and let~$R^\univ$ denote the
corresponding fixed determinant universal deformation ring for~$\rhobar$.

Assume that~$\rhobar$ is odd, and that
$H^0(\Gal_{F,S},(\fg_\F^0)^*(1))=0$.  Suppose also that for each
place~$v|l$ the corresponding Hodge type is regular.  Then
~$R^\univ$ 
has Krull dimension at least one.
\end{ithm}

We use this result to improve on some results about automorphic forms
on unitary groups proved using the methods of~\cite{BLGGT}. 
  Beginning with
the paper~\cite{CHT}, Galois deformations were considered for
representations valued in a certain disconnected group~$\cG_n$, whose
connected component is~$\GL_n\times\GL_1$ (this group is related to
the $L$-group of a unitary group, see~\cite[\S8]{MR3444225}).
In the case that~$G=\cG_n$, Theorem~\ref{thm: intro global Krull} generalises~\cite[Prop.\
1.5.1]{BLGGT}, removing restrictions on the places in~$S$ (which were
chosen to split in the splitting field of the corresponding unitary
group, in order to reduce the local deformation theory to the~$\GL_n$
case).

 We deduce
corresponding improvements to a number of results proved using the
methods of~\cite{BLGGT}, such as 
%
%
the following general result about Serre weights for
rank two unitary groups, which removes a ``split ramification''
hypothesis on the ramification of~$\rbar$ at places away from~$l$.

\begin{ithm}[Theorem~\ref{thm: weight part of Serre for U2}]\label{thm: weight part of Serre for U2 intro version}
  Let~$F$ be an imaginary CM field with maximal totally real
  subfield~$F^+$, and suppose that~$F/F^+$ is unramified at all finite
  places, that each place of~$F^+$ above~$l$ splits in~$F$, and
  that~$[F^+:\Q]$ is even. Suppose that~$l$ is odd, that
  $\rbar:G_{F^+}\to\cG_2(\Flbar)$ is irreducible and modular, and
  that~$\rbar(G_{F(\zeta_l)})$ is adequate.

Then the set of Serre weights for which~$\rbar$ is modular is exactly
the set of weights given by the sets~$W(\rbar|_{G_{F_v}})$, $v|l$.
  \end{ithm}
(See Remark~\ref{rem: non quasi split groups} for a discussion of
further improvements to this result that could be made by techniques
orthogonal to those of this paper.) 
These results are also crucially applied in the forthcoming paper~\cite{CEGglobalrealisable},
where they are used to construct lifts of representations valued
in~$\cG_n$ which have prescribed ramification at certain inert
places.

\subsection{A brief historical overview}\label{subsec: historical overview}We
now give a very brief overview of some of the developments in the
deformation theory of Galois representations, which was introduced for
representations valued in~$\GL_n$ by Mazur in the
paper~\cite{MR1012172}; we apologise for the many important papers
that we do not discuss here for reasons of space. The abstract parts
of this deformation theory were generalised to arbitrary reductive
groups in~\cite{MR1643682}. However, for applications to the Langlands
program (and in particular to proving automorphy lifting theorems),
one needs to study conditions on Galois deformations coming from
$p$-adic Hodge theory.

This was initially done in a somewhat ad-hoc fashion, mostly for the
group~$\GL_2$ and mostly for conditions coming from $p$-divisible
groups, culminating in the paper~\cite{MR1839918}, which used a
detailed study of some particular such deformation rings to complete
the proof of the Taniyama--Shimura--Weil conjecture. This situation
changed with the paper~\cite{MR2373358}, which proved the existence of
local deformation rings for~$\GL_n$ corresponding to general $p$-adic
Hodge theoretic conditions (namely being potentially crystalline or
semi-stable of a given inertial type), and determined the structure of
their generic fibres, in particular showing that they are generically
regular, and computing their dimensions.

The results of ~\cite{MR2373358} were generalised in~\cite{MR3152673}
to the case of general reductive groups~$G$ under the hypothesis of
being potentially crystalline, and in~\cite{2014arXiv1403.1411B} to
the case that~$G$ is connected, and the inertial type is totally
ramified. In the potentially crystalline case the generic fibres of
the deformation rings can easily be shown to be regular, whereas in the
potentially semistable case, one has to gain some control of the
singularities, which is why there are additional restrictions in the
theorems of~\cite{2014arXiv1403.1411B}. Our Theorem~\ref{thm:
  dimensions etc l equals p intro verison} is a common generalisation
of these results to the case that~$G$ is possibly disconnected, and
the representation is potentially semistable with no condition on the
inertial type. (We also simultaneously handle the case that $p\ne l$.)

Another important application of Galois deformation theory to the
Langlands program is to prove results showing that mod~$l$
representations of the Galois groups of number fields admit lifts to
characteristic zero with prescribed local properties; for example,
such results were an important part of Khare--Wintenberger's proof of
Serre's conjecture. The first such results were proved by Ramakrishna
for~$\GL_2$ \cite{MR1935843}, and this method has now been generalised to
a wide class of reductive groups; see in particular~\cite{MR3529115},~\cite{2016arXiv161204237B}, and~\cite{2018arXiv180710743B}. However, it has two disadvantages: it
loses control of the local properties at a finite set of places, and
it only applies in cases where formally smooth deformation rings exist.

A different approach was found in the paper~\cite{MR2480604}, which
observed that in conjunction with the theory of potential modularity,
such lifting results can be deduced from a lower bound on the Krull
dimension of a global deformation ring, which was provided by the
results of~\cite{MR1679172}. In the paper~\cite{MR2459302}, Kisin
improved on the results of~\cite{MR1679172}, proving a result about
presentations of global deformation rings over local ones for~$\GL_n$,
and deducing a lower bound on the dimensions of global deformation
rings. These results were generalised to general reductive groups by
Balaji~\cite{MR3152673}, and given our Theorem~\ref{thm: dimensions
  etc l equals p intro verison}, results such as Theorem~\ref{thm:
  intro global Krull} are essentially immediate from Balaji's.

Finally, the paper~\cite{2017arXiv170807434} (independently and
contemporaneously) proved similar results to those of this paper in
the case~$l\ne p$ by a related but different method; rather than
constructing a large enough supply of unobstructed points, as in this
paper, they instead show that all points can be path connected to
unobstructed points. We refer to the introduction
to~\cite{2017arXiv170807434} for a fuller discussion of the difference
between the approaches.

\subsection{Some details}\label{subsec: some details}
We now explain our local results (and their proofs) in more detail. Theorem~\ref{thm:
    dimensions etc l equals p intro verison} is a generalisation of~\cite[Thm.\ 3.3.4]{MR2373358}, which proves the result
in the case $l=p$ and~$G=\GL_n$. It was previously adapted to the
(much easier)
case $G=\GL_n$ and $l\ne p$ in~\cite{MR2785764} by using Weil--Deligne
representations in place of the filtered $(\varphi,N)$-modules
employed in~\cite{MR2373358}. It was also generalised
in~\cite{2014arXiv1403.1411B} to the case
that~$G$ is connected, $l=p$, and~$\tau$ is totally ramified. Our approach is in some sense a
synthesis of the approaches of~\cite{MR2785764,2014arXiv1403.1411B},
in that we treat the cases $l\ne p$ and $l=p$ essentially
simultaneously, by using Weil--Deligne representations.

We briefly explain our approach, which in broad outline follows that
of~\cite{MR2373358}. It is relatively straightforward (by passing from
Galois representations to Weil--Deligne representations using
Fontaine's constructions in the case $l=p$, and Grothendieck's
monodromy theorem if $l\ne p$) to reduce Theorem~\ref{thm: dimensions
  etc l equals p intro verison} to analogous statements about moduli
spaces of Weil--Deligne representations over $l$-adic fields. These
moduli spaces admit an explicit tangent-obstruction theory given by an
analogue of Herr's complex computing Galois cohomology in terms of
$(\varphi,\Gamma)$-modules, and the key problem is to prove that the
$H^2$ of this complex generically vanishes. We can think of this~$H^2$
as a coherent sheaf over the moduli space, so by considering its
support, we can reduce to the problem of exhibiting sufficiently many
points at which the~$H^2$ vanishes (which turn out to be precisely the
regular points, which in a standard abuse of terminology we refer to
as ``smooth points'').

Our approach to exhibiting these points is related to that taken
in~\cite{2014arXiv1403.1411B}, in that it makes use of the theory of
associated cocharacters (see Section~\ref{subsec: very
  smooth points}), but it is more streamlined and conceptual (for
example, we do not need to  consider the case $N=0$ separately, as was
done in~\cite{2014arXiv1403.1411B}). Surprisingly (at least to us), it
is possible to construct all the smooth points that we need by
considering the single Weil--Deligne representation
$W_K\to\SL_2(\Qlbar)$ which is trivial on $I_K$, takes an arithmetic
Frobenius element of~$W_K$ to \[
  \begin{pmatrix}
    q^{1/2}&0\\0&q^{-1/2}
  \end{pmatrix}
\]where~$q$ is the order of the residue field of~$K$, and has \[N=
  \begin{pmatrix}
    0&1\\0&0
  \end{pmatrix}.\] 
It is easy to check
that this gives a smooth point of the moduli space of Weil--Deligne representations (while the point with the same
representation of~$W_K$ but with $N=0$ is not smooth). 

Returning to the case of general~$G$, suppose that the inertial
type~$\tau$ is trivial. If we consider a nilpotent element $N\in \Lie G$,
the theory of associated cocharacters allows us to construct a particular
homomorphism $\SL_2\to G$ taking $\begin{pmatrix} 0&1\\0&0
\end{pmatrix}$ to~$N$, and an elementary calculation using the
representation theory of~$\sl_2$  shows that the composition of
our fixed representation $W_K\to\SL_2(\Qlbar)$ with this homomorphism
defines a smooth point. We obtain further smooth points by
multiplication by elements of~$G(\Qlbar)$ of finite order, and this turns out
to give us all the smooth points we need (even when~$G$ is not
connected). (See Remark~\ref{rem: SL2 form of WD group} for an
interpretation of this construction in terms of the $\SL_2$ version of
the Weil--Deligne group.)

In the case of general~$\tau$ we reduce to the same situation by
replacing~$G$ by the normaliser in~$G$ of~$\tau$, which is also a
reductive group. This use of Weil--Deligne representations is what
allows us to remove the assumption made in~\cite{2014arXiv1403.1411B}
that the inertial type is totally ramified, which was used in order to
choose coordinates so that the inertial type $\tau$ was invariant
under 
Frobenius. (Similarly, it clarifies the calculations made
for~$\GL_n$ in \cite{MR2373358}, as the semilinear algebra becomes
linear algebra.) Under this assumption, when studying the structure of
the moduli space of $G$-valued $(\varphi,N,\tau)$-modules one could
exploit the fact that $\Phi$ was in the centralizer $Z_G(\tau)$ and
$N$ was in $\Lie Z_G(\tau)$.  Passing to Weil--Deligne representations
$r$ lets us argue similarly for general~$\tau$: a generator $\Phi$ of
the unramified quotient of the Weil group normalizes the inertial type
and $N$ is centralized by the inertial type.  Since
$Z_G(r|_{I_{L/K}})$ has finite index in the normalizer
$N_G(r|_{I_{L/K}})$, we see that $N$ is again in the Lie algebra of
the algebraic group containing $\Phi$.

In view of the functorial nature of our construction of smooth points,
we are able to produce points on each irreducible component of the
generic fiber of the deformation ring which are furthermore ``very
smooth'' in the sense that they give rise to smooth points after
restriction to any finite extension $K'/K$ (these points were called
``robustly smooth'' in~\cite{BLGGT} when $p\ne l$). In particular, the
images of such points on the corresponding deformation rings
for~$\Gal_{K'}$ lie on only one irreducible component, so that we obtain
a well-defined ``base change'' map between irreducible components. We
prove a similar result for the maps between deformation rings induced
by morphisms of algebraic groups $G\to G'$ (see Theorem~\ref{subsec:
  tensor products of components} for this, and for the case of base
change). In particular, this allows one to talk about taking tensor
products of components of deformation rings, which is frequently
convenient when applying the Harris tensor product trick; see for
example~\cite{CEGglobalrealisable}.

We end this introduction by explaining the structure of the paper. In
Section~\ref{sec: moduli of WD representations}, we prove our main
results about the structure of the moduli spaces of Weil--Deligne
representations; we explain the tangent-obstruction theory and exhibit
smooth points, and study the relationship with Galois
representations. In doing so we remove the connectedness hypothesis on
$G$ made in~\cite{2014arXiv1403.1411B}, by studying exact
tensor-filtrations on fiber functors for disconnected reductive
groups.  We do this via a functor of points approach, using the
dynamic approach to parabolic subgroups discussed
in~\cite[\textsection I.2.1]{predbook}. In Section~\ref{sec: local
  deformation rings} we deduce our results on the local structure of
Galois deformation rings, which we then combine with the results of~\cite{MR3152673} to
prove our lower bound on the dimension of a global deformation
ring in Section~\ref{sec:
  global deformation rings}. Finally, in Section~\ref{sec: unitary groups} we specialise
these results to the case of unitary groups.
\subsection{Acknowledgements}We would like to thank Matthew Emerton
for emphasising the importance of Weil--Deligne representations to
us, and for his comments on an earlier draft of this paper. We would
also like to thank Jeremy Booher, George Boxer, Stefan Patrikis, and Jacques
Tilouine for helpful conversations, and Brian Conrad,  Mark Kisin and
Daniel Le
for their comments on an earlier draft. We would like to thank the
referees for their careful reading of the paper and their many helpful
comments. 

\subsection{Notation and conventions}\label{subsec:
  notation}
All representations considered in this paper are assumed to be
continuous with respect to the natural topologies, and we will never
draw attention to this.

If~$K$ is a field then we write~$\Gal_K:=\Gal(\overline{K}/K)$ for its
absolute Galois group, where~$\overline{K}$ is a fixed choice of
algebraic closure; we will regard all algebraic extensions of~$K$ as
subfields of~$\overline{K}$ without further comment, so that in
particular we can take the compositum of any two such extensions. If $L/K$ is a Galois extension then we write
$\Gal_{L/K}:=\Gal(L/K)$, a quotient of~$\Gal_K$.  
If~$K$ is a number field and~$v$
is a place of~$K$ then we fix an embedding $\Kbar\into\Kbar_v$, so
that we have a homomorphism $\Gal_{K_v}\to\Gal_K$. If~$S$ is a finite
set of places of a number field~$K$, then we let~$K(S)$ be the maximal
extension of~$K$ (inside~$\Kbar$) which is unramified outside~$S$, and
write $\Gal_{K,S}:=\Gal(K(S)/K)$.

 If
$K/\Qp$ is a finite extension for some prime~$p$ then we write~$I_K$
for the inertia subgroup of~$\Gal_K$, $W_K$ for the Weil
group, and~$f_K$ for the inertial degree of~$K/\Qp$. We let~$\varphi$
denote the arithmetic Frobenius on~$\Fpbar$, so that we have an exact sequence
\[	1\rightarrow I_{K}\rightarrow W_K\rightarrow \langle \varphi^{f_K}\rangle \rightarrow 1,	\]
and we let $v:W_K\rightarrow \Z$ be
the function such that $v(g)=i$ if the image of $g$ modulo $I_{K}$ is
$\varphi^{if_K}$.
 Recall that a Weil--Deligne representation of~$W_K$ is a
pair $(r,N)$  consisting of a finite-dimensional
representation $r:W_K\to\End(V)$ and a (necessarily nilpotent)
endomorphism $N\in\End(V)$
satisfying \[\rho(g)N=p^{v(g)f_K}N\rho(g) \] for all $g\in
W_K$. 
\subsubsection{Parabolic subgroups}\label{subsec: parabolic subgroups}If~$G$
is a finite-type affine group scheme over~$A$,  and $\lambda:\Gm\to G$ is a
cocharacter of~$G$, then there is a
subgroup~$P_G(\lambda)$ of~$G$
associated to~$\lambda$ as follows. Following~\cite[\S I.2.1]{predbook}, for any
$A$-algebra~$A'$ we define the functors \[ P_G(\lambda)(A')=\{g\in G(A') |
  \lim_{t\rightarrow 0}\lambda(t)g\lambda(t)^{-1}\textrm{ exists}\},\]
and \[  U_G(\lambda)(A')=\{g\in P_G(\lambda)(A') |
  \lim_{t\rightarrow 0}\lambda(t)g\lambda(t)^{-1}=1\}.\] 
We also let $Z_G(\lambda)$ denote the
scheme-theoretic centralizer of $\lambda$.  All of these functors are representable by subgroup schemes of $G$, and they are smooth if $G$ is smooth.  By construction, the formation of $P_G(\lambda)$, $U_G(\lambda)$, and $Z_G(\lambda)$ commutes with base change on $A$.

The cocharacter $\lambda$ induces a grading on the Lie algebra $\mathfrak{g}:=\Lie G$.  Let $\mathfrak{g}_n:=\{v\in\mathfrak{g} : \Ad(\lambda(t))(v)=t^nv\}$ and let $\mathfrak{g}_{\geq 0}:=\oplus_{n\geq 0}\mathfrak{g}_n$.  Then $\Lie P_G(\lambda)=\mathfrak{g}_{\geq 0}$, $\Lie U_G(\lambda)=\mathfrak{g}_{\geq 1}$, and $\Lie Z_G(\lambda)=\mathfrak{g}_0$.

The multiplication map $Z_G(\lambda)\ltimes U_G(\lambda)\rightarrow
P_G(\lambda)$ is an isomorphism.  Furthermore, the fibers of
$U_G(\lambda)$ are unipotent and connected. If the morphism $G\rightarrow\Spec A$ has connected reductive fibers, then $P_G(\lambda)$ is a parabolic subgroup scheme with connected fibers, $U_G(\lambda)$ is its unipotent radical, and $Z_G(\lambda)$ is connected and reductive.

\subsubsection{Deformation rings}\label{subsec: deformation rings}Let $l$ be prime, and let~$\cO$ be
the ring of integers in a finite extension~$E/\Ql$ with residue
field~$\F$. Write~$\CNL_{\cO}$ for the category of complete local
noetherian $\cO$-algebras with residue field~$\F$. 

Let~$\Gamma$ be either the absolute Galois group~$\Gal_K$ of a finite
extension~$K$ of~$\Ql$ for some~$p$ (possibly equal to~$l$), or a
group~$\Gal_{K,S}$ where~$S$ is a finite set of places of a number field~$K$.

 Let~$G$ be a smooth affine group scheme over~$\cO$ whose geometric fibres are reductive (but not necessarily connected), and fix a
homomorphism $\rhobar:\Gamma\to G(\F)$.  A \emph{framed deformation}
of~$\rhobar$ to a ring $A\in\CNL_\cO$ is a homomorphism
$\rho:\Gamma\to G(A)$ whose reduction modulo~$\m_A$ is equal
to~$\rhobar$. The functor of framed deformations is represented by the
universal framed deformation $\cO$-algebra $R^{\square}_{\rhobar}$, an
object of~$\CNL_\cO$ (\cite[Thm.\ 1.2.2]{MR3152673}). 

Suppose from now on for the rest of the paper that the centre~$Z_G$
of~$G$ is smooth over~$\cO$. Write~$\fg_\F$ and~$\fz_\F$ for the
$\F$-points of the Lie algebras of~$G$ and~$Z_G$ respectively; $\Gamma$ acts on $\fg_\F$ via the adjoint action composed with $\rhobar$.
  A \emph{deformation} of~$\rhobar$ to~$A$
is a $(\ker(G(A)\to G(\F)))$-conjugacy class of framed deformations
of~$\rhobar$ to~$A$. If $H^0(\Gamma,\fg_\F)=\fz_\F$, then the functor
of deformations is represented by the universal framed deformation
$\cO$-algebra $R_{\rhobar}$, an object of~$\CNL_\cO$ (see~\cite[Thm.\
1.2.2]{MR3152673} or~\cite[Thm.\ 3.3]{MR1643682}, together with
Comment~(2) following~\cite[Thm.\ 3.3]{MR1643682}).

We will also consider ``fixed determinant'' versions of these (framed)
deformations rings.  Let~$\Gab$ and~$\Gder$ respectively denote the
abelianisation and derived subgroup of~$G$, and write $\ab:G\to\Gab$
for the natural map. Write~$\fg^0_\F$ for the
$\F$-points of the Lie algebra of~$\Gder$. Fix a homomorphism ~$\psi:\Gamma\to\Gab(\cO)$
such that $\ab\circ\rhobar=\psibar$. We
let~$R_{\rhobar}^{\square,\psi}$ (resp.\ $R_{\rhobar}^\psi$) denote
the quotient of~$R^\square_{\rhobar}$ (resp.\ $R_{\rhobar}$)
corresponding to (framed) deformations~$\rho$
with~$\ab\circ\rho=\psi$. 

We write~$G^\circ$ for the connected
component of~$G$ containing the identity. We will always consider
representations up to $G^\circ$-conjugacy, rather than $G$-conjugacy;
note that this is compatible with our definition of deformations, as
an element of $(\ker(G(A)\to G(\F)))$ is necessarily contained in
$G^\circ(A)$. 

We for the most part allow any coefficient field~$E$, although for
some constructions in $p$-adic Hodge theory we need to allow it to be
sufficiently large; we will comment when we do this. The effect of
replacing~$E$ with a finite extension~$E'$ with ring of
integers~$\cO'$ is simply to replace~$R_{\rhobar}^\square$ and~$R_{\rhobar}$
with~$R_{\rhobar}^\square\otimes_{\cO}\cO'$
and~$R_{\rhobar}\otimes_{\cO}\cO'$ respectively.

\section{Moduli of Weil--Deligne representations}\label{sec: moduli of
  WD representations}Let $K/\Qp$ be a finite extension, and let~$l$ be
a prime, possibly equal to~$p$. In this section we prove analogues for $l$-adic
Weil--Deligne representations of some results on moduli spaces of weakly admissible modules
from~\cite{MR2373358,2014arXiv1403.1411B}, and remove some hypotheses imposed in
those papers; in particular, we allow our groups to be disconnected, and
we work with arbitrary inertial types (rather than totally ramified
types). In the case that $l=p$ we relate our moduli spaces to those
for weakly admissible modules. In Section~\ref{sec: local deformation
  rings} we will use these results to study the generic fibers of
deformation rings in both the case $l=p$ and the case $l\ne p$. 
\subsection{Moduli of Weil--Deligne representations}Let $K/\Q_p$ be a finite extension,
and let 
$L/K$ be a finite Galois extension. As in Section~\ref{subsec:
  notation}, we let $E/\Ql$ be a finite extension
for some prime~$l$, with ring of integers~$\cO$. 
We also continue to let $G$ be a (not necessarily connected) reductive
group over~$\cO$; in fact, throughout this section we will be working
with~$l$ inverted, and we will write~$G$ for~$G_E$ without further
comment. We write~$\fg_E$ for the Lie algebra of~$G$. 


A morphism of $G$-torsors $f:D\rightarrow D'$ over an $E$-scheme $X$ is a morphism of the
underlying $X$-schemes which is equivariant for the action of $G_X$.
Such a
morphism is necessarily an isomorphism.  The $G$-equivariant
automorphisms of $D$, which we denote by $\Aut_G(D)$, form a group, and it makes sense to talk about homomorphisms $r:W_K\rightarrow \Aut_G(D)$.  We also define a sheaf of automorphism groups $\underline\Aut_G(D)$ over $X$; if $X'$ is an $X$-scheme, its $X'$-points are given by $\underline\Aut_G(D)(X'):=\Aut_G(D_{X'})$.  This is a representable functor, since $\underline\Aut_G(D)$ is \'etale-locally isomorphic to $G_X$, which is affine.  We abuse notation by writing $\underline\Aut_G(D)$ for the group scheme, as well.
\begin{defn}Let $G-\WD_E(L/K)$ be the category cofibered in groupoids over $E$-Alg whose fiber over an $E$-algebra $A$ is a $G$-torsor
  $D$ over $A$ together with a pair $(r,N)$, where now $r:W_K\rightarrow\Aut_G(D)$ is a representation of the Weil group such that $r|_{I_L}$ is trivial, $N\in \Lie\underline\Aut_G(D)$, and
  $N=p^{-v(g)f_K}\Ad(r(g))(N)$ 
  for all
  $g\in W_K$. 
\end{defn}

Requiring $D$ to be a trivial $G$-torsor equipped with a trivializing section lets us define a representable functor covering $G-\WD_E(L/K)$, as follows.  The exact sequence
\[	0\rightarrow I_K\rightarrow W_K\rightarrow \langle
  \varphi^{f_K}\rangle \cong \Z\rightarrow 0	\]
 is non-canonically split, and choosing a splitting is the same as
 choosing a lift $g_0\in W_K$ of $\varphi^{f_K}$.  Thus, to specify a
 representation $r:W_K\rightarrow \Aut_G(D)$, it suffices to specify
 $r|_{I_K}$ and $r(g_0)$ (which we denote $\Phi$).  Since we are
 interested in representations which are trivial on $I_L$, we may
 replace $r|_{I_K}$ with $r|_{I_{L/K}}$. For an $E$-algebra $A$, we let $\Rep_A I_{L/K}$ denote
 the set of $A$-linear representations of $I_{L/K}$ on~$G(A)$.
\begin{defn}
Choose $g_0\in W_K$ lifting $\varphi^{f_K}$.
We let  $Y_{L/K,\varphi,\mathcal N}$ be the functor on the category of
  $E$-algebras whose $A$-points are triples 
  \[ (\Phi,N,\tau)\in G(A)\times \fg_E(A)\times\Rep_A I_{L/K} \]
  which satisfy
  \begin{itemize}
  \item $N=p^{-f_K}\Ad(\Phi)(N)$,
  \item $\Phi\circ\tau(g)\circ\Phi^{-1}=\tau(g_0gg_0^{-1})$ for all $g\in I_{L/K}$, and
  \item $N=\Ad(\tau(g))(N)$ for all $g\in I_{L/K}$.
  \end{itemize}
\end{defn}
To go from $Y_{L/K,\varphi,\mathcal N}$ to $G-\WD_E(L/K)$, we
need to forget the trivializing section and also forget $g_0$; the
representation associated to $(\Phi,N,\tau)$ is given
by \[r(g_0^nh)=\Phi^n\tau(h)\] where $n\in \Z$ and $h\in I_K$.  

The functor~$Y_{L/K,\varphi,\mathcal N}$ is visibly represented by a finite-type affine
scheme over~$E$, and there is an action of $G$ on $Y_{L/K,\varphi,\mathcal N}$ given by changing the trivializing section; explicitly,
\[	a\cdot (\Phi,N,\{\tau(g)\}_{g\in I_{L/K}}):= (a\Phi a^{-1}, \Ad(a)(N), \{a\tau(g)a^{-1}\}_{g\in I_{L/K}}).	\]
Recall that if $Z$ is an $E$-scheme equipped with a left-action of an algebraic group $H$ over $E$, then for any $E$-scheme $S$, the groupoid $[Z/H](S)$ over $S$ is the category
\[	[Z/H](S):=\{\text{Left }H\text{-bundle }D\rightarrow S\text{ and }H\text{-equivariant morphism }D\rightarrow Z\}.	\]
A morphism $f:D\rightarrow D'$ in this fiber category is a morphism of $H$-torsors over $S$.

\begin{lemma}\label{lem: stack quotient for G modules}
The quotient stack $[Y_{L/K,\varphi,\mathcal N}/G]$ is equivalent to the groupoid~$ G-\WD_E(L/K)$.
\end{lemma}
\begin{proof}
We choose $g_0\in W_K$ lifting $\varphi^{f_K}$.  
 Given an $A$-valued point of $G-\WD_E(L/K)$ with
  underlying $G$-torsor $D$, the base change
  $D\times_AD\rightarrow D$ (which is projection on the first factor)
  is a trivial $G$-torsor (with
  $G$ acting on the second factor).  The identity
  morphism $D\xrightarrow{\sim}D$ induces a canonical trivializing
  section $D\rightarrow D\times_AD$, namely the diagonal.  Pulling
  back $r$ and $N$ to $D\times_AD$,
  writing them in coordinates (with respect to the trivializing
  section), and writing $\tau:=r|_{I_{L/K}}$ and $\Phi:=r(g_0)$ gives us a morphism $D\rightarrow Y_{L/K,\varphi,\mathcal N}$.

  We need to check that the morphism $D\rightarrow Y_{L/K,\varphi,\mathcal N}$
  is $G$-equivariant.  If $A'$ is an $A$-algebra,
  the morphism $D\rightarrow Y_{L/K,\varphi,\mathcal N}$ carries $x\in D(A')$
  to the fiber of $(\Phi,N,\tau)$ over $x$.  The fiber of
  $D\times_AD\rightarrow D$ over $x$ is a copy of $D_{A'}$, together
  with a section (defined by taking the fiber of the diagonal over
  $x$).  If $g\in G(A')$, the fiber of
  $D\times_AD\rightarrow D$ over $g\cdot x$ is also a copy of
  $D_{A'}$, but the section has been multiplied by $g$.  Thus, our
  ``change-of-basis'' formula for triples
  $(\Phi,N,\tau)$ implies that the morphism
  $D\rightarrow Y_{L/K,\varphi,\mathcal N}$ is
  $G$-equivariant, as required.
\end{proof}

Similarly, we let $Y_{L/K,\mathcal N}$ denote the functor on the category of
$E$-algebras parametrizing pairs \[ (N,\tau)\in \fg_E(A)\times\Rep_A I_{L/K} \] such that
$N={\Ad}(\tau(g))(N)$ for all $g\in I_{L/K}$; and we let $Y_{L/K}$ be the functor on the category of $E$-algebras,
whose $A$-points are $\Rep_A I_{L/K}$.

Let $K'/K$ be a finite extension, and write $L'/K'$ for the compositum
of $K'$ and $L$. Then $L'/K'$ is Galois, with Galois group
$\Gal_{L'/K'}\subset\Gal_{L/K}$.  
There are versions of the above functors for $L'/K'$ which we
write $Y_{L'/K',\varphi,\mathcal N}$, $Y_{L'/K',\mathcal N}$, and
$Y_{L'/K'}$. Restriction of Weil--Deligne representations from~$W_K$
to~$W_{K'}$ induces morphisms $Y_{L/K,\varphi,\mathcal N}\to
Y_{L'/K',\varphi,\mathcal N}$, $Y_{L/K,\mathcal N}\to
Y_{L'/K',\mathcal N}$ and $Y_{L/K}\to
Y_{L'/K'}$.


\subsection{A tangent-obstruction theory for $G-\WD_E(L/K)$}
Choose an object $D_A\in G-\WD_E(L/K)$ with coefficients in an
$E$-algebra $A$, and let $\ad D_A$ denote the Weil--Deligne module induced on $\Lie\underline\Aut_GD_A$.  Choose $g_0\in W_K$ which lifts $\varphi^{f_K}$ and
write $\Phi:=r(g_0)$, let $\Ad(\Phi)$ denote the action on~$\ad
D_A$ given by differentiating the homomorphism $\underline\Aut_GD_A\rightarrow \underline\Aut_GD_A$ given by $g\mapsto \Phi g\Phi^{-1}$, and let $\ad_N$
act by $x\mapsto [N,x]$.  If $G=\GL_n$ and $D_A$ is the trivial torsor, these actions become $x\mapsto \Phi\circ x\circ\Phi^{-1}$ and $x\mapsto N\circ x-x\circ N$, respectively.  Then  we have an anti-commutative
diagram
\[	\xymatrix@C=5pc{
(\ad D_A)^{I_{L/K}}\ar[d]^{\ad_N}\ar[r]^{1-\Ad(\Phi)}& (\ad D_A)^{I_{L/K}}\ar[d]^{\ad_N}	\\
(\ad D_A)^{I_{L/K}}\ar[r]^{p^{-f_K}\Ad(\Phi)-1}&	(\ad D_A)^{I_{L/K}}
}	\]
Here $g\in I_{L/K}$ acts on $\ad D_A$ via $\Ad(\tau(g))$; note that
the minus sign in $p^{-f_K}$ arises because~$g_0$ is a lift of
arithmetic Frobenius.  This diagram does not depend on our choice of $g_0$, because any two lifts of $\varphi^{f_K}$ differ by an element of $I_{L/K}$, which acts trivially on $(\ad D_A)^{I_{L/K}}$.

The total complex
$C^\bullet(D_A)$ of this double complex controls the deformation
theory of objects of $G-\WD_E(L/K)$.  We write~$H^i(\ad D_A)$ for the
cohomology groups of~$C^\bullet(D_A)$. The following result will be proved
in a very similar way to~\cite[Proposition 3.1.2]{MR2373358}, which is
an analogous result for semilinear representations in the case~$G=\GL_n$.
\begin{prop}\label{prop: deformations of WD repns controlled by complex}
Let $A$ be a local $E$-algebra with maximal ideal $\mathfrak{m}_A$ and let $I\subset A$ be an ideal with $I\mathfrak{m}_A=(0)$.  Let $D_{A/I}$ be an object of $G-\WD_E(L/K)$ with coefficients in $A/I$, with Weil--Deligne representation $(\overline r,\overline N)$. Then
\begin{enumerate}
\item	if $H^2(\ad D_{A/\mathfrak{m}_A})=0$, then there exists an object
  $D_A$ in $G-\WD_E(L/K)$ with coefficients in $A$, such that
  $(A/I)\otimes_AD_A\cong D_{A/I}$, and
\item	the set of isomorphism classes of liftings of $D_{A/I}$ to
  $D_A$ is either empty or a torsor under
  $I\otimes_{A/\mathfrak{m}_A}H^1(\ad D_{A/\mathfrak{m}_A})$.
\end{enumerate}
\end{prop}
We begin by proving a preliminary lemma.
\begin{lemma}\label{lem: Weil representations lift}
Let $D_A$ be a $G$-torsor over $A$, and suppose there is a representation $\overline r:W_K\rightarrow \Aut_G(D_{A/I})$ such that $\overline r|_{I_L}$ is trivial.  Then there is a representation $r:W_K\rightarrow \Aut_G(D_A)$ such that $r|_{I_L}$ is trivial and $r$ lifts $\overline r$.  Moreover, the set of infinitesimal automorphisms of $r$ \emph{(}as a lift of $\overline r$\emph{)} is a torsor under $H^0(W_K/I_{L},I\otimes_{A/\mathfrak{m}_A}\ad D_{A/\mathfrak{m}_A}^{I_{L}})=I\otimes_{A/\mathfrak{m}_A}\ad D_{A/\mathfrak{m}_A}^{W_K}$, and the set of lifts of $\overline r$ is a torsor under $H^1(W_K/I_{K},I\otimes_{A/\mathfrak{m}_A}\ad D_{A/\mathfrak{m}_A}^{I_{L/K}})$.
\end{lemma}
\begin{proof}
An isomorphism $\overline{f}:D_{A/I}\rightarrow D_{A/I}$ lifts to an isomorphism $f:D_A\rightarrow D_A$, and the set of such lifts is a torsor under either a left- or right-action of $H^0(A,I\otimes_{A/\mathfrak{m}_A}D_{A/\mathfrak{m}_A})$ by~\cite[Lemma 3.5]{2014arXiv1403.1411B}.  Thus, for each $g\in W_K$, we can lift the map $\overline{r}(g):D_{A/I}\rightarrow D_{A/I}$ to an isomorphism $r(g):D_A\rightarrow D_A$.  

The assignment
\[	(g_1,g_2)\mapsto r(g_1)r(g_2)r(g_1g_2)^{-1}	\]
is a $2$-cocycle of $W_K/I_L$ valued in
$I\otimes_{A/\mathfrak{m}_A}\ad D_{A/\mathfrak{m}_A}$.  Since we are in
characteristic~ $0$, and $I_{L/K}$ is a finite group,  the
Hochschild--Serre spectral sequence implies that for each $i>0$, we
have an isomorphism 
\[ H^i(W_K/I_K,I\otimes_{A/\mathfrak{m}_A}\ad D_{A/\mathfrak{m}_A}^{I_{L/K}})\xrightarrow{\sim} H^i(W_K/I_L,I\otimes_{A/\mathfrak{m}_A}\ad D_{A/\mathfrak{m}_A}).\]

In particular, 
\[	H^2(W_K/I_L,I\otimes_{A/\mathfrak{m}_A}\ad D_{A/\mathfrak{m}_A})\cong H^2(\widehat{\Z},I\otimes_{A/\mathfrak{m}_A}\ad D_{A/\mathfrak{m}_A}^{I_{L/K}})=0,	\]
so $\overline r$ lifts to a representation $r:W_K\rightarrow
\Aut_G(D_A)$ with $r|_{I_L}=0$, as claimed.

An isomorphism $f:D_A\rightarrow D_A$ is an infinitesimal automorphism of $r$ if and
only if it is the identity modulo $I$ and $r(g)\circ f=f\circ r(g)$ for all $g\in W_K$.  Equivalently, $f$ is an element of $I\otimes_{A/\mathfrak{m}_A}\ad D_{A/\mathfrak{m}_A}$ fixed by $W_K$, and since $I$ is a vector space over $A/\mathfrak{m}_A$, this is equivalent to $f\in I\otimes_{A/\mathfrak{m}_A}\ad D_{A/\mathfrak{m}_A}^{W_K}$, as desired.

Finally, if $r':W_K\rightarrow \Aut_G(D)$ is another such lift, then $g\mapsto r'(g)r(g)^{-1}$ is a $1$-cocycle of $W_K/I_L$ valued in $I\otimes_{A/\mathfrak{m}_A}\ad D_{A/\mathfrak{m}_A}$.  But $H^1(W_K/I_L,I\otimes_{A/\mathfrak{m}_A}\ad D_{A/\mathfrak{m}_A})\cong H^1(W_K/I_{K},I\otimes_{A/\mathfrak{m}_A}\ad D_{A/\mathfrak{m}_A}^{I_{L/K}})$, so we are done.
\end{proof}


\begin{proof}[Proof of Proposition~\ref{prop: deformations of WD repns controlled by complex}]
By~\cite[Lemma 3.4]{2014arXiv1403.1411B}, the underlying $G$-torsor
$D_{A/I}$ lifts to a $G$-torsor $D_A$ over $\Spec A$, and $D_A$ is
unique up to isomorphism, and by Lemma~\ref{lem: Weil representations lift}, $\overline r$ lifts to a representation $r:W_K\rightarrow \Aut_G(D_A)$.  Moreover, by~\cite[Lemma 3.7]{2014arXiv1403.1411B}, $\overline N\in\ad D_{A/I}$ lifts to some $N\in\ad D_A$ such that $\Ad(r(g))(N)=N$ for all $g\in I_{L/K}$, and any two lifts differ by an element of $I\otimes_{A/\mathfrak{m}_A}(\ad D_{A/\mathfrak{m}_A})^{I_{L/K}}$.

Now $D_A$, together with $r$ and $N$, is an object of $G-\WD_E(L/K)$ if and only if $N=p^{-f_K}\Ad(\Phi)(N)$, where $\Phi:=r(\varphi^{f_K})$.  We define \[h:=N-p^{-f_K}\Ad(\Phi)(N)\in I\otimes_{A/\mathfrak{m}_A}\ad D_{A/\mathfrak{m}_A}^{I_{L/K}}.\]  If $H^2(\ad D_{A/\mathfrak{m}_A})=0$, then by definition there exist $f,g\in I\otimes_{A/\mathfrak{m}_A}\ad D_{A/\mathfrak{m}_A}^{I_{L/K}}$ such that $h=\ad_{\overline N}(f)+(p^{-f_K}\Ad(\overline\Phi)-1)(g)$.  We can view $f$ and $g$ either as elements of $\Aut_G(D_A)$ (congruent to the identity modulo $I$) or as elements of its tangent space.  Then we claim that if we define $\widetilde N:=N+g$ and $\widetilde\Phi:=f^{-1}\circ\Phi$, then $\widetilde{N}=p^{-f_K}\Ad(\widetilde{\Phi})(\widetilde{N})$.  Indeed,
\begin{align*}
\widetilde{N}-p^{-f_K}\Ad(\widetilde{\Phi})(\widetilde{N})&=N+g-p^{-f_K}(\Ad(1-f)\circ\Ad(\Phi))(N+g) 	\\
&=N+g-p^{-f_K} \Ad(\Phi)(N)-p^{-f_K}\Ad(\Phi)(g)\\ &\phantom{{}=1}+p^{-f_K}[f,\Ad(\Phi)(N)]+p^{-f_K}[f,\Ad(\Phi)(g)]	\\
&=\ad_{\overline N}(f)+p^{-f_K}[f,\Ad(\Phi)(N)]	\\
&=[h,f]=0
\end{align*}
Here we have used that $f,g,h\in I\otimes_{A/\mathfrak{m}_A}\ad
D_{A/\mathfrak{m}_A}^{\Gal_{L/K}}$ and $I\cdot I\subset
I\mathfrak{m}_A=0$, so the Lie brackets $[f,\Ad(\Phi)(g)]$ and $[h,f]$
vanish. This proves part~(1).

Now suppose that $\widetilde N=p^{-f_K}\Ad(\widetilde\Phi)(\widetilde N)$, and let $f,g\in I\otimes_{A/\mathfrak{m}_A}\ad D_{A/\mathfrak{m}_A}^{I_{L/K}}$.  Define $\widetilde{N}':=N+g$ and define $\widetilde{\Phi}':=f^{-1}\circ\widetilde\Phi$.  Then 
\begin{align*}
\widetilde{N}'-p^{-f_K}\Ad(\widetilde{\Phi}')(\widetilde{N}')&=\widetilde N+g-p^{-f_K} \Ad(\widetilde\Phi)(\widetilde N)-p^{-f_K}\Ad(\widetilde\Phi)(g)\\&\phantom{{}=1}+p^{-f_K}[f,\Ad(\widetilde\Phi)(\widetilde N)]+p^{-f_K}[f,\Ad(\widetilde\Phi)(g)]	\\
&=(1-p^{-f_K}\Ad(\widetilde\Phi))(g)+[f,\widetilde N]\\
&=-(p^{-f_K}\Ad(\Phi)-1)(g)-\ad_N(f).
\end{align*}
Thus, $\widetilde\Phi',\widetilde N'$ give another lift if and only if $(f,g)\in\ker(d^1)$.

Moreover, if $(\widetilde\Phi',\widetilde N')$ is another lift, it is
isomorphic to $(\widetilde\Phi,\widetilde N)$ if and only if there is
some $j\in I\otimes_{A/\mathfrak{m}_A}D_{A/\mathfrak{m}_A}^{I_{L/K}}$
such that $\widetilde N'=\Ad(1+j)(\widetilde N)$ and
$(1+j)\widetilde\Phi=\widetilde\Phi'(1+j)$.  This is equivalent to
$\widetilde N-\widetilde N'=\ad_N(j)$ and
$\widetilde\Phi(\widetilde\Phi')^{-1}=1-(1-\Ad(\Phi))(j)$.  In other
words, $(\widetilde\Phi,\widetilde N)$ and
$(\widetilde\Phi',\widetilde N')$ differ by an element of $\im(d^0)$,
as required.
\end{proof}

\subsection{Construction of smooth points}\label{subsec: very
  smooth points}

We wish to show that ``most'' points of $Y_{L/K,\varphi,\mathcal N}$ are
smooth, and so are their images in $Y_{L'/K',\varphi,\mathcal N}$ for any
finite extension $K'/K$. In this section we will consider a single
fixed extension $K'/K$, and in section~\ref{subsec: Tate duality}
below we will deduce a result for all extensions $K'/K$ simultaneously.

We begin by fixing an inertial type $\tau:I_{L/K}\rightarrow G(E)$.  This amounts to considering the fiber of $Y_{L/K,\varphi,\mathcal N}\rightarrow Y_{L/K}$ over the point corresponding to $\tau$.
Next, we observe that if we can find $r:W_K\rightarrow G(E)$ such that
$r|_{I_{K}}=\tau$, then $\Phi:=r(g_0)$ is an element of the algebraic group defined over $E$
\[N_G(\tau):=\{h\in G: hr(g)h^{-1}\in r(I_{L/K})\text{ for all }g\in I_{L/K}\}.\]
Note that $\Phi$ is not necessarily an element of  the
centraliser \[Z_G(\tau):=\{h\in G: hr(g)h^{-1}= r(g)\text{ for all
  }g\in I_{L/K}\}.\] However, since $I_{L/K}$ is finite
(and in particular has only finitely many automorphisms),
$Z_G(\tau)\subset N_G(\tau)$ has finite index; so we have $Z_G(\tau)^\circ=N_G(\tau)^\circ$ and $\Lie Z_G(\tau)=\Lie N_G(\tau)$.  In particular, this implies that $N_G(\tau)$ and $Z_G(\tau)$ are reductive:
\begin{theorem}\label{theorem:centralizer-reductive}
The normalizer $N_G(\tau):=\{h\in G: hr(g)h^{-1}\in r(I_{L/K})\text{ for all }g\in I_{L/K}\}$ of $\tau(I_{L/K})$ is a reductive group.
\end{theorem}
\begin{proof}
Since we are working over a field of characteristic $0$, it is enough to prove that the connected component of the identity $N_G(\tau)^\circ=Z_G(\tau)^\circ=Z_{G^\circ}(\tau)^\circ$ is reductive.  But reductivity for the latter group follows from~\cite[Theorem 2.1]{MR1893005}, which states that when a finite group acts on a connected reductive group, the connected component of the identity of the fixed points is reductive.
\end{proof}
\begin{remark}
Prasad and Yu prove their result under the assumption that the characteristic of the ground field does not divide the order of the group.  Conrad, Gabber, and Prasad prove a more general result~\cite[Proposition A.8.12]{predbook}, assuming only that the algebraic group acting is geometrically linearly reductive.
\end{remark}

 Our hypotheses imply that $N\in\Lie Z_G(\tau)$ and
  $\Phi\in N_G(\tau)$.  However, if $(r,N)$ exists and has the correct
  inertial type, the set of $\Phi\in G(E)$ compatible with
  $r|_{I_{L/K}}$ and $N$ is a torsor under $Z_G(\tau)\cap
  Z_G(N)$.

We now briefly recall the theory of associated cocharacters over a
field of characteristic~$0$; we refer
the reader to~\cite{jantzen} (in particular section~5) for further
details and proofs. We will not draw attention to the assumption that
our ground field has characteristic~$0$ below (but we will frequently
use it); on the other hand, we do explain why the results that we are
recalling hold over arbitrary fields of characteristic~$0$.

 If $N\in\mathfrak{g}$ is nilpotent, a
cocharacter $\lambda:\Gm\rightarrow G$ is said to be \emph{associated}
to $N$ if
\begin{itemize}
\item $\Ad(\lambda(t))(N)=t^2N$, and
\item  $\lambda$ takes values in the
  derived subgroup of a Levi subgroup $L\subset G$ for which
  $N\in\mathfrak{l}:=\Lie L$ is distinguished (that is, every torus contained in
  $Z_L(N)$ is contained in the center of $L$).
\end{itemize}
By~\cite[Thm.\ 26]{mcninch}, for
any~$N$ there exists a cocharacter associated to $N$ which is defined over the same field as $N$.  Any two cocharacters associated to $N$ are conjugate under the action of $Z_G(N)^\circ$.

An \emph{$\sl_2$-triple} is as usual a non-zero triple $(X,H,Y)$ of elements of
$\mathfrak{g}$ such that $[H,X]=2X$, $[H,Y]=-2Y$, and $[X,Y]=H$.  The
Jacobson--Morozov theorem~\cite[Ch.\ VIII \S 11 Prop.\ 2]{MR2109105}  states
that for a non-zero nilpotent element $N$ in a semisimple Lie algebra,
an $\sl_2$-triple $(N,H,Y)$ always exists, and any two such triples
$(N,H,Y)$ and $(N,H',Y')$ are conjugate under the action of
$Z_G(N)^\circ$~\cite[Ch.\ VIII \S 11 Prop.\ 1]{MR2109105}.  Given a
pair $(N,H)$ such that $[H,N]=2N$ and
$H\in [N,\mathfrak{g}]$, it is possible to construct an $\sl_2$-triple
$(N,H,Y)$~\cite[Ch.\ VIII \S 11 Lem.\ 6]{MR2109105} (or the zero
triple if $N=H=0$).  Since $\SL_2$ is simply connected, this implies that there
is a homomorphism $\SL_2\rightarrow G$ which sends the ``standard''
basis for $\sl_2$ to $(N,H,Y)$.

If we let $\lambda:\Gm\rightarrow \SL_2\rightarrow G$ be the
composition of the cocharacter $t\mapsto
\left(\begin{smallmatrix}t&0\\0&t^{-1}\end{smallmatrix}\right)$ 
with this homomorphism
$\SL_2\rightarrow G$, then $\lambda$ is associated to $N$.  Moreover,
the association $\lambda\mapsto d\lambda(1)$ sends cocharacters
associated to $N$ to elements $H$ such that $[H,N]=2N$ and
$H\in [N,\mathfrak{g}]$, and this is an injective map~\cite[Prop.\
5.5]{jantzen} (this reference assumes that the ground field is
algebraically closed, but this hypothesis is not used).  Thus (in
characteristic $0$) associated cocharacters are a group-theoretic
analogue of the Jacobson--Morozov theorem.

We use the following properties of associated cocharacters; the given
reference assumes the ground field is algebraically closed, but these
statements can all be checked after extension of the ground field.
\begin{prop}[{\cite[5.9,5.10,5.11]{jantzen}}]\label{prop: associated
    cochar facts} 
Let $G$ be a connected reductive group, let $N\in\mathfrak{g}$ be a nilpotent element, and let $\lambda:\Gm\rightarrow G$ be an associated cocharacter for $N$.  Then
\begin{enumerate}
\item	the associated parabolic $P_G(\lambda)$ depends only on $N$, not on the choice of associated cocharacter.
\item	we have $Z_G(N)\subset P_G(\lambda)$.  In particular,
  $Z_G(N)=Z_{P_G(\lambda)}(N)$.
\item	$Z_G(N)=\left(U_G(\lambda)\cap Z_G(N)\right)\rtimes\left(Z_G(\lambda)\cap Z_G(N)\right)$
\item	$Z_G(\lambda)\cap Z_G(N)$ is reductive.
\end{enumerate}
\end{prop}

In particular, by Proposition~\ref{prop: associated
    cochar facts}~(3), the disconnectedness of $Z_G(N)$ is entirely accounted for by the disconnectedness of $Z_G(\lambda)\cap Z_G(N)$.  The connectedness assumption on $G$ for that part is removed in~\cite[Proposition 4.9]{2014arXiv1403.1411B}, so we may apply it to groups such as $Z_G(\tau)$ (which is reductive but not necessarily connected).

We will use the following lemma in the proof of Theorem~\ref{thm: very
  smooth points exist for Y} below.
  \begin{lem}
    \label{lem: associated cochars weight 2 adN}If~$\lambda$ is an
    associated cocharacter of~$N$, then the weight-2 part
    of~$\mathfrak{g}$ for the adjoint action of~$\lambda$ is in the
    image of~$\ad_N$.
  \end{lem}
  \begin{proof}If~$N=0$, then $\lambda$ is the constant cocharacter and the corresponding weight-$2$ subspace is trivial. Otherwise, we may find an $\sl_2$-triple of the form $(N,d\lambda(1),Y)$ and view $\mathfrak{g}$ as a representation of $\sl_2$.  Then the result follows by the representation theory of $\mathfrak{sl}_2$: if $T\in\mathfrak{g}$ is in the weight-$2$ part, then $\frac 1 2 [Y,T]$ is in the weight-$0$ part and 
\[	[N,\frac 1 2 [Y,T]]=\frac 1 2 [[N,Y],T]=\frac 1 2[d\lambda(1),T] = T	\]
so $T$ is in the image of $\ad_N$.
  \end{proof}
  Let $f:G\to G'$ be a morphism of reductive groups over~$E$, inducing
  a morphism $\mathfrak{g}\to\mathfrak{g'}$ on Lie algebras, which we
  also denote by~$f$. We use the following lemma in the proof of
  Theorem~\ref{thm: smoothness for functorial maps} below.
  \begin{lem}
    \label{lem: associated cochars functoriality}If~$\lambda$ is an
    associated cocharacter for~$N\in\mathfrak{g}$, then
    $f\circ\lambda$ is an associated cocharacter for~$f(N)$.
  \end{lem}
  \begin{proof}
It is clear that $d\lambda(1)$ is semisimple.  Then there exists some $Y\in\mathfrak{g}$ such that $(N,d\lambda(1),Y)$ is an $\sl_2$-triple, and therefore there is a homomorphism $\SL_2\rightarrow G$ such that the precomposition with the diagonal is $\lambda$.  The composition $\Gm\rightarrow\SL_2\rightarrow G\rightarrow G'$ is $f\circ \lambda$.  Moreover, if we consider the composition $\SL_2\rightarrow G\rightarrow G'$ and differentiate, we get a map $\sl_2\rightarrow\mathfrak{g}'$ sending the ``standard'' basis of $\sl_2$ to $(f(N),f(d\lambda(1)),f(Y))$.  This shows that $[f(d\lambda(1)),f(N)]=2f(N)$ and $f(d\lambda(1))$ is in the image of $\ad_{f(N)}$.  Since $f(d\lambda(1))=d(f\circ\lambda)(1)$, this shows that $f\circ\lambda$ is associated to $f(N)$, by~\cite[Prop.\ 5.5]{jantzen}.
  \end{proof}
If $K'/K$ is a finite extension, we write $H^2_{L'/K'}$ for the
coherent sheaf on~$Y_{L/K,\varphi,\mathcal N}$ given by the cokernel of  \[
  (\ad \mathcal D)^{I_{L'/K'}}\oplus(\ad
  \mathcal D)^{I_{L'/K'}}\xrightarrow{\ad_{N_{L'}}-(p^{-f_{K'}}\Ad(\Phi^{f_{K'}/f_K})-1)}(\ad
  \mathcal D)^{I_{L'/K'}}	\]
where~$(\mathcal D,\Phi,N,\tau)$ is the universal
object over~$Y_{L/K,\varphi,\mathcal N}$, $\ad_{N_{L'}}$ acts on the first factor, and $(p^{-f_{K'}}\Ad(\Phi^{f_{K'}/f_K})-1)$ acts on the second factor. 
Then the fiber of $H^2_{L'/K'}$ at a closed point of~$Y_{L/K,\varphi,\mathcal N}$ controls the obstruction theory of the
restriction to~$W_{K'}$ of the corresponding Weil--Deligne representation.
\begin{thm}\label{thm: very smooth points exist for Y}Let $K'/K$ be a
  finite extension. Then there is a dense open subscheme $U\subset
  Y_{L/K,\varphi,\mathcal N}$ \emph{(}possibly depending on~$K'$\emph{)} such that
  $H^2_{L'/K'}|_U=0$.
  \end{thm}
  \begin{proof}Since the support of~$H^2_{L'/K'}$ is closed, it
    suffices to show that if we consider the map
    $Y_{L/K,\varphi,\mathcal N}\rightarrow Y_{L/K,\mathcal N}$, then each component of
    the fiber over some point $N\in Y_{L/K,\mathcal N}$ contains a point
    $(\Phi,N)$ whose corresponding $H^2$ vanishes (when viewed as a
    point of $Y_{L'/K',\varphi,\mathcal N}$).

  To do this, we consider a new moduli problem
  $\widetilde Y_{L/K,\varphi,\mathcal N}$ , which by definition is the functor on the category of
  $E$-algebras whose $A$-points are triples
  \[ (\Phi,N,\tau)\in N_G(\tau)\times \Lie Z_G(\tau)\times\Rep_A I_{L/K} \]
  which satisfy $N=p^{-f_K}\Ad(\Phi)(N)$.

  This is representable by an affine scheme which  we also write as
  $\widetilde Y_{L/K,\varphi,\mathcal N}$, and there is a natural morphism
  $\widetilde Y_{L/K,\varphi,\mathcal N}\rightarrow Y_{L/K,\mathcal N}$.  Indeed, the
  map $Y_{L/K,\varphi,\mathcal N}\rightarrow Y_{L/K,\mathcal N}$ factors through the
  natural inclusion
  $Y_{L/K,\varphi,\mathcal N}\hookrightarrow\widetilde Y_{L/K,\varphi,\mathcal N}$, and
  the fibers of $Y_{L/K,\varphi,\mathcal N}\rightarrow Y_{L/K,\mathcal N}$ are closed
  and open in the fibers of
  $\widetilde Y_{L/K,\varphi,\mathcal N}\rightarrow Y_{L/K,\mathcal N}$.  Thus, it
  suffices to study the fibers of the map
  $\widetilde Y_{L/K,\varphi,\mathcal N}\rightarrow Y_{L/K,\mathcal N}$. (Note that the
  tangent-obstruction complex for objects of $G-\WD_E(L/K)$ makes sense
  over $\widetilde Y_{L/K,\varphi,\mathcal N}$ as well.)

  Choose an associated cocharacter
  $\lambda:\Gm\rightarrow Z_G(\tau)^\circ$ for $N$, so that in
  particular $\Ad(\lambda(t))(N)=t^2N$, and let
  $\Phi:=\lambda(p^{f_K/2})$. Then~$(\Phi,N,\tau)$ is a point of
  $\widetilde Y_{L/K,\varphi,\mathcal N}$, and we wish to study the restriction
  $(\Phi^{f_{K'}/f_K}, N_{L'},\tau|_{I_{L'/K'}})$.

  If $D$ denotes the underlying $G$-torsor for $(\Phi, N,\tau)$, and
  $\ad D$ denotes its pushout via the adjoint representation, then
  $\Ad(\Phi)$ and $\Ad(\Phi^{f_{K'}/f_K})$ are semi-simple operators
  on $(\ad D)^{I_{L/K}}$ and $(\ad D)^{I_{L'/K'}}$, respectively.
  Therefore, $p^{-f_K}\Ad(\Phi)-1$ and
  $p^{f_{K'}}\Ad(\Phi^{f_{K'}/f_K})-1$ are semi-simple as well (since
  they are the difference of commuting semi-simple operators in characteristic $0$).

  Thus, to compute the cokernel of
  $p^{-f_{K'}}\Ad(\Phi^{f_{K'}/f_K})-1$, it suffices to compute its
  kernel.  Now $(\ad D)^{I_{L'/K'}}$ is graded by the adjoint action
  of $\lambda:\Gm\rightarrow Z_G(\tau)\subset Z_G(\tau|_{I_{L'/K'}})$,
  and if $(\ad D)^{I_{L'/K'}}_k$ denotes the weight-$k$ subspace, then
  $p^{-f_{K'}}\Ad(\Phi^{f_{K'}/f_K})-1$ preserves it, so it suffices
  to compute
  \[\ker(p^{-f_{K'}}\Ad(\Phi^{f_{K'}/f_K})-1)|_{(\ad
    D)^{I_{L'/K'}}_k}\] for each~$k$.  But $p^{-f_{K'}}\Ad(\Phi^{f_{K'}/f_K})-1$
  acts invertibly unless $k=2$ (in which case it acts by~ $0$), so the
  cokernel of~$p^{-f_{K'}}\Ad(\Phi^{f_{K'}/f_K})-1$ is exactly $(\ad
    D)^{I_{L'/K'}}_2$.  By
  Lemma~\ref{lem: associated cochars weight 2 adN}, the weight-$2$ part
  of $\mathfrak{g}^{I_{L'/K'}}$ is in the image of $\ad_N$, so we
  conclude that $H^2_{L'/K'}$ vanishes at $(\Phi,N)$, and at its image in
  $Y_{L'/K',\varphi,\mathcal N}$.

  We need to find similar points on every connected component of the fiber of
  $\widetilde Y_{L/K,\varphi,\mathcal N}\rightarrow Y_{L/K,\mathcal N}$ over
  $N\in Y_{L/K,\mathcal N}$. 
  This
  fiber is a torsor under $N_G(\tau)\cap Z_G(N)$, and the
  disconnectedness of $N_G(\tau)\cap Z_G(N)$ is entirely accounted for
  by the disconnectedness of $N_G(\tau)\cap Z_G(\lambda)\cap Z_G(N)$,
  by~\cite[Prop.\ 4.9]{2014arXiv1403.1411B} (applied
  with~$G'=N_G(\tau)$). 
  On each component of $N_G(\tau)\cap Z_G(N)$,
  we may therefore by \cite[Lem.\ 5.3]{2014arXiv1403.1411B} choose a
  finite-order element 
  $c\in N_G(\tau)\cap Z_G(\lambda)\cap Z_G(N)$. (Note that
  $N_G(\tau)\cap Z_G(\lambda)\cap Z_G(N)=Z_{N_G(\tau)}(N)\cap
  Z_{N_G(\tau)}(\lambda)$ is reductive
  by Proposition~\ref{prop: associated
    cochar facts}.) 

We now check that that $H_{L/K}^2$ and  $H_{L'/K'}^2$ vanish at the points of $\widetilde{Y}_{L/K,\varphi,N}$ and $\widetilde{Y}_{L'/K',\varphi,N}$, respectively, corresponding to $(\Phi\cdot c,N)$.

  Firstly, we claim that $p^{-f_{K'}}\Ad((\Phi\cdot c)^{f_{K'}/f_K})-1$ is
  semi-simple, or equivalently, that
  $\Ad((\Phi\cdot c)^{f_{K'}/f_K})$ is semi-simple.  For this,
  it suffices to check that some iterate of
  $\Ad((\Phi\cdot c)^{f_{K'}/f_K})$ is semi-simple (since we are in
  characteristic $0$).  Let $n$ be the order of $c$.  Since $c$ and
  $\Phi=\lambda(p^{f_{K}/2})$ commute,
  $\Ad(\Phi^{f_{K'}/f_K}\cdot c)^n=\Ad(\Phi^{nf_{K'}/f_K}\cdot
  c^n)=\Ad(\Phi^{nf_{K'}/f_K})$.  But since $\Ad(\Phi)$ is semi-simple
  by construction, so is $\Ad(\Phi^{nf_{K'}/f_K})$, as claimed.

  Thus, to compute the cokernel of
  $p^{-f_{K'}}\Ad((\Phi\cdot c)^{f_{K'}/f_K})-1$, it suffices to
  compute its kernel, which is contained in the kernel of
  $p^{-nf_{K'}}\Ad(\Phi^{nf_{K'}/f_K})-1$.  Since
  $p^{-nf_{K'}}\Ad(\Phi^{nf_{K'}/f_K})-1$ acts invertibly on each
  weight space~$(\ad D)^{I_{L/K}}_k$ unless $k=2$, the cokernel of
  $p^{-f_{K'}}\Ad(\Phi^{f_{K'}/f_K}\cdot c)-1$ is contained in
  $(\ad D)^{I_{L/K}}_2$.  Since $(\ad D)^{I_{L/K}}_2$ is again in the
  image of $\ad_N$ by Lemma~\ref{lem: associated cochars weight 2
    adN}, 
  we are done.
\end{proof}

\begin{cor}\label{cor: WD stack is generically smooth zero dimensional
  controlled by H2}
  The stack $G-\WD_E(L/K)$ is generically smooth, and is
  equidimensional of dimension~$0$; equivalently, the scheme $Y_{L/K,\varphi,\mathcal N}$  is
  generically smooth, and is equidimensional of dimension~$\dim
  G$.  The nonsmooth locus is precisely the locus of Weil--Deligne
  representations~$D$ with $H^2(\ad D)\ne 0$.   Moreover, $Y_{L/K,\varphi,\mathcal N}$ is locally a complete intersection and reduced.
\end{cor}
\begin{proof}
  It is enough to prove the statement for~$Y_{L/K,\varphi,\mathcal N}$. Let
  $U\subset Y_{L/K,\varphi,\mathcal N}$ be the dense open subscheme provided by
  Theorem~\ref{thm: very smooth points exist for Y} (with
  $K'=K$). Then at each closed point~$x$ of~$U$, it follows from
  Lemma~\ref{lem: Weil representations lift} and
  Proposition~\ref{prop: deformations of WD repns controlled by
    complex} that $Y_{L/K,\varphi,\mathcal N}$ is formally smooth
  at~$x$. Furthermore, for any closed point~$x$ of~$Y_{L/K,\varphi,\mathcal N}$
  with corresponding Weil--Deligne representation~$D_x$, the dimension
  of the tangent space at~$x$ is $\dim G-\dim H^0(D_x)+\dim
  H^1(D_x)$. Since the Euler characteristic of $C^\bullet(D_x)$ is zero,
  this is equal to~$\dim G+\dim H^2(\ad D_x)=\dim G$, and the claim about $H^2(\ad D)$ 
  follows immediately.

To see that $Y_{L/K,\varphi,\mathcal N}$ is reduced and locally a complete intersection, we proceed as in the proof of~\cite[Corollary 5.4]{2014arXiv1403.1411B}.  We have morphisms $Y_{L/K,\varphi,\mathcal{N}}\rightarrow Y_{L/K,\mathcal N}\rightarrow Y_{L/K}$, and the fiber above a point $\tau\in Y_{L/K}$ is defined by the relation $N=p^{-f_K}\Ad(\Phi)(N)$, where $\Phi\in Z_G(\tau)$ and $N\in\Lie Z_G(\tau)$.  In other words, the fiber $Y_{L/K,\varphi,\mathcal N}|_\tau$ is cut out of the smooth $(2\dim Z_G(\tau))$-dimensional space $Z_G(\tau)\times \Lie Z_G(\tau)$ by $\dim Z_G(\tau)$ equations.  

The quotient map $G\rightarrow G/Z_G(\tau)\cong Y_{L/K}$ admits sections \'etale locally.  Thus, there is an \'etale neighborhood $U\rightarrow Y_{L/K}$ of $\tau$ such that the $U$-pullback $Y_{L/K,\varphi,\mathcal{N}}\times_{Y_{L/K}}U$ is isomorphic to $U\times Y_{L/K,\varphi,\mathcal N}|_\tau$.  Since $Y_{L/K,\varphi,\mathcal{N}}\times_{Y_{L/K}}U$ is \'etale over $Y_{L/K,\varphi,\mathcal N}$, it is equidimensional of dimension $\dim G$.  On the other hand, it is cut out of the smooth $(\dim U+2\dim Z_G(\tau))$-dimensional space $U\times Z_G(\tau)\times \Lie Z_G(\tau)$ by $\dim Z_G(\tau)$ equations.  

Since $\dim U=\dim Y_{L/K}=\dim G-\dim Z_G(\tau)$ and being locally a complete intersection can be checked \'etale locally, it follows that $Y_{L/K,\varphi,\mathcal{N}}$ is locally a complete intersection.  Moreover, schemes which are local complete intersections are Cohen--Macaulay, by~\cite[Theorem 21.3]{matsumura}, and Cohen--Macaulay schemes which are generically reduced are reduced everywhere, by~\cite[Theorem 17.3]{matsumura}, so we are done.
\end{proof}

If~$G\to G'$ is a morphism of reductive groups over~$E$, then for any family of $G$-torsors $D$ over $\Spec A$, we can push out to a family $D'$ of $G'$-torsors.  Therefore, the moduli space $Y_{L/K,\varphi,\mathcal N}$ of (framed) $G$-valued Weil--Deligne representations carries a family $D'$ of $G'$-torsors, and $\ad D':=\Lie\Aut_{G'}(D')$ is a coherent sheaf on $Y_{L/K,\varphi,\mathcal N}$.  Since $D$ is a trivial $G$-torsor, $D'$ is a trivial $G'$-torsor.  Since pushing out $G$-torsors to $G'$-torsors is functorial, $D'$ is a family of $G'$-valued Weil--Deligne representations and we can construct the complex $C^\bullet(D')$.  We let $H^2_{G'}$ denote its cohomology in degree~$2$.

\begin{thm}
  \label{thm: smoothness for functorial maps}Let $f:G\to G'$ be a
  morphism of reductive groups over~$E$. Then there is a dense open
  subset $U\subset Y_{L/K,\varphi,\mathcal N}$ \emph{(}possibly depending
  on~$G'$\emph{)} such that $H^2_{G'}|_U=0$.
\end{thm}
\begin{proof}
As in the proof of Theorem~\ref{thm: very smooth points exist for Y}, it suffices to construct a point on each connected component of each fiber of the map $Y_{L/K,\varphi,\mathcal N}\rightarrow Y_{L/K,\mathcal N}$ where $H^2_{G'}$ vanishes.  In fact, the same points work: by Lemma~\ref{lem: associated cochars functoriality} the composition $f\circ\lambda$ is an associated cocharacter for $f_\ast(N)$.  Therefore, $H^2_{G'}$ vanishes at the point corresponding to $(\lambda(p^{f_K/2}),N)$.  Similarly, if $c\in N_G(\tau)\cap Z_G(\lambda)\cap Z_G(N)$ is a finite order point, then $H_{G'}^2$ vanishes at the point corresponding to $(\lambda(p^{f_K/2})\cdot c,N)$.
\end{proof}

\begin{rem}
  \label{rem: thinking of everything as coming from SL2}The proofs of
  Theorems~\ref{thm: very smooth points exist for Y} and~\ref{thm:
    smoothness for functorial maps} justify the claim we made in the
  introduction, that all of the smooth points that we explicitly
  construct arise from pushing out a single ``standard'' smooth point
  for~$\SL_2$. Indeed, as discussed above, given an associated
  cocharacter~$\lambda$ for~$N$, the map $\lambda\mapsto
  d\lambda(1)$ allows us to determine a homomorphism $\SL_2\to G$, and
  we see that the choice of~$\Phi$, $N$ made in the proof of
  Theorem~\ref{thm: very smooth points exist for Y} is the image under
  this homomorphism of the elements~$\Phi$, $N$ for~$\SL_2$ discussed
  in the introduction.
\end{rem}

\begin{rem}\label{rem: SL2 form of WD group}The
  Jacobson--Morozov theorem allows one to think of semisimple
  Weil--Deligne representations as representations
  of~$W_K\times\SL_2$; see~\cite[Prop.\ 2.2]{MR2730575} for a precise
  statement. From this perspective, our construction of smooth points
  from associated cocharacters can be summarised as follows: 
  given a nilpotent~$N\in\Lie G$, we obtain
  a map $\SL_2\to G$, and the corresponding Weil--Deligne
  representation is obtained by composing with the map
  \[W_K\times\SL_2\to\SL_2\] which on the first factor is unramified and
  takes an arithmetic Frobenius to the matrix~$
  \begin{pmatrix}
    p^{f_K}&0\\0& p^{-f_K}
  \end{pmatrix}
$, and is the identity on the second factor.
\end{rem}

\subsection{Tate local duality for Weil--Deligne
  representations}\label{subsec: Tate duality}

If $D$ is a $G$-valued Weil--Deligne representation 
over a field
$E$, we can also prove an analogue of Tate local duality for the
complex~$C^\bullet(D)$. 
In addition to allowing us to compute with either kernels or cokernels, this pairing allows us to give an explicit characterisation of the
smooth locus (see Corollary~\ref{cor: WD stack is generically smooth zero dimensional
  controlled by H0}). 
Since we only need the pairing
between~$H^0$ and~$H^2$, we have not worked out the details of the
pairing on~$H^1$s, which for reasons of space we leave to the
interested reader.

To construct pairings $H^i((\ad D)^\ast(1))\times H^{2-i}(\ad D)\rightarrow E(1)$, we use the evaluation pairing $\mathrm{ev}:(\ad D)^\ast\times \ad D\rightarrow E$.  Here the ``$(1)$'' means that we multiply the action of $\Ad(\Phi)$ by $p^{f_K}$; since $(\ad D)^\ast$ and $(\ad D)^\ast(1)$ have the same underlying vector space (as do $E$ and $E(1)$), we have an induced pairing $\mathrm{ev}(1):(\ad D)^\ast(1)\times\ad D\rightarrow E(1)$.  Note that if $X\in (\ad D)^\ast$, $Y\in \ad D$, then $\mathrm{ev}(\Ad(\Phi)(X),\Ad(\Phi)(Y))=\mathrm{ev}(X,Y)$, and if $X\in (\ad D)^\ast(1)$, $Y\in \ad D$, then $\mathrm{ev}(1)(\Ad(\Phi)(X),\Ad(\Phi)(Y))=\mathrm{ev}(p^{f_K}\Ad(\Phi)(X),\Ad(\Phi)(Y)) = p^{f_K}\mathrm{ev}(X,Y)=\Ad(\Phi)(\mathrm{ev}(1)(X,Y))$.

\begin{prop}\label{prop: Tate local duality}
Let $D$ be as above.  Then the evaluation pairing induces a perfect pairing $H^0((\ad D)^\ast(1))\times H^2(\ad D)\rightarrow E(1)$.
\end{prop}
\begin{proof}
We first check that the pairing $\mathrm{ev}(1):(\ad D)^\ast(1)\times \ad D\rightarrow E(1)$ descends to a well-defined pairing $H^0((\ad D)^\ast(1))\times H^2(\ad D)\rightarrow E(1)$.  If $X\in (\ad D)^\ast(1)^{I_{L/K}}$ is in the kernel of $\ad_N$ and the kernel of $1-\Ad(\Phi)$, and $Y\in (\ad D)^{I_{L/K}}$, then
\begin{align*}
\mathrm{ev}(1)(X,Y+\ad_N(Z))&=\mathrm{ev}(1)(X,Y)+\mathrm{ev}(1)(X,\ad_N(Z))	\\
&=\mathrm{ev}(1)(X,Y)-\mathrm{ev}(1)(\ad_N(X),Z)\\&=\mathrm{ev}(1)(X,Y),\end{align*}and
\begin{align*}\mathrm{ev}(1)(X,Y+(p^{-f_K}\Ad(\Phi)-1)(Z))&=\mathrm{ev}(1)(X,Y)+\mathrm{ev}(1)(X,p^{-f_K}\Ad(\Phi)(Z))-\mathrm{ev}(1)(X,Z)	\\
&=\mathrm{ev}(1)(X,Y)+p^{-f_K}\mathrm{ev}(1)(\Ad(\Phi)(X),\Ad(\Phi)(Z))-\mathrm{ev}(1)(X,Z)	\\
&=\mathrm{ev}(1)(X,Y)+\mathrm{ev}(1)(X,Z)-\mathrm{ev}(1)(X,Z)\\&=\mathrm{ev}(1)(X,Y),
\end{align*}so the pairing is indeed well-defined.

Next, we need to check that this pairing is perfect.  Suppose $X\in H^0((\ad D)^\ast(1))$ and $\mathrm{ev}(1)(X,Y)=0$ for all $Y\in H^2(\ad D)$.  Then $\mathrm{ev}(1)(X,Y)=0$ for all $Y\in (\ad D)^{I_{L/K}}$, so $X=0$.  This implies that the natural map $H^0((\ad D)^\ast(1))\rightarrow (H^2(\ad D)^\ast)(1)$ is injective.  

On the other hand, let $f:H^2(\ad D)\rightarrow E(1)$ be an element of $(H^2(\ad D)^\ast)(1)$.  By composition, we have a linear functional
\[	f:(\ad D)^{I_{L/K}}\rightarrow H^2(\ad D)\rightarrow E(1)	\]
This is an element of $\left((\ad D)^{I_{L/K}}\right)^\ast(1)$; we need to show that $\ad_N(f)=(1-\Ad(\Phi))(f)=0$.  But for any $Y\in(\ad D)^{I_{L/K}}$,
\[	\mathrm{ev}(1)(\ad_N(f),Y)=\mathrm{ev}(f,-\ad_N(Y))=0	\]
since $f$ factors through $H^2(\ad D)$.  Similarly, for any $Y\in (\ad D)^{I_{L/K}}$,
\begin{align*}
\mathrm{ev}(1)((1-\Ad(\Phi))(f),Y)&=\mathrm{ev}(1)(f,Y)-\mathrm{ev}(1)(\Ad(\Phi)(f),Y)	\\
&=\mathrm{ev}(1)(f,Y)-\mathrm{ev}(1)(f,p^{-f_K}\Ad(\Phi)^{-1}(Y))	\\
&=\mathrm{ev}(1)(f,(1-p^{-f_K}\Ad(\Phi)^{-1})(Y))	\\
&=\mathrm{ev}(1)(f,(p^{f_K}\Ad(\Phi)-1)(p^{-f_K}\Ad(\Phi^{-1})(Y)))=0
\end{align*}
Since $\Ad(\Phi):(\ad D)^{I_{L/K}}\rightarrow (\ad D)^{I_{L/K}}$ is an isomorphism, this suffices.
\end{proof}
\begin{cor}\label{cor: WD stack is generically smooth zero dimensional
  controlled by H0}
   The nonsmooth locus of the stack stack $G-\WD_E(L/K)$ is precisely the locus of Weil--Deligne
  representations~$D$ with $H^0((\ad D)^*(1))\ne 0$.
\end{cor}
\begin{proof}This is immediate from Corollary~\ref{cor: WD stack is generically smooth zero dimensional
  controlled by H2} and Proposition~\ref{prop: Tate local duality}.  
\end{proof}
We now use Corollary~\ref{cor: WD stack is generically smooth zero dimensional
  controlled by H0} to deduce that there is a dense set of points
of~$ Y_{L/K,\varphi,\mathcal N}$ which give smooth points for every finite
extension~$K'/K$.
\begin{defn}
A point $x\in Y_{L/K,\varphi,\mathcal N}$ is \emph{very smooth} if its image in $Y_{L'/K',\varphi,\mathcal N}$ is smooth for every finite extension $K'/K$.
\end{defn}

\begin{lemma}\label{lem: existence of uniform extension}
Fix a finite extension $E'/E$.  There is a finite extension $K'/K$ \emph{(}which depends only on $E'$\emph{)} such that $H_{L'/K'}^2$ vanishes at $x\in Y_{L/K,\varphi,\mathcal{N}}(E')$ if and only if $x$ is very smooth.
\end{lemma}
\begin{proof}
Suppose $(D,\Phi,N,\tau)$ corresponds to a point of $Y_{L/K,\varphi,\mathcal N}$
such that $H_{L''/K''}^2$ does not vanish at its image in
$Y_{L''/K'',\varphi,\mathcal N}$.  By Corollary~\ref{cor: WD stack is generically smooth zero dimensional
  controlled by H0}, this holds if and only if $H^0((\ad D)^\ast(1)$ does not vanish.

Thus, it suffices to consider the injectivity of
$1-p^{f_{K''}}\Ad(\Phi^{f_{K''}/f_K})^\ast:(\ad D)^{I_{L''/K''}}\rightarrow
(\ad D)^{I_{L''/K''}}$ on $\ker(\ad_N)$, where $\Ad(\Phi^{f_{K''}/f_K})^\ast$ denotes the dual of $\Ad(\Phi^{f_{K''}/f_K})$.  If this map is not injective, this implies that
$p^{f_K}\Ad(\Phi)^\ast$ has a generalized eigenvalue $\lambda$ satisfying
$\lambda^{f_{K''}/f_K}=1$. 
But the
characteristic polynomial of $\Ad(\Phi)$ acting on $\ad D$ has degree 
$\dim \ad D=\dim G$ and there are only finitely many roots of unity with minimal polynomial of bounded degree over $E'$.  It follows that there are only a
finite number of possibilities for~$\lambda$.  

In other words, to check whether $1-p^{f_{K''}}\Ad(\Phi^{f_{K''}/f_K})^\ast$ has a non-trivial kernel for any finite extension $K''/K$, it suffices to consider some fixed $K'$ such
that $f_{K'}/f_K$ is divisible by all $n$ such that $\phi(n)\leq \dim G$ and such that $\tau|_{I_{L'/K'}}$
is trivial (where $\phi(n)$ denotes Euler's totient function), as required.
\end{proof}

\begin{cor}\label{cor: WD stack is generically very smooth zero dimensional
  controlled by H2}
  The set of closed points of $G-\WD_E(L/K)$ which are very smooth is
  Zariski dense. 
\end{cor}
\begin{proof}
Let $E'/E$ be a finite extension such that $Y_{L/K,\varphi,\mathcal N}(E')$ is Zariski-dense in $Y_{L/K,\varphi,\mathcal N}$.  By Lemma~\ref{lem: existence of uniform extension}, there is a finite extension $K'/K$ such that $x\in Y_{L/K,\varphi,\mathcal N}(E')$ is very smooth if $H_{L'/K'}^2$ vanishes at $x$.  By Theorem~\ref{thm: very smooth points exist for Y}, there is a Zariski-dense open subscheme $U\subset Y_{L/K,\varphi,\mathcal N}$ such that $H_{L'/K'}^2|_{U}=0$.  But then the intersection $U\cap Y_{L/K,\varphi,\mathcal N}(E')$ is a Zariski dense subset of $Y_{L/K,\varphi,\mathcal N}$ consisting of very smooth points, so we are done.
%
%
\end{proof}

\subsection{$l$-adic Hodge theory}\label{subsec: l not p Galois to
  WD}We suppose in this subsection that $l\ne p$. We briefly recall some
results from \cite{fonl}, which will allow us to relate $l$-adic
representations of~$\Gal_K$ to Weil--Deligne representations. 

Recall that by a theorem of Grothendieck, a continuous representation
$\rho:\Gal_K\to\GL_d(E)$ is automatically potentially semi-stable, in the
sense that there is a finite extension $L/K$ such that $\rho|_{I_L}$
is unipotent. After making a choice of a compatible system of $l$-power
roots of unity in $\overline{K}$, we see from \cite[Prop.\ 1.3.3,
2.3.4]{fonl} that there is an equivalence of Tannakian categories
between the category of $E$-linear representations of $\Gal_K$ which
become semi-stable over $L$, and the full subcategory of Weil--Deligne
representations $(r,N)$ of $W_K$ over~$E$ with the properties
that~$r|_{I_L}$ 
is trivial and  the roots of the characteristic polynomial
of any arithmetic Frobenius element of~$W_L$ are $l$-adic units (such an equivalence is given by the
functor $\widehat{\underline{WD}}_{pst}$ of \cite[\S
2.3.7]{fonl}).   

\subsection{The case $l=p$: 
  $(\varphi,N)$-modules}\label{subsec: phi N modules versus WD
  representations}In this section we let $l=p$, and we explain the
relationship between Weil--Deligne representations and
$(\varphi,N)$-modules. Let $K_0$, $L_0$ be the
maximal unramified subfields of~$K$, $L$ respectively, of respective
degrees $f_K$, $f_L$ over~$\Qp$. Let~$E/\Qp$ be a finite extension,
which is large enough that it contains the image of all embeddings
$L_0\into E$, 
so that we may identify $E\otimes_{\Q_p}L_0$ with
$\oplus_{L_0\hookrightarrow E}E$.  Let $\varphi$ denote the arithmetic
Frobenius.

If $D$ is a $\Res_{E\otimes_{\Qp}L_0/E}G$-torsor over $\Spec A$, we may also view $D$ as a $G$-torsor over $A\otimes_{\Qp}L_0$.  Then any automorphism $g:L_0\rightarrow L_0$ extends to an automorphism of $A\otimes_{\Qp}L_0$, and we may pull $D$ back to a $G$-torsor $g^\ast D$ over $A\otimes_{\Qp}L_0$.  Then we may view $g^\ast D$ as a $\Res_{E\otimes_{\Qp}L_0/E}G$-torsor over $\Spec A$, which we also denote $g^\ast D$.  In particular, we may pull $D$ back by Frobenius and obtain another $\Res_{E\otimes_{\Qp}L_0/E}G$-torsor $\varphi^\ast D$ over $\Spec A$.



This motivates us to define the following groupoid on $E$-algebras.
\begin{defn}
The category of $G$-valued $(\varphi,N,\Gal_{L/K})$-modules, which we denote $G-\mathrm{Mod}_{L/K,\varphi,N}$, is the groupoid whose
  fiber
  over an $E$-algebra $A$ consists of a $\Res_{E\otimes L_0/E}G$-torsor $D$ over $A$, equipped with:
  \begin{itemize}
  \item an isomorphism
    $\Phi:\varphi^\ast D\xrightarrow{\sim}D$, 
  \item  a nilpotent element
    $N\in \Lie\Aut_GD$, and
  \item  for each $g\in\Gal_{L/K}$, an isomorphism $\tau(g):g^*D\isoto
    D$.
  \end{itemize}These are required to satisfy the following compatibilities:
\begin{enumerate}
\item	$\underline\Ad\Phi (N)=\frac{1}{p}N$.
\item	$\underline\Ad\tau(g)(N)=N$ for all $g\in\Gal_{L/K}$.
\item	$\tau(g_1g_2)=\tau(g_1)\circ g_1^\ast\tau(g_2)$ for all
  $g_1,g_2\in\Gal_{L/K}$.
\item	$\tau(g) \circ g^\ast\Phi= \Phi\circ \varphi^\ast\tau(g)$ for
  all $g\in\Gal_{L/K}$. 
\end{enumerate}
\end{defn}
Here $\underline\Ad\Phi$ and $\underline\Ad\tau(g)$ are ``twisted adjoint'' actions on $\Lie\Aut_GD$; after pushing out $Y$ by a representation $\sigma\in \Rep_E(G)$, they are given by $M\mapsto \Phi_\sigma\circ M\circ \Phi_\sigma^{-1}$ and $M\mapsto \tau(g)_\sigma\circ M\circ \tau(g)_\sigma^{-1}$, respectively.  Note that the action of $\Gal_{L/K}$ on scalars factors through the abelian quotient $\langle \varphi^{f_K}\rangle$, which also commutes with $\varphi$, so $(g_1g_2)^\ast=g_1^\ast\circ g_2^\ast$ and $g^*\varphi^*=\varphi^*g^*$.

Requiring $D$ to be a trivial $\Res_{E\otimes L_0/E}$-torsor equipped
with a trivializing section lets us define a representable functor
which covers $G-\mathrm{Mod}_{L/K,\varphi,N,\tau}$, as follows.
\begin{defn}
  Let $X_{L/K,\varphi,\mathcal N}$ denote the functor on the category of
  $E$-algebras whose $A$-points are triples
  \[ (\Phi,N,\tau)\in (\Res_{E\otimes L_0/E}G)(A)\times
    (\Res_{E\otimes L_0/E}\fg_E)(A)\times\Rep_{A\otimes
      L_0}\Gal_{L/K} \] which satisfy
  \begin{itemize}
  \item $N=p\underline\Ad(\Phi)(N)$,
  \item $\tau(g)\circ\Phi=\Phi\circ\tau(g)$, and
  \item $\underline{\Ad}(\tau(g))(N)=N$ for all $g\in\Gal_{L/K}$.
  \end{itemize}
\end{defn}

This functor is visibly representable by a finite-type
affine scheme over $E$, which we also denote by
$X_{L/K,\varphi,\mathcal N}$.  Moreover,
there is a left action of $\Res_{E\otimes L_0/E}G$ on
$X_{L/K,\varphi,\mathcal N}$ coming from changing the choice of trivializing
section.  Explicitly, \[ a\cdot (\Phi,N,\{\tau(g)\}_{g\in\Gal_{L/K}})
= (a\Phi\varphi(a)^{-1},\Ad(a)(N),\{a\tau(g)(g\cdot a)^{-1}\}_{g\in\Gal_{L/K}}).
\]

As in Lemma~\ref{lem: stack quotient for G modules}, we have the following:
\begin{lem}
  The stack quotient $[X_{L/K,\varphi,\mathcal N}/\Res_{E\otimes L_0/E}G]$ is
  isomorphic to $G-\mathrm{Mod}_{L/K,\varphi,N}$.
\end{lem}
\begin{proof}
  The proof follows as in Lemma~\ref{lem: stack quotient for G modules}.
\end{proof}

Given a $(\varphi,N,\Gal_{L/K})$-module, there is a standard recipe
due to Fontaine
for constructing a Weil--Deligne representation, and there is an
analogous construction for $\Res_{E\otimes L_0/E}G$-torsors.  Indeed,
let $A$ be an $E$-algebra.  Given a $\Res_{E\otimes L_0/E}G$-torsor
$D$ over $A$, and an embedding $\sigma:L_0\hookrightarrow E$, the $\sigma$-isotypic part is a $G$-torsor over $A$
which we denote $D_\sigma$.  Moreover, if $N_\sigma$ denotes the
$\sigma$-isotypic component of $N$, then $N_\sigma\in\Lie\Aut_G(D_\sigma)$ is nilpotent.
  
Given an isomorphism
$\Phi:\varphi^\ast D\xrightarrow{\sim}D$,
 $\Phi^{f_L}:=\Phi\circ\varphi^\ast(\Phi)\circ\cdots(\varphi^{f_L-1})^\ast(\Phi)$
restricts to an isomorphism $D_\sigma\rightarrow D_\sigma$ for each~$\sigma$.
\begin{lemma}\label{lemma:d-phi-di-phif}
For any $\sigma$ and any $E$-algebra $A$, the association $(D,\Phi)\rightsquigarrow(D_\sigma,\Phi^{f_L})$ defines an equivalence of categories between $\Res_{E\otimes L_0/E}G$-torsors $D$ over $A$ equipped with an isomorphism $\Phi:\varphi^\ast D\xrightarrow{\sim}D$, and $G$-torsors $D_\sigma$ over $A$ equipped with an isomorphism $\Phi_\sigma':D_\sigma\xrightarrow{\sim}D_\sigma$.
\end{lemma}
\begin{proof}
 Write the embeddings
$\sigma_i:L_0\hookrightarrow E$, $i\in\Z/f_L\Z$, with the numbering
chosen so that $\sigma_1=\sigma$, and ~$\Phi$ induces isomorphisms
$\sigma_i:D_{i+1}\isoto D_i$ for each~$i$ (where we write~$D_i$ for~$D_{\sigma_i}$).
  
Let $A\rightarrow A'$ be an fpqc cover trivializing $D$, so that $D_{A'}$ is a trivial torsor and we may choose a section.  Then we can write $\Phi=(\Phi_1,\ldots,\Phi_{f_L})$.  
  
We define 
\[	\underline a:=(1,(\Phi_2\cdots\Phi_{f_L})^{-1},(\Phi_3\cdots\Phi_{f_L})^{-1},\ldots,\Phi_{f_L}^{-1}).	\]
 Then if we multiply our choice of trivializing section by $\underline a$, we replace $\Phi$ by
\[	\underline a\Phi\varphi(\underline a)^{-1}=(\Phi_1\cdots\Phi_{f_L},1,\ldots,1)	\]
Thus, we can recover $(D_{A'},\Phi)$ from $((D_\sigma)_{A'},\Phi^{f_L})$.
  
Furthermore, $D_{A'}$ is equipped with a descent datum, since it is the base change of $D$.  Therefore, $(D_i)_{A'}$ has a descent datum, and since $(D_i)_{A'}\rightarrow\Spec A'$ is affine, it is effective.
  
Now suppose that $f=(f_1,\ldots,f_{f_L}):D\xrightarrow{\sim}D'$ is an
isomorphism of $\Res_{E\otimes L_0/E}G$-torsors equipped with
isomorphisms $\Phi:\varphi^\ast D\xrightarrow{\sim}D$,
$\Phi':\varphi^\ast D'\xrightarrow{\sim}D'$.  We obtain a corresponding isomorphism $f_{A'}:D_{A'}\xrightarrow{\sim}D_{A'}'$, together with a covering datum.  Then each $f_i:D_i\xrightarrow{\sim} D_i'$ is an isomorphism of $G$-torsors, and we have
\[	f_i\circ \Phi_{i}=\Phi_{i}'\circ f_{i+1}:D_{i+1}\to D'_i.	\]
Multiplying the trivializing section of $D_{A'}$ by $\underline a$ and multiplying the trivializing section of $D_{A'}$ by $\underline a'$ has the effect of replacing $\underline f$ with $\underline a'\circ \underline f\circ \underline a^{-1}$.  Then if we let $\underline a$ and $\underline a'$ be as above, $\underline f$ becomes~$(f_1,\ldots,f_1)$.
Thus, we can also recover morphisms of pairs $(D,\Phi)\rightarrow
(D',\Phi')$ from the associated morphisms of pairs
$(D_i,\Phi^{f_L})\rightarrow (D_i',(\Phi')^{f_L})$, as required.
   \end{proof}

Now suppose that $D$ is a $\Res_{E\otimes L_0/E}G$-torsor equipped
with an isomorphism $\Phi:\varphi^\ast D\xrightarrow{\sim}D$, and
suppose in addition that $D$ is equipped with a semi-linear action
$\tau$ of $\Gal_{L/K}$, compatible with $\Phi$ in the sense that
$\Phi\circ\varphi^\ast \tau(g)=\tau(g)\circ g^\ast(\Phi)$ for all
$g\in\Gal_{L/K}$. For each~$\sigma$, we will construct a Weil--Deligne representation on $D_\sigma$ which is trivial on $I_L$.  

There is a surjective map $W_K\twoheadrightarrow \Gal_{L/K}$ which restricts to a surjection $I_K\twoheadrightarrow I_{L/K}$.  If $g\in W_K$, we write $\overline{g}$ for its image in $\Gal_{L/K}$.
For $g\in W_K$, we have an isomorphism
\[	\tau(\overline{g}):g^\ast D\xrightarrow{\sim} D	\]
and we have an isomorphism
\[	\Phi^{-v(g)f_K}:=D\xrightarrow{\Phi^{-1}}\varphi^\ast D\xrightarrow{\varphi^\ast\Phi^{-1}}\cdots\xrightarrow{(g\varphi^{-1})^\ast\Phi^{-1}}g^\ast D.	\]
Accordingly, we define $r(g):D_\sigma\xrightarrow{\sim} D_\sigma$ to be the restriction of
\[
  r(g):=\tau(\overline{g})\circ\Phi^{-v(g)f_K}:D\xrightarrow{\sim} D.	\]
Note that $r|_{I_L}$ is trivial.

\begin{lemma}\label{lemma:wl-centralizes-wk}
Let $D$ be a $G$-torsor and let $r:W_K\rightarrow \Aut_G(D)$ be a homomorphism such that $r|_{I_L}$ is trivial.  Then $r(W_L)$ centralizes $r(W_K)$.
\end{lemma}
\begin{proof}
Let $g\in W_K$ and let $h\in W_L$.  Then $v(ghg^{-1}h^{-1})=0$, so $ghg^{-1}h^{-1}\in I_K$.  Moreover, 
$W_L\subset W_K$ is a normal subgroup, so that
$ghg^{-1}h^{-1}\in W_L$. 
But $I_K\cap W_L=I_L$, so $r (ghg^{-1}h^{-1})=1$, as required.
\end{proof}
We now prove the equivalence between Weil--Deligne representations and
$(\varphi,N)$-modules. In the case that~$G=\GL_n$ the following lemma is~\cite[Prop.\ 4.1]{MR2359853}.
\begin{lemma}\label{lem: WD stack agrees with phi N stack}
The map $r:W_K\rightarrow\Aut_G(D_\sigma)$ is a homomorphism, and
$(D,\Phi,N,\tau)\rightsquigarrow (D_\sigma,r,N_\sigma)$
is an equivalence of categories between $G-\mathrm{Mod}_{L/K,\varphi,N}$ and $G-\WD_E(L/K)$.
\end{lemma}
\begin{proof}
Since $\tau(\overline g)\circ g^\ast(\Phi)=\Phi\circ
\varphi^\ast(\tau(\overline g))$, we have $\Phi^{-1}\circ \tau(\overline g)=\varphi^\ast(\tau(\overline g))\circ g^\ast(\Phi^{-1})$ as isomorphisms $g^\ast D\xrightarrow{\sim}\varphi^\ast D$.  It follows that
\begin{align*}
r(g_1)r(g_2)&=\left(\tau(\overline {g_1})\circ \Phi^{-v(g_1)f_K}\right)\circ\left(\tau(\overline{g_2})\circ \Phi^{-v(g_2)f_K}\right)	\\
&=\tau(\overline {g_1})\circ (\varphi^{v(g_1)f_K})^\ast\left(\tau(\overline{g_2})\circ \Phi^{-v(g_1g_2)f_K}\right)	\\
&=\tau(\overline{g_1g_2})\circ \Phi^{-v(g_1g_2)f_K}=r(g_1g_2)
\end{align*}
and $r$ is a homomorphism. Another short computation shows that
\[N_\sigma=p^{-v(g)f_K}\Ad(r(g))(N_\sigma),\] so that $(E_\sigma,r,N_\sigma)$ is a $G$-valued Weil--Deligne representation.

The association $(D,\Phi,N,\tau)\rightsquigarrow (D_\sigma,r,N_\sigma)$ is clearly functorial.  Moreover, if $f:D\rightarrow D'$ is a morphism of $G$-valued $(\varphi,N,\Gal_{L/K})$-modules, then $\Phi'\circ \varphi^\ast(f)=f\circ\Phi$.  This implies that $f$ is determined by its restriction $f|_{D_\sigma}$ to the $\sigma$-isotypic piece, and therefore, the functor is fully faithful.

We need to check that this functor is essentially surjective.  In
other words, we need to check that we can construct $(D,\Phi,N,\tau)$
from $(D_\sigma,r,N_\sigma)$. To do so, we number the embeddings as
$\sigma_i$, as in the proof of Lemma~\ref{lemma:d-phi-di-phif}. For
each element $h\in I_{L/K}$, we fix a lift to an element~$\tilde{h}\in
I_K$; note that since $r|_{I_L}$ is trivial, $r(\tilde{h})$ is
independent of the choice of~$\tilde{h}$. 

To construct $\Phi^{f_L}|_{D_i}$ from $r$, we observe that if $g_0\in W_K$ lifts $\varphi^{f_K}$ and $(D_i,r,N_i)$ is in the essential image of our functor, then
\[	r(g_0^{f_L/f_K})=\tau(\overline{g_0}^{f_L/f_K})\Phi^{-f_L}.	\]
But $\overline{g_0}^{f_L/f_K}\in I_{L/K}$, so 
we can define $\Phi^{f_L}|_{D_i}:=r(g_0^{f_L/f_K})^{-1}r(\widetilde{\overline{g_0}^{f_L/f_K}})$.

We need to check that $\Phi^{f_L}|_{D_i}$ does not depend on our
choice of $g_0$.  Indeed, if $h\in I_K$, then
$(g_0h)^{f_L/f_K}=h_1\cdots h_{f_L/f_K-1}g_0^{f_L/f_K}$, where
$h_i:=g_0^ihg_0^{-i}\in I_K$, so we may write
$(g_0h)^{f_L/f_K}=h'g_0^{f_L/f_K}$ for some $h'\in I_K$.  Then
$r(\widetilde{\overline{h'}})=r(h')$, so
\begin{align*}
r((g_0h)^{f_L/f_K})^{-1}r(\widetilde{\overline{g_0h}^{f_L/f_K}})&= r(g_0^{f_L/f_K})^{-1}r(h')^{-1}r(\widetilde{\overline{h'}})r(\widetilde{\overline{g_0}^{f_L/f_K}})	\\
&=r(g_0^{f_L/f_K})^{-1}r(\widetilde{\overline{g_0}^{f_L/f_K}}),
\end{align*}as required.

Lemma~\ref{lemma:d-phi-di-phif} now implies that we can construct $(D,\Phi)$
from $(D_i,\Phi^{f_L}|_{D_i})$.  Since $W_K\rightarrow\Gal_{L/K}$ is
surjective, we define for $g\in\Gal_{L/K}$
\[	\tau(g):=r(\widetilde g)\circ\Phi^{v(\widetilde g)f_K}=r(\widetilde g)\circ\left(\Phi\circ\cdots\circ (\varphi^{-1})^\ast g^\ast\Phi\right)	\] 
as a map $D_{i+ v(g)f_K}\rightarrow D_i$.  We need to check that this is well-defined.  Note that the kernel of $W_K\rightarrow \Gal_{L/K}$ is $W_L$, and if $h\in W_L$, then $v(h)=(f_L/f_K)\cdot i$ for some $i\in\Z$.  Thus, for any $h\in W_L$, 
\[	r(\widetilde g h)\circ\Phi^{v(\widetilde g h)f_K}=r(\widetilde g)r(h)\circ\Phi^{i\cdot f_L}\circ\Phi^{v(\widetilde g)f_K},	\]
so it suffices to show that $r(h)\circ\Phi^{i\cdot f_L}=1$.  Since $r|_{I_L}$ is trivial, it suffices to consider the case $i=1$, i.e., $h$ generates the unramified quotient of $W_L$.  But then $r(h)\circ\Phi^{f_L}=r(h)r(g_0^{f_L/f_K})^{-1}r(\widetilde{\overline{g_0}^{f_L/f_K}})$; on the one hand $hg_0^{-f_L/f_K}\in I_K$ and $\widetilde{\overline{g_0}^{f_L/f_K}}\in I_K$, and on the other hand $g_0^{-f_L/f_K}\widetilde{\overline{g_0}^{f_L/f_K}}\in W_L$.  It follows that 
\[	hg_0^{-f_L/f_K}\widetilde{\overline{g_0}^{f_L/f_K}}\in I_K\cap W_L=I_L	\]
and the result follows.

We can also construct $\tau(g):D_{j+ v(\widetilde
  g)f_K}\xrightarrow{\sim} D_j$ for the remaining $\sigma_j$-isotypic
factors.  Indeed, the desired compatibility between $\Phi$ and $\tau$
forces us to set $\varphi^\ast\tau(g):=\Phi^{-1}\circ \tau(g)\circ
{g}^\ast\Phi:D_{i+ v(\widetilde g)f_K+ 1}\xrightarrow\sim D_{i+ 1}$
(and we proceed inductively). 

We need to check that this is well-defined.  More precisely, we need to check that $(\varphi^{f_L})^\ast\tau(g)=\tau(g)$ for all $g\in\Gal_{L/K}$.  In other words, we need to check that
\[	\tau(g)\circ\left({g}^\ast\Phi\circ \varphi^\ast{g}^\ast\Phi\circ\cdots(\varphi^{f_L-1})^\ast{g}^\ast\Phi\right)=\left(\Phi\circ \varphi^\ast\Phi\circ\cdots(\varphi^{f_L-1})^\ast\Phi\right)\circ\tau(g)	\]
as isomorphisms $D_{i+v(\widetilde g)f_K}\xrightarrow\sim D_i$, or equivalently that
\[	\tau(g)\circ g^\ast\Phi^{f_L}=\Phi^{f_L}\circ\tau(g).	\]
But
\begin{align*}
\tau(g)\circ g^\ast\Phi^{f_L}&=\left(r(\widetilde g)\circ\Phi^{v(\widetilde g)f_K}\right)\circ g^\ast(\Phi^{f_L})	\\
&= r(\widetilde g)\circ\Phi^{f_L}\circ \Phi^{v(\widetilde g)f_K}	\\
&=r(\widetilde g)\cdot r(g_0^{-f_L/f_K}\widetilde{\overline{g_0}^{f_L/f_K}})\circ \Phi^{v(\widetilde g)f_K}	\\
&=r(g_0^{-f_L/f_K}\widetilde{\overline{g_0}^{f_L/f_K}})\cdot r(\widetilde g)\circ \Phi^{v(\widetilde g)f_K}	\\
&=\Phi^{f_L}\circ \tau(g).
\end{align*}
Here we used Lemma~\ref{lemma:wl-centralizes-wk} and the fact that $g_0^{-f_L/f_K}\widetilde{\overline{g_0}^{f_L/f_K}}\in W_L$.

It remains to show that $\tau$ is a semi-linear representation, or
more precisely, that $\tau(g_1g_2)=\tau(g_1)\circ g_1^\ast\tau(g_2)$
for all $g_1,g_2\in\Gal_{L/K}$.  Now since by definition we have $\varphi^\ast\tau(g):=\Phi^{-1}\circ \tau(g)\circ
{g}^\ast\Phi:D_{i+ v(\widetilde g)f_K+ 1}\xrightarrow\sim D_{i+ 1}$,
we see that  
\begin{align*}
\tau(g_1)\circ g_1^\ast\tau(g_2)&=\tau(g_1)\circ \left(((g_1\varphi^{-1})^\ast\Phi^{-1}\circ\cdots\circ\Phi^{-1})\circ\tau(g_2)\circ (g_2^\ast\Phi\circ\cdots\circ(g_1\varphi^{-1})^\ast g_2^\ast\Phi)\right)	\\
&=\tau(g_1)\circ \left((g_1\varphi^{-1})^\ast\Phi^{-1}\circ\cdots\circ\Phi^{-1}\right)\circ\tau(g_2)\circ g_2^\ast\left(\Phi\circ\cdots\circ(g_1\varphi^{-1})^\ast\Phi\right)	\\
&=r(\widetilde{g_1})\circ r(\widetilde{g_2})\circ \Phi^{v(\widetilde{g_2})f_K}\circ g_2^\ast\Phi^{v(\widetilde{g_1})f_K}	\\
&=r(\widetilde{g_1})r(\widetilde{g_2})\circ \Phi^{v(\widetilde{g_1}\widetilde{g_2})f_K} \\&= \tau(g_1g_2),
\end{align*}as required.

Finally, we construct $N$.  We have $N_i$, and we use the desired relation $N=p\underline\Ad(\Phi)(N)$ to construct the Frobenius-conjugates of $N_i$.  It then follows that for any $g\in\Gal_{L/K}$
\begin{align*}
\underline\Ad(\tau(g))(N)&=\underline\Ad(r(\widetilde g)\circ\Phi^{v(g)f_K})(N)	\\
&=\Ad(r(\widetilde g)\circ\Phi^{v(g)f_K})(p^{-v(g)f_K}\Ad(\Phi^{-v(g)f_K})(N))	\\
&=\Ad(r(\widetilde g))(N)=N
\end{align*}
so we are done.

The assignment $(D_i,r,N_i)\rightsquigarrow (D,\Phi,N,\tau)$ is clearly functorial and quasi-inverse to $(D,\Phi,N,\tau)\rightsquigarrow(D_i,r,N_i)$.
\end{proof}

\subsection{Exact $\otimes$-filtrations for disconnected
  groups}
In this section we prove some results on tensor filtrations that we
will apply to the Hodge filtration in $p$-adic Hodge theory.



Let $G$ be an affine group scheme over a field~$k$ of characteristic
zero, let $A$ be a $k$-algebra, and let~$\eta$ be a fiber functor from
$\Rep_k(G)$ to~$\Proj_A$. 
 More precisely, $\Rep_k(G)$ is the category of $k$-linear finite-dimensional
representations of $G$, $\Proj_A$ is the category of finite projective
$A$-modules (which we will also think of as being vector bundles on
$\Spec A$), and by a ``fiber functor'' we mean that
\begin{enumerate}
\item	$\eta$ is $k$-linear, exact, and faithful.
\item\label{tensor-def}	$\eta$ is a tensor functor, that is, $\eta(V_1\otimes_kV_2)=\eta(V_1)\otimes_A\eta(V_2)$.
\item	If $\mathbf{1}$ denotes the trivial representation of $G$, then $\eta(\mathbf{1})$ is the trivial $A$-module of rank $1$.
\end{enumerate}

Given a fiber functor $\eta:\Rep_k(G)\rightarrow \Proj_A$ and an $A$-algebra $A'$, there is a natural fiber functor $\eta':\Rep_k(G)\rightarrow \Proj_{A'}$ given by composing $\eta$ with the natural base extension functor $\iota_{A'}:\Proj_A\rightarrow\Proj_{A'}$ sending $M$ to $M\otimes_AA'$.

\begin{defn}
Let $\omega,\eta:\Rep_k(G)\rightrightarrows \Proj_A$ be fiber functors.  Then $\underline\Hom^\otimes(\omega,\eta)$ is the functor on $A$-algebras given by
\[	\underline\Hom^\otimes(\omega,\eta)(A'):=\Hom^\otimes(\iota_{A'}\circ\omega,\iota_{A'}\circ\eta).	\]
Here $\underline{\Hom}^\otimes$ refers to natural transformations of functors which preserve tensor products.
\end{defn}

\begin{thm}[{\cite[Thm. 3.2]{dm}\label{g-x-homs}}]
Let $\omega:\Rep_k(G)\rightarrow \Vect_k$ be the natural forgetful functor.
\begin{enumerate}
\item	For any fiber functor $\eta:\Rep_k(G)\rightarrow \Proj_A$, $\underline\Hom^\otimes(\iota_A\circ\omega,\eta)$ is representable by an affine scheme faithfully flat over $\Spec A$; it is therefore a $G$-torsor.
\item	The functor $\eta\rightsquigarrow\underline\Hom^\otimes(\iota_A\circ\omega,\eta)$ is an equivalence between the category of fiber functors $\eta:\Rep_k(G)\rightarrow\Proj_A$ and the category of $G$-torsors over $\Spec A$.  The quasi-inverse assigns to any $G$-torsor $X$ over $A$ the functor $\eta$ sending any $\rho:G\rightarrow \GL(V)$ to the $M\in\Proj_A$ associated to the push-out of $X$ over $A$.
\end{enumerate}
\end{thm}

\begin{cor}
Let $\eta:\Rep_k(G)\rightarrow \Proj_A$ be a fiber functor, corresponding to a $G$-torsor $X\rightarrow \Spec A$.  Then the functor $\underline\Aut^\otimes(\eta)$ is representable by the $A$-group scheme $\Aut_G(X)$.  This is a form of $G_A$.
\end{cor}

We now assume that $\eta$ is equipped with an exact $\otimes$-filtration, i.e., for each $V\in \Rep_{k}(G)$, we have a decreasing filtration $\mathcal{F}^\bullet(\eta(V))$ of vector sub-bundles on each $\eta(V)$ such that
\begin{enumerate}
\item	the specified filtrations are functorial in $V$.
\item	the specified filtrations are tensor-compatible, in the sense that
\[	\mathcal{F}^n\eta(V\otimes_kV')=\sum_{p+q=n}\mathcal{F}^p\eta(V)\otimes_A\mathcal{F}^q\eta(V')\subset \eta(V)\otimes_A\eta(V').	\]
\item	$\mathcal{F}^n(\eta(\mathbf{1}))=\eta(\mathbf{1})$ if $n\leq 0$ and $\mathcal{F}^n(\eta(\mathbf{1}))=0$ if $n\geq 1$.
\item	the associated functor from $\Rep_{k}(G)$ to the category of graded projective $A$-modules is exact.
\end{enumerate}
Equivalently, an exact $\otimes$-filtration of $\eta$ is the same as a
factorization of $\eta$ through the category of filtered vector
bundles over $\Spec A$.

We define two auxiliary subfunctors of
$\underline\Aut^\otimes(\eta)$: 
\begin{itemize}
\item	$P_{\mathcal{F}}=\underline\Aut_{\mathcal{F}}^\otimes(\eta)$ is the functor on $A$-algebras such that
\begin{equation*}
\begin{split}
P_{\mathcal{F}}(A') = \{\lambda\in\underline{\Aut}^\otimes(\eta)(A')| \lambda(\mathcal{F}^n\eta(V))&\subset\mathcal{F}^n\eta(V)\text{ for all }\\V\in\Rep_k(G)\text{ and }n\in\Z\}.
\end{split}
\end{equation*}
\item	$U_{\mathcal{F}}=\underline\Aut_{\mathcal{F}}^{\otimes!}(\eta)$ is the functor on $A$-algebras such that
\begin{equation*}
\begin{split}
U_{\mathcal{F}}(A') = \{\lambda\in\underline{\Aut}^\otimes(\eta)(A')| (\lambda-\mathrm{id})(\mathcal{F}^n\eta(V))&\subset\mathcal{F}^{n+1}\eta(V)\text{ for all }\\V\in\Rep_k(G)\text{ and }n\in\Z\}.
\end{split}
\end{equation*}
\end{itemize}
By~\cite[Chapter IV, 2.1.4.1]{saavedra}, these functors are both
representable by closed subgroup schemes of $\Aut_G(X)$, and they are
smooth if $G$ is.  This holds for any affine group $G$ over $k$ (since
it is automatically flat); there is no need for reductivity or
connectedness hypotheses.  Furthermore, $\Lie
P_{\mathcal{F}}=\mathcal{F}^0(\Lie \underline{\Aut}^\otimes(\eta))$ and $\Lie U_{\mathcal{F}}=\mathcal{F}^1(\Lie
\underline{\Aut}^\otimes(\eta))$, by the same result. 

We also have a notion of a $\otimes$-grading on $\eta$: a \emph{$\otimes$-grading} of $\eta$ is the specification of a grading $\eta(V)=\oplus_{n\in\Z}\eta(V)_n$ of vector bundles on each $\eta(V)$ such that
\begin{enumerate}
\item	the specified gradings are functorial in $V$.
\item	the specified grading are tensor-compatible, in the sense that
\[	\eta(V\otimes_kV')_n=\bigoplus_{p+q=n}(\eta(V)_p\otimes_A\eta(V')_q).	\]
\item	$\eta(\mathbf{1})_0=\eta(\mathbf{1})$.
\end{enumerate}
Equivalently, a  $\otimes$-grading of $\eta$ is a factorization of $\eta$ through the category of graded vector bundles on $\Spec A$.  A $\otimes$-grading induces a homomorphism of $A$-group schemes $\Gm\rightarrow\underline{\Aut}^\otimes(\eta)$.

Given a $\otimes$-grading of $\eta$, we may construct a $\otimes$-filtration of $\eta$, by setting
\[	\mathcal{F}^n\eta(V)=\oplus_{n'\geq n}\eta(V)_{n'}.	\]
We say that a $\otimes$-filtration $\mathcal{F}^\bullet$ is \emph{splittable} if it arises
in this way, and we say that $\mathcal{F}^\bullet$ is \emph{locally
  splittable} if fpqc-locally on $\Spec A$ it arises in this way.  A \emph{splitting} of $\mathcal{F}^\bullet$ is a $\otimes$-grading on $\eta$ giving rise to $\mathcal{F}^\bullet$.

Given an exact $\otimes$-filtration $\mathcal{F}^\bullet$ on $\eta$, we may define a fiber functor $\gr(\eta)$ equipped with a $\otimes$-grading, by setting
\[	\gr(\eta)(V)_n:=\mathcal{F}^n(V)/\mathcal{F}^{n+1}(V)	\]
Thus, a splitting of $\mathcal{F}^\bullet$ is equivalent to an isomorphism of filtered fiber functors $\gr(\eta)\cong \eta$.

In fact, by a theorem of Deligne (proved in \cite[Chapter IV,
2.4]{saavedra}), every $\otimes$-filtration is locally splittable (in
fact, splittable Zariski-locally on $\Spec A$), because $G$ is smooth
and $A$ has characteristic $0$ (this result also holds under various
other sets of hypotheses on $G$ and $A$). 
Again, this does not require $G$ to be reductive or connected.  If $\lambda:\Gm\rightarrow \underline\Aut^\otimes(\eta)$ is a cocharacter splitting the filtration, then $P_{\mathcal{F}}=U_{\mathcal{F}}\rtimes Z_G(\lambda)$, by~\cite[Chapter IV, 2.1.5.1]{saavedra}.  In particular, $\lambda$ factors through $P_{\mathcal{F}}$.

If $\mathcal{F}^\bullet$ is a splittable filtration on $\eta$, we may consider
the functor $\underline{\mathrm {Scin}}(\eta,\mathcal{F}^\bullet)$ of
splittings. 
Then $\underline{\mathrm {Scin}}(\eta,\mathcal{F}^\bullet)$ is
the same as the functor
$\underline\Isom_{\mathcal{F}}^{\otimes!}(\gr_{\mathcal{F}}(\eta),\eta)$, which is the subset of $\underline\Isom_{\mathcal{F}}^\otimes(\gr_{\mathcal{F}}(\eta),\eta)$ inducing the identity $\gr_{\mathcal{F}}(\eta)\rightarrow\gr_{\mathcal{F}}(\eta)$. 
Thus, $\underline{\mathrm {Scin}}(\eta,\mathcal{F}^\bullet)$ is a left torsor under $U_{\mathcal{F}}$.  It follows that the composition $\lambda:\Gm\rightarrow P_{\mathcal{F}}\rightarrow P_{\mathcal{F}}/U_{\mathcal{F}}$ is independent of the choice of splitting.

In other words, $P_{\mathcal{F}}$ and $U_{\mathcal{F}}$ depend only on
the filtration, and if it is locally splittable, 
there is a homomorphism $\overline\lambda:\Gm\rightarrow P_{\mathcal{F}}/U_{\mathcal{F}}$ which also only depends on the filtration.  If the filtration is actually splittable, a choice of splitting lets us lift $\overline\lambda$ to a cocharacter $\lambda:\Gm\rightarrow P_{\mathcal{F}}$.  In that case, since both $\underline{\mathrm {Scin}}(\eta,\mathcal{F})$ and the set of lifts of cocharacters from $P_{\mathcal{F}}/U_{\mathcal{F}}$ to $P_{\mathcal{F}}$ are torsors under $U_{\mathcal{F}}$ (in the latter case, $U_{\mathcal{F}}$ acts by conjugation), they are isomorphic.  In particular, any two cocharacters $\lambda,\lambda':\Gm\rightrightarrows P_{\mathcal{F}}$ splitting the $\otimes$-filtration $\mathcal{F}$ are conjugate by $U_{\mathcal{F}}$.

Let $\mathcal{G}:=\underline\Aut^\otimes(\eta)$, so that the geometric fibers of $\mathcal{G}$ are isomorphic to $G_{\overline k}$. Then for any geometric point $x\in\Spec A$,
the $G^\circ(\kappa(x))$-conjugacy class of $\mathcal{F}_x^\bullet$ induces a
unique $G^\circ(\kappa(x))$-conjugacy class of cocharacters, and this
conjugacy class is Zariski-locally constant on $\Spec A$.

Recall that when $\lambda:\Gm\rightarrow \mathcal{G}$ is a cocharacter, we defined subgroups $U_{\mathcal{G}}(\lambda)\subset P_{\mathcal{G}}(\lambda)\subset \mathcal{G}$ in~\textsection\ref{subsec: notation}.
\begin{prop}\label{prop: parabolic-splittable-filtration}
Suppose that $G$ is a (possibly disconnected) algebraic group.  Let $\eta:\Rep_{k}(G)\rightarrow \Proj_A$ be a fiber functor equipped with a splittable exact $\otimes$-filtration $\mathcal{F}^\bullet$, and let $\lambda:\Gm\rightarrow \underline\Aut^\otimes(\eta)$ be a splitting.  Let $\mathcal{G}$ denote the group scheme representing $\underline\Aut^\otimes(\eta)$.  Then $P_{\mathcal{F}}=P_{\mathcal{G}}(\lambda)$, $U_{\mathcal{F}}=U_{\mathcal{G}}(\lambda)$, and the fibers of $U_{\mathcal{F}}$ are connected.
\end{prop}
\begin{proof}
We consider the map $\mu:\Gm\times P_{\mathcal{F}}\rightarrow
\underline\Aut^\otimes(\eta)$ defined by
$\mu(t,g):=\lambda(t)g\lambda(t^{-1})$, and for $g\in
P_{\mathcal{F}}(A')$, we let
$\mu_g:(\Gm)_{A'}\rightarrow(\underline\Aut^\otimes(\eta))_{A'}$ be the restriction
$\mu|_{\Gm\times\{g\}}$. 
Let $\sigma:G\rightarrow\GL(V)$ be a representation of $G$.  Then the pushout $\eta(V)$ is a filtered vector bundle, and if $g\in P_{\mathcal{F}}(A')$, the action of $g$ preserves the filtration on $\eta(V)$.  The choice of a splitting in particular specifies an isomorphism $\gr^\bullet(\eta(V))\xrightarrow{\sim}\eta(V)$, and $t\in \Gm(A')$ acts via $t^n$ on $(\eta(V))_n$.

Let $\sigma_\ast(\lambda)$ denote the corresponding cocharacter $\sigma_\ast(\lambda):\Gm\rightarrow\Aut_{\GL(V)}(\eta(V))$.  Since this cocharacter induces the filtration on $\eta(V)$, we see that the morphism 
\[	\sigma_\ast(\mu_g):=\sigma_\ast(\lambda)(t)g\sigma_\ast(\lambda)(t^{-1}):\Gm\rightarrow P_{\Aut_{\GL(V)}(\eta(V))}(\sigma_\ast(\lambda))	\]
extends uniquely to a morphism
\[	\widetilde{\sigma_\ast(\mu_g)}:\A^1\rightarrow P_{\Aut_{\GL(V)}(\eta(V))}(\sigma_\ast(\lambda)).	\]

We claim that the collection
$\{\widetilde{\sigma_\ast(\mu_g)}\}_\sigma$ is functorial in $\sigma$
and tensor-compatible.  Indeed, since the collection
$\{\widetilde{\sigma_\ast(\mu_g)}|_{\Gm}\}_\sigma$ is functorial in
$\sigma$ and tensor-compatible, and the extensions to $\A^1$ are
unique, it follows that $\{\widetilde{\sigma_\ast(\mu_g)}\}_\sigma$ is
functorial in $\sigma$ and tensor-compatible.  Thus, there is a
morphism
$\widetilde{\mu_g}:\A^1\rightarrow\underline\Aut_{\mathcal{F}}^\otimes(\eta)$ whose
restriction to $\Gm$ is $\mu_g$.  It follows that $g\in
P_{\mathcal{G}}(\lambda)(A')$.

Suppose in addition that $g\in U_{\mathcal{F}}(A')$.  Then for every representation $\sigma:G\rightarrow \GL(V)$, $g$ induces the identity map from $\gr^\bullet(\sigma(\mathcal{F}^\bullet))$ to itself.  It follows that $\widetilde{\sigma_\ast(\mu_g)}(0)=\mathbb{1}$ for all $\sigma$, and therefore $\widetilde{\mu_g}(0)=\mathbb{1}$.

On the other hand, if $g\in P_{\mathcal{G}}(\lambda)(A')$, then the morphism $\mu_g:(\Gm)_{A'}\rightarrow\underline{\Aut}^\otimes(\eta)_{A'}$ defined by $t\mapsto \lambda(t)g\lambda(t^{-1})$ extends to a morphism $\widetilde{\mu_g}:(\A^1)_{A'}\rightarrow\underline{\Aut}^\otimes(\eta)_{\A'}$.  It therefore induces a family of morphisms
\[	\sigma_\ast(\widetilde{\mu_g}):(\A^1)_{A'}\rightarrow\GL(V)_{A'}	\]
and so $\sigma_{\ast}(g)\in P_{\Aut_{\GL(V)}(\eta(V))}(\sigma_\ast(\lambda))$.  But then $\sigma_\ast(g)$ preserves the filtration on $\eta(V)$ induced by $\sigma_\ast(\lambda)$; since this holds for all $V\in \Rep_k(G)$, $g\in P_{\mathcal{F}}(A')$.  A similar argument shows that if $g\in U_{\mathcal{G}}(\lambda)(A')$, then $g\in U_{\mathcal{F}}(A')$.

Finally, since $\widetilde\mu_g:\A^1\rightarrow \underline\Aut^\otimes(\eta)$ is a morphism from a connected scheme such that $\widetilde\mu_g(0)=\mathbb{1}$ and $\widetilde\mu_g(1)=g$, we see that $g$ is in the connected component of the identity for all $g\in U_\mathcal{F}(A')$.
\end{proof}

%
\begin{lemma}
Let $\mathcal{F}^\bullet$ be a locally splittable exact $\otimes$-filtration on $\eta$.  Then the geometric fibers of $P_{\mathcal{F}}$ are parabolic subgroups of $G_{\overline k}$.
\end{lemma}
\begin{proof}
We may work locally on $\Spec A$ and assume that we have a cocharacter $\lambda:\Gm\rightarrow\mathcal{G}_A$ splitting the exact $\otimes$-filtration.  Then $P_{\mathcal{F}}\cong P_{\mathcal{G}}(\lambda)$.  Since the formation of $P_{\mathcal{G}}(\lambda)$ commutes with base
change on $A$, we may assume that $A=k=\overline k$ and $\mathcal{G}=G=G_{\overline k}$.  Then $P_{G^\circ}(\lambda)\subset G^\circ$ is a parabolic subgroup, so $G^\circ/P_{G^\circ}(\lambda)$ is proper.  There is a sequence of maps 
\[	G^\circ/P_{G^\circ}(\lambda)\rightarrow G/P_{G^\circ}(\lambda)\twoheadrightarrow G/P_G(\lambda)	\]
Since $G^\circ\subset G$ has finite index, the properness of $G^\circ/P_{G^\circ}(\lambda)$ implies the properness of $G/P_{G^\circ}(\lambda)$.  This implies that $G/P_G(\lambda)$ is proper, so $P_G(\lambda)\subset G$ is a parabolic subgroup.
\end{proof}We will also need the following result:
\begin{theorem}[{\cite[IX.3.6]{SGA3.II}}]\label{thm:lift-cocharacters}
Let $S$ be an affine scheme, $S_0$ a subscheme defined by
a nilpotent ideal $J$, $H$ a group of multiplicative type over $S$, $G$ a smooth group scheme
over $S$, $\mu_0 : H\times_S S_0\rightarrow G\times_S S_0$ a
homomorphism of $S_0$-groups. 

Then there exists a homomorphism $\mu:H\rightarrow G$ of $S$-groups
which lift $\mu_0$, and any two such lifts are conjugate by an
element of $G(S)$ which reduces to the identity modulo~ $J$.
\end{theorem}

\begin{cor}\label{cor:forgetting filtrations is smooth}
Let $A$ be an artin local $k$-algebra with maximal ideal
$\mathfrak{m}_A$, and let $I\subset A$ be an ideal such that
$I\mathfrak{m}_A=(0)$.  Then if $D_A$ is a $G$-torsor over $A$ such
that the reduction $D_{A/I}:=D_A\otimes_AA/I$ is equipped with an
exact $\otimes$-filtration $\mathcal{F}_{A/I}^\bullet$, then the set of
lifts of $\mathcal{F}_{A/I}^\bullet$ to an exact $\otimes$-filtration
on $D_A$ is non-empty, and is a torsor under $I\otimes_{A/\mathfrak{m}_A}(\ad
D_{A/\mathfrak{m}_A}/\mathcal{F}_{A/\mathfrak{m}_A}^0(\ad
D_{A/\mathfrak{m}_A}))$.
\end{cor}
\begin{proof}
Suppose that $D_{A/I}$ is a $G$-torsor over $\Spec A/I$, equipped with an exact $\otimes$-filtration $\mathcal{F}^\bullet_{A/I}$.  Since $A/I$ is local, $\mathcal{F}^\bullet_{A/I}$ is split, so it is induced by a cocharacter $\lambda_{A/I}:\Gm\rightarrow\Aut_G(D_{A/I})$.  By Theorem~\ref{thm:lift-cocharacters}, $\lambda_{A/I}$ lifts to a cocharacter $\lambda_A:\Gm\rightarrow \Aut_G(D_A)$.  Then $\lambda_A$ induces an exact $\otimes$-filtration $\mathcal{F}^\bullet_A$ on $D_A$ which lifts that on $D_{A/I}$.

Suppose there are two exact $\otimes$-filtrations, $\mathcal{F}^\bullet_A$ and ${\mathcal{F}'_A}^\bullet$ on $D_A$ lifting $\mathcal{F}^\bullet_{A/I}$, induced by cocharacters $\lambda_A$ and $\lambda_A'$, respectively, which lift $\lambda_{A/I}$.  Then $\lambda_A$ and $\lambda_A'$ are conjugate by an element of $\Aut_G(D_A)$ which is the identity modulo $I$.  In other words, there is some $j\in \ad D_{A/\mathfrak{m}_A}\otimes_{A/\mathfrak{m}_A}I$ such that $\lambda_A'=(1+j)\lambda_A(1-j)$.  This implies that $\mathcal{F}^\bullet_A$ and ${\mathcal{F}'_A}^\bullet$ are conjugate.

On the other hand, conjugation by $1+j$ preserves $\mathcal{F}^\bullet_A$ if
and only if $1+j\in P_{\mathcal{F}_A}(\Aut_G(D_A))$.  This holds if
and only if $j\in
\mathcal{F}_{A/\mathfrak{m}_A}^0\Lie\Aut_G(D_{A/\mathfrak{m}_A})\otimes_{A/\mathfrak{m}_A}I=\mathcal{F}_{A/\mathfrak{m}_A}^0\ad
D_{A/\mathfrak{m}_A}\otimes_{A/\mathfrak{m}_A}I$.
\end{proof}

\subsection{$p$-adic Hodge theory}\label{subsec: p-adic Hodge
  theory}
Our goal
is to study deformations of potentially semi-stable Galois
representations.  That is, we wish to consider deformations of
representations $\rho:\Gal_K\rightarrow G(E)$ such that
$\rho|_{\Gal_L}$ is semi-stable.  Such representations can be
described by linear algebra.  
Briefly, for every representation
$\sigma:G\rightarrow\GL_d$, $\sigma\circ\rho$ is a potentially
semi-stable representation, and $D_{\st}^L(\sigma\circ\rho)$ is a
weakly admissible filtered $(\varphi,N,\Gal_{L/K})$-module.  The
formation of $D_{\st}^L(\sigma\circ\rho)$ is exact and
tensor-compatible in $\sigma$, and if $\mathbf{1}$ denotes the trivial
representation of $G$, then $D_{\st}^L(\mathbf{1}\circ\rho)$ is the
trivial filtered $(\varphi,N,\Gal_{L/K})$-module with coefficients in $E$.

Therefore, as in~\cite[\textsection A.2.8-9]{2014arXiv1403.1411B}, $\sigma\mapsto D_{\st}^L(\sigma\circ\rho)$ is a fiber
functor $\eta:\Rep_E(G)\rightarrow \Proj_{E\otimes_{\Qp}L_0}$, and we obtain from~$\rho$ a
$G$-torsor $D=D_{\st}^L(\rho)$ 
over $E\otimes L_0$ equipped with 
\begin{itemize}
\item	an isomorphism $\Phi:\varphi^\ast D\isoto D$,
\item	a nilpotent element $N\in\Lie\Aut_GD$,
\item	for each $g\in\Gal_{L/K}$, an isomorphism $\tau(g):g^\ast D\isoto D$,
\item	a $\Gal_{L/K}$-stable exact $\otimes$-filtration on $D_L$, or equivalently (by Galois descent), an exact $\otimes$-filtration on the $\Res_{E\otimes K/E}G$-torsor $D_L^{\Gal_{L/K}}$ over $K$.
\end{itemize}
These satisfy the requisite compatibilities such that forgetting the filtration on~$D_{\st}^L(\rho)$ gives
us an object of~$G-\mathrm{Mod}_{L/K,\varphi,N}$.

\begin{defn}
The category of $G$-valued filtered $(\varphi,N,\Gal_{L/K})$-modules, which we denote $G-\Mod_{L/K,\varphi,N,\Fil}$, is the category cofibered in groupoids over $E$-Alg whose fiber over an $E$-algebra $A$ consists of a $\Res_{E\otimes L_0/E}G$-torsor $D$ over $A$, equipped with:
\begin{itemize}
\item	an isomorphism $\Phi:\varphi^\ast D\isoto D$,
\item	a nilpotent element $N\in\Lie\Aut_GD$,
\item	for each $g\in\Gal_{L/K}$, an isomorphism $\tau(g):g^\ast D\isoto D$,
\item	a $\Gal_{L/K}$-stable exact $\otimes$-filtration on $D_L$, or equivalently, an exact $\otimes$-filtration on the $\Res_{E\otimes K/E}G$-torsor $D_L^{\Gal_{L/K}}$ over $A$.
\end{itemize}
The $\Res_{E\otimes L_0/E}G$-torsor $D$, together with $\Phi$, $N$, and $\{\tau(g)\}_{g\in\Gal_{L/K}}$, is required to be an object of $G-\Mod_{L/K,\varphi,N}$.
\end{defn}

\begin{defn}\label{defn: p adic Hodge type}
Suppose that $\rho:\Gal_K\rightarrow G(E)$ is a potentially semi-stable Galois representation which becomes semi-stable when restricted to $\Gal_L$.  The \emph{$p$-adic Hodge type} $\mathbf{v}$ of $\rho$ is the $(\Res_{E\otimes K/E}G)^\circ(\overline E)$-conjugacy class of cocharacters $\lambda:\Gm\rightarrow (\Res_{E\otimes K/E}G)_{\overline E}$ which split the $\otimes$-filtration on $D_{\st}^L(\rho)_L^{\Gal_{L/K}}$.  We let $P_{\mathbf{v}}$ denote the $(\Res_{E\otimes K/E}G)^\circ(\overline E)$-conjugacy class of $P_{\Res_{E\otimes K/E}G}(\lambda)$ for $\lambda\in\mathbf{v}$.
\end{defn}

While we do not need it, for completeness we record the following
definition and result, which control the deformation theory of
filtered $(\varphi,N,\Gal_{L/K})$-modules. Given an object $D_A\in G-\Mod_{L/K,\varphi,N,\Fil}$, we consider the diagram
\[
\xymatrix{ 
(\ad D_A)^{\Gal_{L/K}} \ar[r]\ar[d] &  (\ad D_A)^{\Gal_{L/K}}\oplus(\ad D_A)^{\Gal_{L/K}}\ar[r] & (\ad D_A)^{\Gal_{L/K}} \\
(\ad D_{A,L}/\!\Fil^0\!\ad D_{A,L})^{\Gal_{L/K}} & &
}
\]
where the top line is the total complex of
\[\xymatrix{
(\ad D_A)^{\Gal_{L/K}}\ar[r]^{1-\underline\Ad(\Phi)}\ar[d]^{\ad_N} & (\ad D_A)^{\Gal_{L/K}}\ar[d]^{\ad_N}	\\
(\ad D_A)^{\Gal_{L/K}}\ar[r]^{p\underline\Ad(\Phi)-1} & (\ad D_A)^{\Gal_{L/K}}
}\]
and the vertical map is the natural quotient map.  We let $C_{\Fil}^\bullet$ denote its total complex.  Then $C_{\Fil}^\bullet$ controls the deformation theory of $D_A$:
\begin{prop}\label{prop: tangent-obstruction for filtered
    objects}
Let $A$ be an artin local $E$ algebra with maximal ideal $\mathfrak{m}_A$ and let $I\subset A$ be an ideal such that $I\mathfrak{m}_A=(0)$.  Let $D_{A/I}$ be an object of $G-\Mod_{L/K,\varphi,N,\Fil}(A/I)$ and set $D_{A/\mathfrak{m}_A}:=D_{A/I}\otimes_{A/I}A/\mathfrak{m}_A$.
\begin{enumerate}
\item	If $H_{\Fil}^2(D_{A/I})=0$, then there exists an object $D_A\in G-\Mod_{L/K,\varphi,N,\Fil}(A)$ lifting $D_{A/I}$.
\item	The set of isomorphism classes of lifts of $D_{A/I}$ to $D_A\in G-\Mod_{L/K,\varphi,N,\Fil}(A)$ is either empty or a torsor under $H_{\Fil}^1(D_{A/\mathfrak{m}_A})\otimes_{A/\mathfrak{m}_A}I$.
\end{enumerate}
\end{prop}
\begin{proof}
This follows by combining~\cite[Proposition 3.2]{2014arXiv1403.1411B} and Corollary~\ref{cor:forgetting filtrations is smooth}.
\end{proof}





\section{Local deformation rings}\label{sec: local deformation
  rings}As in Section~\ref{subsec: deformation
  rings}, we let~$K/\Q_p$ be a finite extension for some
prime~$p$, possibly equal to~$l$, and let $\rhobar:\Gal_K\to G(\F)$ be a
continuous representation.  We have a universal framed deformation $\cO$-algebra
$R_{\rhobar}^\square$, and if we fix a  a homomorphism ~$\psi:\Gamma\to\Gab(\cO)$
such that $\ab\circ\rhobar=\psibar$, we also have the
quotient~$R_{\rhobar}^{\square,\psi}$ 
corresponding to framed deformations~$\rho$
with~$\ab\circ\rho=\psi$. When we define quotients
of~$R_{\rhobar}^\square$, there are corresponding quotients
of~$R_{\rhobar}^{\square,\psi}$, which we will not explicitly define,
but will denote by a superscript~$\psi$.   
An inertial type is by definition a $G^\circ(\overline{E})$-conjugacy
class of representations $\tau:I_K\to G(\overline{E})$ with open
kernel which admit extensions to~$\Gal_K$; any such~$\tau$ is defined
over some finite extension of~$E$. We choose a finite Galois extension
$L/K$ for which~$\tau|_{I_L}$ is trivial.  If $E'/E$ is a finite
extension, and $\rho:\Gal_K\to G(E')$ is a representation, which we
assume to be potentially semi-stable if
$l=p$, 
then we say that
$\rho$ has type~$\tau$ if the restriction to~$I_K$ (forgetting~$N$) of
the corresponding Weil--Deligne representation $\WD(\rho)$ is
equivalent to~$\tau$. 

\subsection{The case $l\neq p$} Suppose firstly that $l\ne p$. The proof of~\cite[Prop.\ 3.0.12]{MR3152673} shows that for
each~$\tau$ we may define a $\Zl$-flat
quotient~$R_{\rhobar}^{\square,\tau}$ of $R_{\rhobar}^\square$ whose characteristic $0$ points correspond to
representations of type~$\tau$.
The usual construction of the Weil--Deligne
representation associated to a Galois representation makes sense over
$R_{\rhobar}^{\square}[1/l]$, so 
we have a natural morphism 
\[ \Spec R_{\rhobar}^{\square,\tau}[1/l]\to G-\WD_E(L/K). \]




\subsection{The case $l=p$}\label{subsec: l equals p}
Now suppose that $l=p$.  If we fix a $p$-adic Hodge type~$\mathbf{v}$  in the sense of Definition~\ref{defn: p adic Hodge
  type} (that is, a  $(\Res_{E\otimes K/E}G)^\circ(\overline
E)$-conjugacy class of cocharacters $\lambda:\Gm\rightarrow
(\Res_{E\otimes K/E}G)_{\overline E}$), and an inertial type~$\tau$,
then by~\cite[Prop.\ 3.0.12]{MR3152673} there is a
unique $\Zl$-flat quotient $R_{\rhobar}^{\square,\tau,\mathbf{v}}$
of~$R_{\rhobar}^{\square}$ with the property that if $B$ is a finite
local $E$-algebra, then a morphism
$R_{\rhobar}^{\square}\to B$ factors through
$R_{\rhobar}^{\square,\tau,\mathbf{v}}$ if and only if the
corresponding representation $\rho:\Gal_K\to G(B)$ is potentially semi-stable
with Hodge type~$\mathbf{v}$ and inertial type $\tau$. 
For each finite-dimensional representation $V$ of $G$, we may compose with
the representation $\Gal_K\rightarrow
G(R_{\rhobar}^{\square,\tau,\mathbf{v}}[1/p])$ to obtain a representation
$\Gal_K\rightarrow\GL(V)(R_{\rhobar}^{\square,\tau,\mathbf{v}}[1/p])$.
Then exactly as in~\cite[Thm.\ 2.5.5]{MR2373358} we obtain a
corresponding ($\GL(V)$-valued) filtered $(\varphi,N,\Gal_{L/K})$-module over
$R_{\rhobar}^{\square,\tau,\mathbf{v}}[1/p]$ (note that we have been working
with covariant functors in this paper, while Kisin uses contravariant
functors, it is necessary to dualize the
construction in~\cite[\textsection 2.4]{MR2373358}). 
As these
filtered $(\varphi,N,\Gal_{L/K})$-modules are exact and
tensor-compatible, we obtain a $G$-valued filtered
$(\varphi,N,\Gal_{L/K})$-module over
$R_{\rhobar}^{\square,\tau,\mathbf{v}}[1/p]$.  By Lemma~\ref{lem: WD stack agrees with phi N stack}, we again have a natural
morphism 
\[ \Spec R_{\rhobar}^{\square,\tau,\mathbf{v}}[1/l]\to G-\WD_E(L/K). \]

\subsection{Denseness of very smooth points}

We continue to fix an inertial type $\tau$ and (if $p=l$) a $p$-adic Hodge type $\mathbf{v}$.  For convenience, if $l\ne p$ then for the rest of this section
we write $R_{\rhobar}^{\square,\tau,\mathbf{v}}$ for
$R_{\rhobar}^{\square,\tau}$; this notational convention allows us to
treat the cases $l\ne p$ and $l=p$ simultaneously. We study the
generic fibre $R_{\rhobar}^{\square,\tau,\mathbf{v}}[1/l]$ via the
morphism \numequation\label{eqn:morphism from deformation ring to WD
  stack} \Spec R_{\rhobar}^{\square,\tau,\mathbf{v}}[1/l]\to
  G-\WD_E(L/K). \end{equation}

In a standard abuse of terminology, we say that a closed point  $x\in\Spec
R_{\rhobar}^{\square,\tau,\mathbf{v}}[1/l]$ is \emph{smooth}  if the
(completed) local ring at~$x$ is regular. We will see in the proof of
Theorem~\ref{thm: dense set of very smooth points} that these are the
points whose images in~$G-\WD_E(L/K)$ are smooth points, which perhaps
justifies this terminology. Similarly, we say that~$x$ is \emph{very
  smooth} if for any finite extension~$K'/K$, the image of~$x$ in
(with obvious notation) $\Spec R_{\rhobar|_{G_{K'}}}^{\square,\tau|_{I_{K'}},\mathbf{v}_{K'}}[1/l]$ is smooth.

As in~\cite[Proposition 2.3.5]{MR2600871}, if $x\in\Spec
R_{\rhobar}^{\square,\tau,\mathbf{v}}[1/l]$ is a closed point
corresponding to a representation $\rho_x$, then the completed local
ring $A_x$ at $x$ pro-represents framed deformations of $\rho_x$ which
are potentially semi-stable of $p$-adic Hodge type $\mathbf{v}$ (if
$l=p$), and have inertial type $\tau$.

\begin{prop}
  \label{prop: morphism to WD is formally smooth and flat}
  \begin{enumerate}
  \item If~$x$ is a closed point of the Jacobson scheme
    $\Spec R_{\rhobar}^{\square,\tau,\mathbf{v}}[1/l]$, then the
    completion at~$x$ of the 
    morphism~\emph{(\ref{eqn:morphism from deformation ring to WD stack})} is
formally smooth.
  \item The morphism~\emph{(\ref{eqn:morphism from deformation ring to WD stack})} is flat.
  \end{enumerate}
\end{prop}
\begin{proof}The formal smoothness follows from the proofs
  of~\cite[Lemma 3.2.1, Proposition 3.3.1]{MR2373358}, which carries over verbatim
  to our setting (since the morphism of groupoids from framed deformations to unframed deformations is formally
  smooth). 
  Part~(2) then follows from the fact that formally
  smooth morphisms between locally noetherian schemes are flat, which
  in turn follows from ~\cite[\S0 Thm.\
  19.7.1]{MR0173675}. 
\end{proof}

\begin{thm}
  \label{thm: dense set of very smooth points}Assume that
  $R_{\rhobar}^{\square,\tau,\mathbf{v}}\ne 0$. There is a dense open
  subscheme $U\subset\Spec R_{\rhobar}^{\square,\tau,\mathbf{v}}[1/l]$
  which is regular, and there is a Zariski dense subset of $\Spec R_{\rhobar}^{\square,\tau,\mathbf{v}}[1/l]$
  consisting of very smooth points.  Furthermore, $\Spec
  R_{\rhobar}^{\square,\tau,\mathbf{v}}[1/l]$ is equidimensional of
  dimension $\dim G+\delta_{l=p}\dim\Res_{E\otimes
    K/E}G/P_{\mathbf{v}}$, locally a complete intersection, and reduced. 

Similarly,
  $\Spec R_{\rhobar}^{\square,\tau,\mathbf{v},\psi}[1/l]$ contains a regular dense open
  subscheme and a Zariski dense subset of very smooth points, and is  equidimensional of
  dimension~$\dim \Gder + \delta_{l=p}\dim(\Res_{E\otimes K/E}G)/P_{\mathbf{v}}$.
\end{thm}
\begin{remark}
In contrast to previous work (in particular the
papers~\cite{MR2373358}, \cite{MR2785764} and \cite{2014arXiv1403.1411B}), we only claim that $U$ is regular, not formally smooth over $\Q_p$.  We are grateful to Jeremy Booher and Stefan Patrikis~\cite{2017arXiv170807434} for drawing our attention to this.
\end{remark}
\begin{proof}
  Since the formation of scheme-theoretic images is compatible with
  flat base change, the existence of a dense open subscheme~$U$ consisting of smooth points
  follows from Corollary~\ref{cor: WD stack is generically smooth zero dimensional
  controlled by H2} and Proposition~\ref{prop: morphism to WD is formally smooth and flat}.  The existence of a Zariski dense subset of very smooth points follows from Corollary~\ref{cor: WD stack is generically very smooth zero dimensional
  controlled by H2}. 
We claim that if $x\in \Spec
R_{\rhobar}^{\square,\tau,\mathbf{v}}[1/l]$ is a closed point in $U$,
then the completion $A_x$
of~$R_{\rhobar}^{\square,\tau,\mathbf{v}}[1/l]$ at~$x$ is a formally
smooth $\Qp$-algebra, and is in particular regular.  Indeed, if $\mathfrak{m}_x$ is the maximal
ideal of $A_x$, then $\Spec A_x/\mathfrak{m}_x^n\subset U$ for all
$n\geq 1$ (since $U$ is open).  Let $B$ be a local $\Q_p$-algebra with
maximal ideal $\mathfrak{m}_B$ and let $I\subset B$ be an ideal such
that $I\mathfrak{m}_B=(0)$.  If there is a local homomorphism
$A_x\rightarrow B/I$, let $D_{B/I}$ be the induced object of
$G-\WD_E(L/K)(B/I)$.  Then $H^2(\ad D_{B/I})=0$, since the homomorphism
$A_x\rightarrow B/I$ factors through $A/\mathfrak{m}_x^n$ for some
$n$.  It follows that $D_{B/I}$ lifts to $D_B\in G-\WD_E(L/K)(B)$.
Since $\Spf A_x\rightarrow G-\WD_E(L/K)$ is formally smooth, $D_B$ is
induced from a map $A_x\rightarrow B$ lifting $A\rightarrow
B/I$. Since~$R_{\rhobar}^{\square,\tau,\mathbf{v}}[1/l]$ is
Noetherian, it follows that the localisation of
~$R_{\rhobar}^{\square,\tau,\mathbf{v}}[1/l]$ at~$x$ is
regular~\cite[\href{http://stacks.math.columbia.edu/tag/07NY}{Tag
  07NY}]{stacks-project}, so~$U$ is regular by~\cite[\href{http://stacks.math.columbia.edu/tag/02IT}{Tag 02IT}]{stacks-project}
, as claimed.

Thus, to compute the dimension of
$\Spec R_{\rhobar}^{\square,\tau,\mathbf{v}}[1/l]$, it is 
enough to compute the dimension of the tangent spaces at closed points
in~$U$. Let~$x$ be such a closed point, let $E'$ be its residue field, and write $A_x$ for the
completion of~$R_{\rhobar}^{\square,\tau,\mathbf{v}}[1/l]$
at~$x$. Since the morphism $\Spf A_x\to G-\WD_E(L/K)$ is formally
smooth by
Proposition~\ref{prop: morphism to WD is formally smooth and flat}, it
is  versal at~$x$. More precisely, in the case that 
$l\ne p$ we see (by the equivalence between Galois representations and
Weil--Deligne representations recalled in Section~\ref{subsec: l not p Galois to
  WD}) that  the induced map $\Spf A_x\to G-\WD_E(L/K)^{\wedge}_x$
(with the right hand side denoting the completion of the target
at~$x$) is a $\widehat{G}$-torsor, 
where $\widehat{G}$ is the completion of~$G_{E}$ along the closed
subgroup given by the centraliser of the representation corresponding to~$x$, in the sense that there is an evident isomorphism \[\Spf
  A_x\times\widehat{G}\isoto\Spf
  A_x\times_{G-\WD_E(L/K)^{\wedge}_x}\Spf A_x. \]In particular, we have
$\dim A_x\times_{G-\WD_E(L/K)_x^\wedge} A_x=\dim A_x+\dim\widehat{G}$, and the
claim about the dimension then follows from~\cite[Lem.\
2.40]{EGcomponents} and Corollary~\ref{cor: WD stack is generically very smooth zero dimensional
  controlled by H2}.

If $l=p$, let $D_x:=D_{\st}^L(\rho_x)$; it is equipped with a filtration $\mathcal{F}_x^\bullet$.  We consider the set $(\Spf A_x)(E'[\varepsilon])$.  Forgetting the framing on liftings is a formally smooth morphism of groupoids and makes the tangent space at $x$ into a $\Lie G$-torsor over the groupoid of unframed deformations.  But since $E'[\varepsilon]$ is an artin local $E$-algebra, by~\cite[Proposition 2.4]{2014arXiv1403.1411B} the category of (unframed) potentially semi-stable representations of $\Gal_K$ over $E'[\varepsilon]$ deforming $\rho_x$ is equivalent to the subcategory of $G-\Mod_{L/K,\varphi,N,\Fil}(E'[\varepsilon])$ deforming $D_{\st}^L(\rho_x)$.  

There is a natural morphism of groupoids
\[
\xymatrix{
G-\Mod_{L/K,\varphi,N,\Fil}\ar[r] & G-\Mod_{L/K,\varphi,N}	}	\]
and therefore a commutative diagram
\[	\xymatrix{
G-\Mod_{L/K,\varphi,N,\Fil}(E'[\varepsilon])\ar[r]\ar[d] & G-\Mod_{L/K,\varphi,N}(E'[\varepsilon])\ar[d]	\\
G-\Mod_{L/K,\varphi,N,\Fil}(E')\ar[r] & G-\Mod_{L/K,\varphi,N}(E')
}	\]
By Corollary~\ref{cor:forgetting filtrations is smooth}, the fibers of \[
\xymatrix{
G-\Mod_{L/K,\varphi,N,\Fil}(E'[\varepsilon])\ar[r] & G-\Mod_{L/K,\varphi,N}(E'[\varepsilon])	}	\]
over the filtered $G$-torsor $D_x$ are torsors under $\left(\ad D_x/\mathcal{F}^0(\ad D_x)\right)^{\Gal_{L/K}}$.  Since $G-\Mod_{L/K,\varphi,N}\cong G-\WD_{E}(L/K)$ is equidimensional of dimension $0$ and $x\in\Spec R_{\rhobar}^{\square,\tau,\mathbf{v}}[1/l]$ is a smooth point, we conclude that
\begin{align*}
\dim A_x&=\dim \Lie G+\dim \left(\ad D_x/\mathcal{F}^0(\ad D_x)\right)^{\Gal_{L/K}} 	\\
&= \dim G+ \dim\Res_{E\otimes K/E}G/P_{\mathbf{v}}
\end{align*}
as desired.

To prove that $R_{\rhobar}^{\square,\tau,\mathbf{v}}[1/l]$ is reduced and locally a complete intersection, we consider the fiber product $\Spec R_{\rhobar}^{\square,\tau,\mathbf{v}}[1/l]\times_{G-\WD_E(L/K)}Y_{L/K,\varphi,\mathcal N}$.  This is a $G$-torsor, hence smooth, over $\Spec R_{\rhobar}^{\square,\tau,\mathbf{v}}[1/l]$, so it suffices to prove that this fiber product is reduced and locally a complete intersection.  But by Proposition~\ref{prop: morphism to WD is formally smooth and flat}, the natural morphism $\Spec R_{\rhobar}^{\square,\tau,\mathbf{v}}[1/l]\times_{G-\WD_E(L/K)}Y_{L/K,\varphi,\mathcal N}\rightarrow Y_{L/K,\varphi,\mathcal N}$ is formally smooth, so completed local rings at points of $\Spec R_{\rhobar}^{\square,\tau,\mathbf{v}}[1/l]\times_{G-\WD_E(L/K)}Y_{L/K,\varphi,\mathcal N}$ are power series rings over completed local rings of $Y_{L/K,\varphi,\mathcal N}$.  Since the latter are reduced and complete intersection (by Corollary~\ref{cor: WD stack is generically very smooth zero dimensional controlled by H2}), the same holds for the former.

The corresponding statements for~
$R_{\rhobar}^{\square,\tau,\mathbf{v},\psi}$ can be proved in the same
way; we leave the details to the reader.
\end{proof}

The following is a generalisation of~\cite[Thm.\
D]{allen2014deformations} (which treats the case that $l=p$ and
$G=\GL_n$). We let~$x$ be a closed point
of~$R_{\rhobar}^{\square,\tau,\mathbf{v}}[1/l]$ with residue
field~$E_x$ (a finite extension of~$E$), and write~$\rho_x:\Gal_K\to
G(E_x)$ for the corresponding representation.
\begin{cor}
  \label{cor: smooth points given by WD rep condition}
The point~$x$ is a formally smooth point
of~$R_{\rhobar}^{\square,\tau,\mathbf{v}}[1/l]$ if and only if $H^0((\ad\WD(\rho_x))^*(1))=0$.
\end{cor}
\begin{proof}
Corollary~\ref{cor: WD stack is generically smooth zero dimensional controlled by H0} implies that the formally smooth points of $G-\WD_E(L/K)$ are precisely those points $x$ for which $H^0((\ad D_x)^*(1))$.  Thus, we need to show that $x\in\Spec R_{\rhobar}^{\square,\tau,\mathbf{v}}[1/l]$ is formally smooth if and only if its image in $G-\WD_E(L/K)$ is formally smooth.

We have a morphism $\Spec R_{\rhobar}^{\square,\tau,\mathbf{v}}[1/l]_x^\wedge\rightarrow G-\WD_E(L/K)_x^\wedge$, which is formally smooth by Proposition~\ref{prop: morphism to WD is formally smooth and flat}.  But this implies that for any $\Qp$-finite artin local ring $B$, the map $\Spec R_{\rhobar}^{\square,\tau,\mathbf{v}}[1/l]_x^\wedge(B)\rightarrow G-\WD_E(L/K)_x^\wedge(B)$ is surjective.  Hence, $\Spec R_{\rhobar}^{\square,\tau,\mathbf{v}}[1/l]_x^\wedge$ is formally smooth if and only if $G-\WD_E(L/K)_x^\wedge$ is formally smooth.
\end{proof}

\begin{rem}
  \label{rem: smooth equals generic}If~$G$ is the $L$-group of a
  quasisplit reductive group over~$K$, then it seems plausible that
  the condition of Corollary~\ref{cor: smooth points given by WD rep
    condition} could be equivalent to the condition that the
  (conjectural) $L$-packet of representations associated to the
  Frobenius semisimplification of
  $\WD(\rho_x)$ contains a generic element. In the case that
  $G=\GL_n$ (where the $L$-packets are singletons) and~$\WD(\rho_x)$
  is Frobenius semisimple, this is proved
  in~\cite[\S1]{allen2014deformations}, and in the general case it is
  closely related to~\cite[Conj.\ 2.6]{MR1186476} (which relates
  genericity to poles at $s=1$ of the adjoint $L$-function).
\end{rem}
\begin{rem}
  In the case that $l\ne p$, the equivalence between Galois
  representations and Weil--Deligne representations means that we can
  rewrite the condition in Corollary~\ref{cor: smooth points given by
    WD rep condition} as $H^0(\Gal_K,\ad\rho_x^*(1))=0$.
\end{rem}



We can also consider the quotient
$R_{\rhobar}^{\square,\tau,\mathbf{v},N=0}$, corresponding to the
union of the irreducible components
of~$R_{\rhobar}^{\square,\tau,\mathbf{v}}[1/l]$ for which the
monodromy operator~$N$ vanishes identically (if $l=p$, this is the
locus of potentially crystalline representations, and if $l\ne p$, it
is the locus of potentially unramified representations).
\begin{thm}\label{thm: dimensions etc l equals p N equals 0}
  Fix an inertial type~$\tau$, and if~$l=p$ then fix a $p$-adic Hodge
  type~$\mathbf{v}$. Assume that
  $R_{\rhobar}^{\square,\tau,\mathbf{v},N=0}\ne 0$. Then $R_{\rhobar}^{\square,\tau,\mathbf{v},N=0}[1/l]$
  is regular, and is equidimensional of
  dimension~$\dim_E G+\delta_{l=p}\dim_E(\Res_{E\otimes
    K/E}G)/P_{\mathbf{v}}$. Similarly
  $R_{\rhobar}^{\square,\tau,\mathbf{v},N=0,\psi}[1/l]$ is regular and equidimensional of
  dimension~$\dim_E \Gder + \delta_{l=p}\dim_E(\Res_{E\otimes K/E}G)/P_{\mathbf{v}}$.
\end{thm}
\begin{proof}This can be proved in exactly the same way as
  Theorem~\ref{thm: dense set of very smooth points}, replacing the use of
  the three term complex $\cC^\bullet(D)$ considered in Section~\ref{prop:
    deformations of WD repns controlled by complex} with the two term complex
  \[\xymatrix{ (\ad D_A)^{I_{L/K}}\ar[r]^{1-\Ad(\Phi)} &(\ad
      D_A)^{I_{L/K}}}\]
concentrated in degrees $0$ and $1$; see~\cite[Thm.\ 3.3.8]{MR2373358} for more details in the
  case that~$l=p$ and~$G=\GL_n$.
\end{proof}

\subsection{Components of deformation rings}\label{subsec: components
  of deformation rings}We now prove 
the following
reassuring lemma, which shows that the components of universal
deformation rings are invariant under $G(\cO)$-conjugacy. 
It  is a
generalization of \cite[Lem.\ 1.2.2]{BLGGT}, which treats the case $G=\GL_n$; the proof there is by an
explicit homotopy, while we use the theory of reductive group schemes over $\cO$ to construct less explicit homotopies.

  \begin{lem}
    \label{lem: conjugation preserves components}
  Let $h\in G(\cO')$ be an element which reduces to the identity modulo the maximal ideal, where $\cO'$ is the ring of integers in a finite extension of $E$.  Then conjugation by $h$ induces a map $\Spec (R_{\rhobar}^{\square,\tau,\mathbf{v}}\otimes_{\cO}\cO')[1/l]\rightarrow\Spec (R_{\rhobar}^{\square,\tau,\mathbf{v}}\otimes_{\cO}\cO')[1/l]$, and it fixes each irreducible component.
  \end{lem}

Before we prove it, we record a preliminary lemma on irreducible components of the generic fiber of $R_{\rhobar}^{\square,\tau,\mathbf{v}}$:
\begin{lemma}\label{lemma: analytify irred cpts}
Let $A:=\cO[\![X_1,\ldots,X_n]\!]/I$ be the quotient of a power series ring.  If $x,x'\in (\Spf A)^{\mathrm{rig}}$ lie on the same irreducible component, then they lie on the same irreducible component of $\Spec A[1/l]$.
\end{lemma}
\begin{proof}
If $x=x'$ as points of $(\Spf A)^{\mathrm{rig}}$, then by~\cite[Lemma 7.1.9]{dejong-formalrigid}, $x=x'$ as points of $\Spec A[1/l]$.  Thus, we may assume that $x\neq x'$.  Let $A\rightarrow \widetilde A$ denote the normalization of $A$.  Then by~\cite[Theorem 2.1.3]{conrad-rigid-irred-cpts}, $(\Spf \widetilde A)^{\mathrm{rig}}\rightarrow (\Spf A)^{\mathrm{rig}}$ is a normalization of the rigid space $(\Spf A)^{\mathrm{rig}}$, and $x, x'$ lift to points $\widetilde x, \widetilde x'\in (\Spf \widetilde A)^{\mathrm{rig}}$ on the same connected component.  By~\cite[Lemma 7.1.9]{dejong-formalrigid}, $\widetilde x$ and $\widetilde x'$ correspond to distinct closed points of $\Spec \widetilde{A}[1/l]$.

If $\widetilde x$ and $\widetilde x'$ lie on distinct connected components of $\Spec \widetilde{A}[1/l]$, there are idempotents $e_x, e_{x'}\in \widetilde{A}[1/l]$ such that $e_x$ is $1$ at $\widetilde x$ and $0$ at $\widetilde x'$ and $e_{x'}$ is $1$ at $\widetilde x'$ and $0$ at $\widetilde x$.  Again by~\cite[Lemma 7.1.9]{dejong-formalrigid}, the natural map $(\Spf \widetilde A)^{\mathrm{rig}}\rightarrow \Spec \widetilde A[1/l]$ induces isomorphisms on residue fields of closed points.  It follows that the pullbacks of $e_x$ and $e_{x'}$ to $(\Spf \widetilde A)^{\mathrm{rig}}$ are again idempotents (in the global sections of the structure sheaf of $(\Spf \widetilde A)^{\mathrm{rig}}$) such that $e_x$ is $1$ at $\widetilde x$ and $0$ at $\widetilde x'$ and $e_{x'}$ is $1$ at $\widetilde x'$ and $0$ at $\widetilde x$.  But this would contradict the fact that $\widetilde x$ and $\widetilde x'$ lie on the same connected component of $(\Spf \widetilde A)^{\mathrm{rig}}$, so they must actually lie on the same connected component of $\Spec \widetilde{A}[1/l]$.  This in turn implies that they lie on the same irreducible component of $\Spec A[1/l]$.
\end{proof}

\begin{proof}[Proof of Lemma~\ref{lem: conjugation preserves components}]
Let $R_{\rhobar}^{\square,\tau,\mathbf{v}}\otimes_{\mathcal{O}}\mathcal{O}''\rightarrow \cO''$ be a
homomorphism corresponding to a lift $\rho:\Gal_K\rightarrow
G(\cO'')$, where $\cO''$ is the ring of integers in a finite extension
of~$E$ and contains $\cO'$. 
We continue to
write $h$ for the image of~$h$ in $G(\cO'')$.  There is a finite surjective morphism $\Spec(R_{\rhobar}^{\square,\tau,\mathbf{v}}\otimes_{\mathcal{O}}\mathcal{O}'')[1/l]\rightarrow\Spec(R_{\rhobar}^{\square,\tau,\mathbf{v}}\otimes_{\cO}\cO')[1/l]$, so to show that conjugation by $h$ preserves irreducible components of $\Spec(R_{\rhobar}^{\square,\tau,\mathbf{v}}\otimes_{\cO}\cO')[1/l]$, it suffices to show that conjugation by $h$ preserves irreducible components of $\Spec(R_{\rhobar}^{\square,\tau,\mathbf{v}}\otimes_{\mathcal{O}}\mathcal{O}'')[1/l]$.  Moreover, by Lemma~\ref{lemma: analytify irred cpts}, it suffices to work with the rigid analytic generic fiber $\Spf(R_{\rhobar}^{\square,\tau,\mathbf{v}}\otimes_{\mathcal{O}}\mathcal{O}'')^{\mathrm{rig}}$ of $R_{\rhobar}^{\square,\tau,\mathbf{v}}\otimes_{\mathcal{O}}\mathcal{O}''$.

After possibly extending $\cO''$, we may assume that $G$ splits over $\cO''$.  Since $h$ is residually the identity element of $G$, it is a point of $G^\circ$.  After possibly further increasing $\cO''$, there is some Borel subgroup $B_{\cO''[1/l]}\subset G_{\cO''[1/l]}^\circ$ containing the image of $h$; it extends to a Borel subgroup $B\subset G_{\cO''}^\circ$ which contains $h$.  Since $\cO''$ is local, by~\cite[Proposition 5.2.3]{conrad-luminy} there is a cocharacter $\lambda:(\Gm)_{\cO''}\rightarrow G_{\cO''}^\circ$ such that $B=P_{G^\circ}(\lambda)=U_{G^\circ}(\lambda)\rtimes Z_{G^\circ}(\lambda)$.  Write $h_z$ for the projection of $h$ to $Z_{G^\circ}(\lambda)$ and $h_u$ for the projection to $U_{G^\circ}(\lambda)$.  Since this decomposition is unique, both $h_z$ and $h_u$ reduce to the identity modulo $\varpi$ (where $\varpi$ is a uniformizer of $\cO''$).  

Since $Z_{G^\circ}(\lambda)$ is a split torus, 
there is a map
$z_t:(\Gm)_{\cO''}\rightarrow G_{\cO''}^\circ$ which specializes to both
$h_z$ and the identity.  After analytifying this map, $h_z$ and the
identity lie in the same residue disk.  Choosing coordinates on this
residue disk, and rescaling them if necessary,
we obtain a Galois representation $\widetilde\rho:\Gal_K\rightarrow G(\cO''[\![T]\!])$ by considering the conjugation map $z_t\rho z_t^{-1}:\Gal_K\rightarrow G(\cO''[T])$.  This induces a homomorphism $R_{\rhobar}^{\square,\tau,\mathbf{v}}\otimes_{\cO}\cO''\rightarrow \cO''[\![T]\!]$, which in turn induces a morphism of rigid spaces $\Spf(\cO''[\![T]\!])^{\mathrm{rig}}\rightarrow \Spf (R_{\rhobar}^{\square,\tau,\mathbf{v}}\otimes_{\cO}\cO'')^{\mathrm{rig}}$.  Since the source is irreducible and its image contains points corresponding to both $\rho$ and $h_z\rho h_z^{-1}$, they lie on the same irreducible component of $\Spf (R_{\rhobar}^{\square,\tau,\mathbf{v}}\otimes_{\cO}\cO'')^{\mathrm{rig}}$.

Thus, we may assume that $h\in U_{G^\circ}(\lambda)$.  By definition, if $A$
is an $\cO'$-algebra, $U_{G^\circ}(\lambda)(A)=\{g\in G^\circ(A) |
\lim_{t\rightarrow 0}\lambda(t)g\lambda(t)^{-1}=1\}$, so conjugating
$h$ by $\lambda$ induces a map $u_t:\A_{\cO''}^1\rightarrow G_{\cO''}$
with $u_1=h$ and $u_0=1$.  We therefore obtain a Galois representation
$\widetilde\rho':\Gal_K\rightarrow G(\cO''\langle T\rangle)$ by
$l$-adically completing the map $u_t\rho u_t^{-1}:\Gal_K\rightarrow
G(\cO''[T])$.  Since $u_t$ is the identity modulo $\varpi$,
$\widetilde\rho'$ in fact lands in $G(\cO''\langle \varpi T\rangle)$, and
therefore in $G(\cO''[\![\varpi T]\!])$.  This induces a map
$R_{\rhobar}^{\square,\tau,\mathbf{v}}\otimes_{\cO}\cO''\rightarrow \cO''[\![\varpi T]\!]$,
and therefore a morphism of rigid spaces
$\Spf(\cO''[\![\varpi T]\!])^{\mathrm{rig}}\rightarrow\Spf(R_{\rhobar}^{\square,\tau,\mathbf{v}}\otimes_{\cO}\cO'')^{\mathrm{rig}}$.
Since the source is irreducible and its image contains points
corresponding to both $\rho$ and $h_u\rho h_u^{-1}$, they lie on the
same irreducible component of
$\Spf(R_{\rhobar}^{\square,\tau,\mathbf{v}}\otimes_{\cO}\cO'')^{\mathrm{rig}}$, as required.
\end{proof}


\subsection{Tensor products of components, and base change}\label{subsec: tensor
  products of components}
By a ``component for~$\rhobar$'' we mean a choice of~$\tau$
and~$\mathbf{v}$ (in the case~$l=p$) such
that~$R_{\rhobar}^{\square,\tau,\mathbf{v}}[1/l]\ne 0$, and a choice of an
irreducible component
of~$\Spec R_{\rhobar}^{\square,\tau,\mathbf{v}}[1/l]$. 

Let $\rbar:\Gal_K\to\GL_n(\F),\sbar:\Gal_K\to\GL_m(\F)$ be representations, let $C$ be a
component for $\rbar$ and let $D$ be a component for~$\sbar$. Let
$K'/K$ be a finite extension. The following lemma will be useful in section~\ref{sec: unitary groups}.
\begin{lem}
  \label{lem: operations on components make sense}There is a unique
  component $C\otimes D$ for $\rbar\otimes\sbar$ with the property
  that, if $r:\Gal_K\to\GL_n(\Qlbar)$ and $s:\Gal_K\to\GL_m(\Qlbar)$
  correspond to closed points of $C$ and $D$ respectively, then
  $r\otimes s$ corresponds to a closed point of
  $C\otimes D$. Similarly, there is a unique  component
  $C|_{K'}$ for $\rbar|_{\Gal_{K'}}$ such that for all~$r$, $r|_{\Gal_{K'}}$ corresponds
  to a closed point of $C|_{K'}$.
\end{lem}
\begin{proof}
  If a point of $\Spec R_{\rbar}^{\square,\tau,\mathbf{v}}[1/l]$ or a point of $\Spec R_{\rbar\otimes \sbar}^{\square,\tau,\mathbf{v}}[1/l]$ is smooth, then it lies on a unique irreducible component.  Then the first part  follows as in the proof of Theorem~\ref{thm: dense
    set of very smooth points}, replacing the appeal to Corollary~\ref{cor: WD stack is generically very smooth zero dimensional
  controlled by H2} with one to Theorem~\ref{thm:
    smoothness for functorial maps}, applied to the tensor product map
  \[\GL_n\times\GL_m\to\GL_{nm}.\]  The second part follows from
  Theorem~\ref{thm: dense set of very smooth points} (more precisely,
  from the existence of very smooth points on each irreducible component).
\end{proof}
In the setting of the previous lemma, we will sometimes say that the component
$C\otimes D$ is the tensor product of the components $C$ and $D$, and
that $C|_{K'}$ is the base change to~$K'$ of the component~$C$.
\section{Global deformation rings}\label{sec: global deformation
  rings}
\subsection{A result of Balaji}\label{subsec: Balaji}In this section
we recall one of the main results of~\cite{MR3152673}, which we will then
combine with the results of section~\ref{sec: local deformation
  rings} to prove Proposition~\ref{prop: global deformation ring with types is positive
    dimensional}, which gives a lower bound for the dimension
of certain global deformation rings. In~\cite[\S 4.2]{MR3152673} the
group~$G$ is assumed to be connected, but this is unnecessary. Indeed, the
assumption is only made in order to use the results of~\cite[\S
5]{MR1643682}, where it is also assumed that~$G$ is connected;
however, this assumption is never used in any of the arguments
of~\cite[\S 5]{MR1643682}, which apply unchanged to
general~$G$. Accordingly, we will freely use the results of~\cite[\S
4.2]{MR3152673} without assuming that~$G$ is connected. We assume in
this section that~$E$  is taken large enough that~$G_E$ is quasisplit.

Let~$F$ be a number field, and let~$S$ be a finite set of places of~$F$
containing all of the places dividing~$l\infty$. We work in the fixed determinant setting, and accordingly we
fix homomorphisms $\rhobar:\Gal_{F,S}\to G(\F)$
and~$\psi:\Gal_{F,S}\to\Gab(\cO)$ such that $\ab\circ\rhobar=\psibar$.


Write
$R^{\square,\psi}_{F,S}\in\CNL_\cO$ for the universal
fixed determinant framed deformation $\cO$-algebra
of~$\rhobar$. 
Let~$\Sigma\subset S$ be a subset containing all of the places lying
over~$l$. For each~$v\in\Sigma$, we let~$R_v^{\square,\psi}$ denote
the 
universal fixed determinant framed deformation $\cO$-algebra
of~$\rhobar|_{\Gal_{F_v}}$, and we set~$R_\Sigma^{\square,\psi}:=\wotimes_{v\in\Sigma,\cO}
R_v^{\square,\psi}$.

The following result is a special case of~\cite[Prop.\ 4.2.5]{MR3152673}.

\begin{prop}
  \label{prop: presentation of global over local Balaji}Suppose that
  $H^0(\Gal_{F,S},(\fg_\F^0)^*(1))=0$, and 
  let \[s:=(|\Sigma|-1)\dim_{\F}\fg_\F^0
    +\sum_{v|\infty,v\notin\Sigma}\dim_\F H^0(\Gal_{F_v},\fg_\F^0). \] Then for some~$r\ge 0$ there is a
  presentation \[R_{F,S}^{\square,\psi}\isoto R_{\Sigma}^{\square,\psi}[\![x_1,\dots,x_r]\!]/(f_1,\dots,f_{r+s}).\]
\end{prop}

\subsection{Global deformation rings of fixed type}\label{subsec:
  Balaji with types}In this section we combine our local results with Proposition~\ref{prop:
  presentation of global over local Balaji} to prove a lower bound for
the Krull dimension of a global deformation ring, following Balaji. This lower bound
will only be non-trivial in the following setting. 

\begin{defn}
  \label{defn: oddness}If $l>2$ then we say that~$\rhobar$ is \emph{discrete series
    and odd} if~$F$ is totally real, and if for all places~$v|\infty$
  of~$F$ we have $\dim_\F H^0(\Gal_{F_v},\fg_\F^0)=\dim_E G-\dim_EB$,
  where~$B$ is a Borel subgroup of~$G$.
\end{defn}
\begin{rem}
  \label{rem: relationship to other definitions of oddness}Recall that
  we chose~$E$ to be large enough that~$G_E$ is quasisplit, so this
  definition makes sense. The
  condition that~$\rhobar$ is discrete series and odd is  needed to make
  the usual Taylor--Wiles method work; see the introduction
  to~\cite{CHT}. If~$G$ is the $L$-group of a simply connected
  group 
  then one can check that this condition is equivalent to~$F$ being
  totally real and~$\rhobar$ being odd in the sense of~\cite{grossodd}
  (cf.\ \cite[Lem.\ 4.3.1]{MR3152673}). We use the term ``discrete
  series'' because the (conjectural) Galois representations associated
  to tempered automorphic representations which are discrete series at
  infinite places are expected to satisfy this property; see
  section~\ref{sec: unitary groups} for an example of this,
  and~\cite{grossodd} for a more general discussion.
\end{rem}

\begin{defn}\label{defn: regular Hodge type}
  We say that a $p$-adic Hodge type $\mathbf{v}$ is \emph{regular}
  if the conjugacy class~$P_\mathbf{v}$ consists of parabolic subgroups of $\Res_{E\otimes K/E}G$ whose connected components are Borel subgroups of $(\Res_{E\otimes K/E}G)^\circ$.
  \end{defn}
  \begin{rem}
    \label{rem: regular Hodge type GLn}If~$G=\GL_n$ then
    Definition~\ref{defn: regular Hodge type} is equivalent to the
    usual definition, that for each embedding $K\into E$ the
    Hodge--Tate weights are pairwise distinct.
  \end{rem}
\begin{rem}
If $E'/E$ is a field extension, then $(\Res_{E\otimes K/E}G)_{E'}\cong \Res_{E'\otimes K/E'}G$.  Furthermore, the formation of $P_{\Res_{E\otimes K/E}G}(\lambda)$ is compatible with extension of scalars from $E$ to $E'$.  Thus, if $\mathbf{v}$ is regular after extending scalars, it was regular over $E$ (and $\Res_{E\otimes K/E}G$ is automatically quasisplit).
\end{rem}

Write~$S^\infty$ for the set of finite places in~$S$.
For each  place~$v\in S^\infty$, we fix an inertial type~$\tau_v$, and
if~$v|l$ then we fix a Hodge type~$\bv_v$. If $v\nmid l$ (resp.\ if
$v|l$), we let~$\Rbar_v$ be a quotient
of the corresponding fixed determinant framed deformation
ring~$R_{\rhobar|_{\Gal_{F_v}}}^{\square,\tau_v,\psi}$ (resp.\
$R_{\rhobar|_{\Gal_{F_v}}}^{\square,\tau_v,\bv_v,\psi}$) corresponding
to a non-empty
union of irreducible components of the generic fiber. Set
$R^{\square,\univ}:=R_{F,S}^{\square,\psi}\otimes_{R_{\Sigma}^{\square,\psi},\cO}\wotimes_{v\in
S^\infty}\Rbar_v$; this is nonzero, because we are assuming that
each~$\Rbar_v$ is nonzero.

Assume that~$H^0(\Gal_{F,S},\fg_\F)=\fz_\F$, so
that 
~$\rhobar$ admits a universal fixed determinant deformation
$\cO$-algebra $R^\psi_{F,S}\in\CNL_\cO$, and write~$R^\univ$ for the
quotient of~$R_{F,S}$ corresponding to~$R^{\square,\univ}$ (as in the discussion preceding~\cite[Lemma 1.3.3]{BLGGT}, this quotient
exists by Lemma~\ref{lem: conjugation preserves
  components}). In the case that we fix potentially crystalline types
at the places $v|l$, and do not fix types at places away from~$l$, the
following result is~\cite[Thm.\ 4.3.2]{MR3152673}; the general case
follows from the same arguments as those of Balaji, given the input of
our local results.

\begin{prop}
  \label{prop: global deformation ring with types is positive
    dimensional}
  Assume that $l>2$, that~$\rhobar$ is discrete series and odd
  \emph{(}so that in particular~$F$ is totally real\emph{)}, and that
  $H^0(\Gal_{F,S},(\fg_\F^0)^*(1))=0$. Maintain our assumption that
  the local deformation rings~$\Rbar_v$ are nonzero.

  Suppose that for each place~$v|l$ the Hodge type~$\bv_v$ is regular.
  Then~$R^\univ$ has Krull dimension at least one.
\end{prop}
\begin{proof} By Proposition~\ref{prop: presentation of global over
    local Balaji} (taking~$\Sigma=S^\infty$) we see that for some
  $r\ge \dim_\F\fg_{\F}^0$ we have a presentation

\[R^{\square,\univ}\isoto \left(\wotimes_{v\in S^\infty}\Rbar_v\right)[\![x_1,\dots,x_{r}]\!]/(f_1,\dots,f_{r+s})\]where
 \[s=(|S^\infty|-1)\dim_{\F}\fg_\F^0
    +\sum_{v|\infty}\dim_\F H^0(\Gal_{F_v},\fg_\F^0). \] Since  $R^{\square,\univ}$ is formally smooth over~$R^{\univ}$
  of relative dimension~$\dim_\F\fg_{\F}^0$, it follows that
  the Krull dimension of~$R^\univ$ is at least \[\dim \wotimes_{v\in S^\infty,\cO}\Rbar_v-|S^\infty|\dim_{\F}\fg_\F^0
    -\sum_{v|\infty}\dim_\F H^0(\Gal_{F_v},\fg_\F^0), \] which by
  Theorem~\ref{thm: dense set of very smooth points}, and our assumption that each Hodge type~$\bv_v$
  is regular, is equal to \[1+ \sum_{v|p}[F_v:\Qp]\dim_E G/B -\sum_{v|\infty}\dim_\F
    H^0(\Gal_{F_v},\fg_\F^0), \]which in turn (by the assumption
  that~$\rhobar$ is discrete series and odd) equals~$1$, as required.
\end{proof}
\section{Unitary groups}\label{sec: unitary groups}
\subsection{The group~$\cG_n$}Let~$F$ be a CM field with maximal totally real
subfield~$F^+$. In this section we generalise some results of
~\cite{BLGGT} on the deformation theory of Galois representations
associated to polarised representations
of~$\Gal_F$, by allowing ramification at primes of~$F^+$ which are
inert or ramified in~$F$. This allows us to make cleaner statements,
and is also useful in applications; for example, in Theorem~\ref{thm:
  weight part of Serre for U2} we remove a ``split ramification''
condition in 
the proof of the weight part of Serre's conjecture for rank two unitary groups. Our
results are also needed in~\cite{CEGglobalrealisable}, where they are
used to construct lifts with specified ramification at certain places
of~$F^+$ which are inert in~$F$.


Recall from~\cite{CHT} the reductive group $\cG_n$
over~$\Z$ given by the semi-direct product of $\cG_n^0=\GL_n \times \GL_1$ by the group $\{1 , \jmath\}$
where 
\[ \jmath (g,a) \jmath^{-1}=(a(g^t)^{-1},a). \]
We let
$\nu:\cG_n \to \GL_1$ be the character which sends $(g,a)$ to $a$ and
sends $\jmath$ to $-1$. Our results in this section are for the most
part a straightforward application of the results of the earlier
sections to the particular case~$G=\cG_n$, but we need to begin by
comparing our definitions to those of~\cite{CHT}; we will follow the
notation of~\cite{CHT} where possible.

Fix a place~$v|\infty$. By~\cite[Lem.\ 2.1.1]{CHT}, for any ring~$R$
there is a natural bijection between the set of homomorphisms
$\rho:\Gal_{F^+}\to\cG_n(R)$ inducing an isomorphism
$\Gal_{F^+}/\Gal_F\isoto\cG_n/\cG_n^0$, and the set of triples
$(r,\mu,\langle,\rangle)$ where $r:\Gal_F\to\GL_n(R)$,
$\mu:\Gal_{F^+}\to R^\times$, and $\langle,\rangle:R^n\times R^n\to R$
is a perfect $R$-linear pairing such that
$\langle x,y\rangle=-\mu(c_v)\langle y,x\rangle$, and
$\langle r(\delta)x,r^{c_v}(\delta)y\rangle=\mu(\delta)\langle
x,y\rangle$ for all~$\delta\in\Gal_F$. We refer to such a triple as a
$\mu$-polarised representation of~$\Gal_F$, and we will sometimes
denote it as a pair~$(r,\mu)$, the pairing being implicit.

This bijection is given by setting~$r:=\rho|_{\Gal_F}$ (more
precisely, the projection of~$\rho|_{\Gal_F}$ to~$\GL_n(R)$),
$\mu:=\nu\circ\rho$, and $\langle x,y\rangle=x^t A^{-1}y$,
where $\rho(c_v)=(A,-\mu(c_v))\jmath$. If~$v$ is a finite place
of~$F^+$ which is inert or ramified in~$F$, then we have an induced
bijection between representations $\Gal_{F^+_v}\to\cG_n(R)$ and
$\mu$-polarised representations $\Gal_{F_v}\to\GL_n(R)$.

There is an isomorphism $\GL_1\to Z_{\cG_n}$ given by $g\mapsto
(g,g^2)\in \GL_1\to\GL_1\subset \GL_n\times\GL_1$, and we
have 
$\cG_n^\der=\GL_n\times 1$, and
$\cG_n^\ab=\GL_1\times \{1 , \jmath\}$. (It is easy to check by direct
calculation that $\cG_n^\der\subset \cG_n^\circ$, and
indeed $\cG_n^\der\subset\GL_n\times 1$. Since $\GL_n^{\der}=\SL_n$,
we have $\SL_n\times 1\subset\cG_n^\der$, and since
$\jmath(1,a)\jmath^{-1}(1,a^{-1})=(a,1)$, we also have $\GL_1\times
1\subset\cG_n^\der$, whence $\GL_n\times
1\subset\cG_n^\der$. Similarly, one checks easily that
$Z_{\cG_n}\subset \cG_n^\circ$, so that $Z_{\cG_n}\subset
\GL_1\times\GL_1$. If $(g,a)\in\GL_1\times\GL_1$ then
$\jmath(g,a)\jmath^{-1}=(ag^{-1},a)$, so we see that $(g,a)\in
Z_{\cG_n}$ if and only if $a=g^2$, as required.)

 We fix a prime $l>2$ and a representation
$\rhobar:\Gal_{F^+}\to\cG_n(\F)$ with
$\rhobar^{-1}(\cG_n^0(\F))=\Gal_F$. We fix a character $\mu:\Gal_{F^+}\to\cO^\times$ with
$\nu\circ\rhobar=\mubar$. Write $\psi:\Gal_{F^+}\to\cG_n^\ab(\cO)$ for
the character taking~$g\in \Gal_F$ to $(\mu(g),1)$
and~$g\in\Gal_{F^+}\setminus\Gal_F$ to
$(-\mu(g),\jmath)$. 

Note that if~$R\in\CNL_\cO$ then a
deformation~$\rho:\Gal_{F^+}\to\cG_n(R)$ of~$\rhobar$ has
$\ab\circ\rho=\psi$ if and only if $\nu\circ\rho=\mu$, in which case
we say that it is $\mu$-polarised. By~\cite[Prop.\
2.2.3]{2016arXiv160103752A}, restriction to~$\Gal_F$ gives an
equivalence between the $\mu$-polarised (framed) deformations of~$\rhobar$ and
the~$\mu$-polarised (framed) deformations~$r$
of~$\rbar:=\rhobar|_{\Gal_F}:\Gal_F\to\GL_n(\F)$, the latter by definition
being those~$r$ which satisfy~$r^{c}\cong r^\vee\mu$ (where we are
writing~$c$ for~$c_v$, as~$r^c$ is independent of the choice
of~$v|\infty$). 

The same equivalence pertains to deformations
of~$\rhobar|_{\Gal_{F^+_v}}$, where~$v$ is inert or ramified
in~$F$. On the other hand, if~$v$ splits as~$\tv\tv^c$ in~$F$, then
restriction to~$\Gal_{F_{\tv}}$ gives an equivalence
between~$\mu$-polarised (framed) deformations
of~$\rhobar|_{\Gal_{F^+_v}}$ and (framed) deformations
of~$\rbar|_{\Gal_{F_{\tv}}}$; thus at such places the deformation
theory of representations valued in~$\cG_n$ is reduced to the case
of~$\GL_n$. It is for this reason that~\cite{CHT} and its sequels only
permit ramification at places which split in~$F$.


By~\cite[Lem.\ 2.1.3]{CHT}, $\rhobar$ is discrete series and odd in
the sense of Definition~\ref{defn: oddness} if and only if for each
place $v|\infty$ of~$F^+$ with corresponding complex conjugation~$c_v\in\Gal_{F^+}$
we have $\mubar(c_v)=-1$. This is by definition equivalent to the
corresponding polarised representation~$(\rhobar|_{\Gal_F},\mubar)$
being totally odd in the sense of~\cite[\S 2.1]{BLGGT}.
Let~$S$ be a finite set of places of~$F^+$, including all the places
where~$\rbar$ or~$\mu$ are ramified, all the infinite places, and all
the places dividing~$l$. The following is a generalisation
of~\cite[Prop.\ 1.5.1]{BLGGT} (which is the case that every finite
place in~$S$ splits in~$F$, and is actually proved in~\cite{CHT});
note that the assumption that~$\rhobar|_{\Gal_{F(\zeta_l}}$ is
absolutely irreducible is missing from the statement of~\cite[Prop.\
1.5.1]{BLGGT}, but should have been included there. 
Note also that this assumption implies
that~$\rhobar$ admits a universal deformation ring; indeed, we
have~$H^0(\Gal_{F^+},\fg_{\F})=H^0(\Gal_{F^+},\fgl_{n,\F}\times\fgl_{1,\F})=\fgl_{1,\F}$
by Schur's lemma (note that~$\Gal(F/F^+)$ acts by~$-1$ on the scalar
matrices in~$\fgl_{n,\F}$).

\begin{cor}
  \label{cor: global deformation ring is positive dimensional, unitary
  case}
  Let~$l>2$ be prime, and let~$\rhobar:\Gal_{F^+}\to\cG_n(\F) $ be such
  that~$\rhobar|_{\Gal_{F(\zeta_l)}}$ is absolutely irreducible. Assume that~$\rhobar$ is
  discrete series and odd.

Let~$\mu$ be a de Rham
  lift of~$\mubar$, and let ~$S$ be a finite set of 
  places of~$F^+$ including all the places at  that which either
  ~$\rbar$ or~$\mu$ is ramified, and all the places dividing~$l\infty$.
 For each finite place $v\in S$, fix an inertial type~$\tau_v$, and
 if~$v|l$, fix a regular Hodge type~$\bv_v$. Fix quotients of the
 corresponding local~$\mu$-polarised framed deformation rings which
 correspond to a \emph{(}non-empty\emph{)} union of irreducible components of the generic fiber.

  Let~$R^\univ$ be the universal deformation ring for $\mu$-polarised
  deformations of~$\rhobar$ which are unramified outside~$S$, and lie on
  the given union of irreducible components for each finite place ~$v\in S$. Then~$R^\univ$ has Krull
  dimension at least one.
\end{cor}
\begin{proof}By Proposition~\ref{prop: global deformation ring with
    types is positive dimensional}, we need only check that
  $H^0(\Gal_{F^+,S},(\fgl_{n,\F})^*(1))$ vanishes,
  where~$\fgl_{n,\F}$ is the Lie algebra of~$\cG_n^{\der}$. By
  inflation-restriction this group injects into
  $H^0(\Gal_{F(\zeta_l)},(\fgl_{n,\F})^*(1))^{\Gal(F(\zeta_l)/F^+)}=H^0(\Gal_{F(\zeta_l)},(\fgl_{n,\F}))^{\Gal(F(\zeta_l)/F^+)}$.
  Since~$\rhobar|_{\Gal_{F(\zeta_l)}}$ is absolutely irreducible by
  assumption, this group vanishes by Schur's lemma (noting again
  that~$\Gal(F/F^+)$ acts by~$-1$ on the scalar matrices
  in~$\fgl_{n,\F}$).
\end{proof}

\subsection{Existence of lifts and the weight part of Serre's conjecture}\label{subsec: existence of
  lifts}
We now prove a strengthening
of~\cite[Thm.\ A.4.1]{MR3072811}, removing the condition that the
places at which our Galois representations are ramified are split
in~$F$. We refer the reader to~\cite{BLGGT} for any unfamiliar
terminology; in particular, potential diagonalizability is defined
in~\cite[\S 1.4]{BLGGT}, while adequacy and the notion of a polarised
Galois representation being potentially diagonalizably automorphic are
defined in~\cite[\S 2.1]{BLGGT}. 
\begin{thm}\label{diaglift} Let $l$ be an
  odd prime not dividing~$n$, and suppose that~$\zeta_l\notin F$. Let~$\rhobar:\Gal_{F^+}\to\cG_n(\F) $ be such
  that~$\rhobar|_{\Gal_{F(\zeta_l)}}$ is absolutely irreducible. Assume that~$\rhobar$ is
  discrete series and odd. 
  Let $S$ be a finite set of places of $F^+$,
  including all places dividing~$l\infty$. 

Let~$\mu$ be a de Rham
  lift of~$\mubar$, and let ~$S$ be a finite set of 
  places of~$F^+$ including all the places at  that which either
  ~$\rbar$ or~$\mu$ is ramified, and all the places dividing~$l\infty$.
 For each finite place $v\in S$, fix an inertial type~$\tau_v$, and
 if~$v|l$, fix a regular Hodge type~$\bv_v$. Fix quotients of the
 corresponding local~$\mu$-polarised framed deformation rings which
 correspond to an irreducible component of the generic fiber;
 if~$v|l$, assume also that this component is potentially
 diagonalizable.




Assume further that there is a finite extension of CM fields $F'/F$
such that~ $F'$ does not contain $\zeta_l$, all finite places
of~$(F')^+$ above~$S$ split in~$F$, and $\rhobar(\Gal_{F'(\zeta_l)})$ is
adequate; and that there exists a
lift~$\rho':\Gal_{F^+,S}\to\cG_n(\cO)$ of~$\rhobar|_{\Gal_{(F')^+,S}}$
with~$\nu\circ \rho'=\mu|_{\Gal_{F^+,S}}$, with the further property
that~$\rho'$ is potentially diagonalizably automorphic.

Then there
is a lift
\[ \rho:\Gal_{F^+,S} \longrightarrow \cG_n(\cO) \]
of $\rhobar$ such that
\begin{enumerate}
\item $\nu \circ \rho = \mu$;
\item if $v \in S$ is a finite place, then ${\rho}|_{G_{F^+_v}}$
  corresponds to a point on our chosen component of the local
  deformation ring.
\item $\rho|_{\Gal_{(F')^+,S}}$ is potentially diagonalizably automorphic.
\end{enumerate} \end{thm}
\begin{proof}  Let~$R^\univ$ be the universal deformation ring for $\mu$-polarised
  deformations of~$\rhobar$ which are unramified outside~$S$, and lie on
  the given  irreducible component for each finite place~$v\in S$. Then~$R^\univ$ has Krull
  dimension at least one by Corollary~\ref{cor: global deformation ring is positive dimensional, unitary
  case}. We claim that~$R^\univ$ is a finite~$\cO$-algebra. Admitting
this claim, we can choose a homomorphism~$R^\univ\to E$, and let~$\rho$
be the corresponding representation. This satisfies properties~(1)
and~(2) by construction.

Let~$R_{F'}^\univ$ be the universal
deformation ring for  $\mu|_{G_{(F')^+,S}}$-polarised
  deformations of~$\rbar|_{G_{F',S}}$ which lie on the base changes of
  our chosen components. 
  By~\cite[Lem.\ 1.2.3 (1)]{BLGGT}, $R^\univ$ is a
  finite $R_{F'}^\univ$-algebra, so in order to prove the claim it is enough to show that
  $R_{F'}^\univ$ is a finite $\cO$-algebra.

By~\cite[Thm.\ A.4.1]{MR3072811} (with~$F$ there taken to equal~$F'$),
there is a representation~$\rho'':G_{(F')^+,S}\to\cG_n(\cO)$
corresponding to an $\cO$-point of $R_{F'}^\univ$, which is
furthermore potentially diagonalizably automorphic. Then
$R_{F'}^\univ$ is a finite $\cO$-algebra by~\cite[Thm.\ 2.3.2]{BLGGT}.
as required. Finally, property~(3) holds by~\cite[Thm.\ 2.3.2]{BLGGT}
(applied to~$\rho''$ and~$\rho|_{G_{(F')^+,S}}$). 
\end{proof}
We now apply this result to the weight part of Serre's conjecture for
unitary groups. We
restrict ourselves to the case~$n=2$, where the existing results in
the literature are strongest; our results should also allow the
removal of the hypothesis of ``split ramification'' from results in
the literature for higher rank unitary groups, such as the results
of~\cite{blggUn}. We recall that if~$K/\Ql$ is a finite extension,
there is associated to any representation $\rhobar:\Gal_K\to\GL_2(\F)$
a set~$W(\rhobar)$ of Serre weights. A definition of~$W(\rhobar)$ was
first given in~\cite{bdj} in the case that $K/\Ql$ is unramified, and
various generalisations and alternative definitions have subsequently
been proposed. As a result of the main theorems
of~\cite{MR3324938,2016arXiv160806059C}, all of these definitions are
equivalent; we refer the reader to the introductions to those papers
for a discussion of the various definitions.

Suppose that~$F$ is an imaginary CM field with maximal totally real
subfield~$F^+$, such that~$F/F^+$ is unramified at all finite places,
that each place of~$F^+$ above~$l$ splits in~$F$, and that~$[F^+:\Q]$
is even. Then as in~\cite{MR3072811} we have a unitary group~$G/F^+$
which is quasisplit at all finite places and compact at all infinite
places. If $\rbar:\Gal_{F^+}\to\cG_2(\Flbar)$ is irreducible, the
notion of~$\rbar$ being modular of a Serre weight is defined
in~\cite[Defn.\ 2.1.9]{MR3072811}. This definition (implicitly)
insists that~$\rbar$ is only ramified at places which split in~$F$,
and we relax it as follows: we change the definition of a good compact
open subgroup $U\subset G(\A_{F^+}^\infty)$ in~\cite[Defn.\
2.1.5]{MR3072811} to require only that at all places~$v|l$ we
have~$U_v=G(\cO_{F^+_v})$, and at all places~$v\nmid l$ we have
$U_v\subset G(\cO_{F^+_v})$. (Consequently, we are now considering
automorphic forms of arbitrary level away from~$l$, whereas
in~\cite{MR3072811} the level is hyperspecial at all places which do
not split in~$F$.)

Having made this change, everything in~\cite[\S 2]{MR3072811} goes
through unchanged, except that all mentions of ``split ramification''
can be deleted. The following theorem strengthens~\cite[Thm.\ A]{MR3164985}, removing
a hypothesis on the ramification away from~$l$ (and also a hypothesis
on the ramification at~$l$, although that could already have been
removed thanks to the results of~\cite{MR3324938}).
 \begin{thm}
  \label{thm: weight part of Serre for U2}
  Let~$F$ be an imaginary CM field with maximal totally real
  subfield~$F^+$, and suppose that~$F/F^+$ is unramified at all finite
  places, that each place of~$F^+$ above~$l$ splits in~$F$, and
  that~$[F^+:\Q]$ is even. Suppose that~$l$ is odd, that
  $\rbar:G_{F^+}\to\cG_2(\Flbar)$ is irreducible and modular, and
  that~$\rbar(G_{F(\zeta_l)})$ is adequate.

Then the set of Serre weights for which~$\rbar$ is modular is exactly
the set of weights given by the sets~$W(\rbar|_{G_{F_v}})$, $v|l$.
\end{thm}
\begin{proof}We begin by observing that the proof of~\cite[Thm.\
  5.1.3]{MR3072811} goes through in our more general context (that is,
  without assuming ``split ramification''). Indeed,  we have
  already observed that the results of~\cite[\S 2]{MR3072811} are valid in
  our context, and chasing back
  through the references, we see that the only change that needs to be
  made is to relax the hypotheses in~\cite[Thm.\
  3.1.3]{MR3072811} by no longer requiring that the places~$v\in S$,
  $v\nmid l$ split in~$F$. This follows by replacing the citation of~\cite[Thm.\
  A.4.1]{MR3072811} in the proof of ~\cite[Thm.\
  3.1.3]{MR3072811} with a reference to Theorem~\ref{diaglift}
  above (after making a further extension of~$F'$ to arrange
  that all of the places of~$(F')^+$ lying over~$S$ split in~$F'$). 

This shows that~$\rbar$ is modular of every weight given by
the~$W(\rbar|_{G_{F_v}})$, $v|l$. For the converse, observe
that~\cite[Cor.\ 4.1.8]{MR3072811} also holds in our context (again,
since the results of~\cite[\S 2]{MR3072811} go through); the result
then follows immediately from~\cite[Thm.\ 6.1.8]{MR3324938}.
\end{proof}

\begin{rem}
  \label{rem: non quasi split groups}It is presumably possible to
  prove in the same way a further strengthening of Theorem~\ref{thm: weight part of
    Serre for U2} where we allow our unitary group to be ramified at
  some finite places (and thus allow~$[F^+:\Q]$ to be odd, and~$F/F^+$
  to be ramified at some finite places), but to do so would involve a
  lengthier discussion of automorphic representations on unitary
  groups, which would take us too far afield. 
\end{rem}

\begin{rem}
  \label{rem: allowing places above p to ramify}We have assumed that
  the places of~$F^+$ above~$l$ split in~$F$, because the weight part
  of Serre's conjecture has not been considered in the literature for
  unitary groups which do not split above~$l$ (although if $l$ is
  unramified in~$F$, and we are in the generic semisimple case, such a
  conjecture is a special case of the conjectures
  of~\cite{2015arXiv150902527G}). However, it seems likely that it is
  possible to formulate and prove a generalisation of
  Theorem~\ref{thm: weight part of Serre for U2} which removes this
  assumption, following the ideas of~\cite{MR3292675}
  and~\cite{MR3449190} (that is, using the Breuil--M\'ezard conjecture
  for potentially Barsotti--Tate representations). Again, this would
  take us too far afield from the main concerns of this paper, so we
  do not pursue this; and in any case we understand that this will be
  carried out in forthcoming work of~Koziol and~ Morra.
\end{rem}


\bibliographystyle{amsalpha}
\bibliography{deformation}

\renewcommand{\MR}[1]{}\def\cftil#1{\ifmmode\setbox7\hbox{$\accent"5E#1$}\else
  \setbox7\hbox{\accent"5E#1}\penalty 10000\relax\fi\raise 1\ht7
  \hbox{\lower1.15ex\hbox to 1\wd7{\hss\accent"7E\hss}}\penalty 10000
  \hskip-1\wd7\penalty 10000\box7}
  \def\cftil#1{\ifmmode\setbox7\hbox{$\accent"5E#1$}\else
  \setbox7\hbox{\accent"5E#1}\penalty 10000\relax\fi\raise 1\ht7
  \hbox{\lower1.15ex\hbox to 1\wd7{\hss\accent"7E\hss}}\penalty 10000
  \hskip-1\wd7\penalty 10000\box7}
  \def\cftil#1{\ifmmode\setbox7\hbox{$\accent"5E#1$}\else
  \setbox7\hbox{\accent"5E#1}\penalty 10000\relax\fi\raise 1\ht7
  \hbox{\lower1.15ex\hbox to 1\wd7{\hss\accent"7E\hss}}\penalty 10000
  \hskip-1\wd7\penalty 10000\box7}
  \def\cftil#1{\ifmmode\setbox7\hbox{$\accent"5E#1$}\else
  \setbox7\hbox{\accent"5E#1}\penalty 10000\relax\fi\raise 1\ht7
  \hbox{\lower1.15ex\hbox to 1\wd7{\hss\accent"7E\hss}}\penalty 10000
  \hskip-1\wd7\penalty 10000\box7}
\providecommand{\bysame}{\leavevmode\hbox to3em{\hrulefill}\thinspace}
\providecommand{\MR}{\relax\ifhmode\unskip\space\fi MR }
\providecommand{\MRhref}[2]{%
  \href{http://www.ams.org/mathscinet-getitem?mr=#1}{#2}
}
\providecommand{\href}[2]{#2}
\begin{thebibliography}{BLGGT14}

\bibitem[All]{2016arXiv160103752A}
Patrick~B. Allen, \emph{{On automorphic points in polarized deformation
  rings}}, Amer. J. Math. (to appear).

\bibitem[All16]{allen2014deformations}
\bysame, \emph{Deformations of polarized automorphic {G}alois representations
  and adjoint {S}elmer groups}, Duke Math. J. \textbf{165} (2016), no.~13,
  2407--2460. \MR{3546966}

\bibitem[Bal12]{MR3152673}
Sundeep Balaji, \emph{G-valued potentially semi-stable deformation rings},
  ProQuest LLC, Ann Arbor, MI, 2012, Thesis (Ph.D.)--The University of Chicago.
  \MR{3152673}

\bibitem[BCDT01]{MR1839918}
Christophe Breuil, Brian Conrad, Fred Diamond, and Richard Taylor, \emph{On the
  modularity of elliptic curves over {$\bold Q$}: wild 3-adic exercises}, J.
  Amer. Math. Soc. \textbf{14} (2001), no.~4, 843--939 (electronic).
  \MR{1839918 (2002d:11058)}

\bibitem[BDJ10]{bdj}
Kevin Buzzard, Fred Diamond, and Frazer Jarvis, \emph{On {S}erre's conjecture
  for mod {$l$} {G}alois representations over totally real fields}, Duke Math.
  J. \textbf{155} (2010), no.~1, 105--161.

\bibitem[Bel16]{2014arXiv1403.1411B}
Rebecca Bellovin, \emph{Generic smoothness for {$G$}-valued potentially
  semi-stable deformation rings}, Ann. Inst. Fourier (Grenoble) \textbf{66}
  (2016), no.~6, 2565--2620. \MR{3580181}

\bibitem[BG14]{MR3444225}
Kevin Buzzard and Toby Gee, \emph{The conjectural connections between
  automorphic representations and {G}alois representations}, Automorphic forms
  and {G}alois representations. {V}ol. 1, London Math. Soc. Lecture Note Ser.,
  vol. 414, Cambridge Univ. Press, Cambridge, 2014, pp.~135--187. \MR{3444225}

\bibitem[BLGG]{blggUn}
Tom Barnet-Lamb, Toby Gee, and David Geraghty, \emph{Serre weights for
  ${U}(n)$}, J. Reine Angew. Math. (to appear).

\bibitem[BLGG13]{MR3072811}
Thomas Barnet-Lamb, Toby Gee, and David Geraghty, \emph{Serre weights for rank
  two unitary groups}, Math. Ann. \textbf{356} (2013), no.~4, 1551--1598.
  \MR{3072811}

\bibitem[BLGGT14]{BLGGT}
Thomas Barnet-Lamb, Toby Gee, David Geraghty, and Richard Taylor,
  \emph{Potential automorphy and change of weight}, Ann. of Math. (2)
  \textbf{179} (2014), no.~2, 501--609. \MR{3152941}

\bibitem[B{\"o}c99]{MR1679172}
Gebhard B{\"o}ckle, \emph{A local-to-global principle for deformations of
  {G}alois representations}, J. Reine Angew. Math. \textbf{509} (1999),
  199--236. \MR{1679172}

\bibitem[{Boo}a]{2018arXiv180710743B}
J.~{Booher}, \emph{{Minimally Ramified Deformations when $\ell \neq p$}}, J.
  Number Theory (to appear).

\bibitem[{Boo}b]{2016arXiv161204237B}
\bysame, \emph{{Producing Geometric Deformations of Orthogonal and Symplectic
  Galois Representations}}, Compos. Math. (to appear).

\bibitem[Bou05]{MR2109105}
Nicolas Bourbaki, \emph{Lie groups and {L}ie algebras. {C}hapters 7--9},
  Elements of Mathematics (Berlin), Springer-Verlag, Berlin, 2005, Translated
  from the 1975 and 1982 French originals by Andrew Pressley. \MR{2109105}

\bibitem[BP]{2017arXiv170807434}
Jeremy Booher and Stefan Patrikis, \emph{{$G$}-valued {G}alois deformation
  rings when $\ell\neq p$}, Math. Res. Let. (to appear).

\bibitem[BS07]{MR2359853}
Christophe Breuil and Peter Schneider, \emph{First steps towards {$p$}-adic
  {L}anglands functoriality}, J. Reine Angew. Math. \textbf{610} (2007),
  149--180. \MR{2359853}

\bibitem[CEG]{CEGglobalrealisable}
Frank Calegari, Matthew Emerton, and Toby Gee, \emph{Globally realizable
  components of local deformation rings}.

\bibitem[CEGM]{2016arXiv160806059C}
F.~{Calegari}, M.~{Emerton}, T.~{Gee}, and L.~{Mavrides}, \emph{{Explicit
  {S}erre weights for two-dimensional {G}alois representations}}, Compos. Math.
  (to appear).

\bibitem[CGP15]{predbook}
Brian Conrad, Ofer Gabber, and Gopal Prasad, \emph{Pseudo-reductive groups},
  second ed., New Mathematical Monographs, vol.~26, Cambridge University Press,
  Cambridge, 2015. \MR{3362817}

\bibitem[CHT08]{CHT}
Laurent Clozel, Michael Harris, and Richard Taylor, \emph{Automorphy for some
  {$l$}-adic lifts of automorphic mod {$l$} {G}alois representations}, Publ.
  Math. Inst. Hautes \'Etudes Sci. (2008), no.~108, 1--181, With Appendix A,
  summarizing unpublished work of Russ Mann, and Appendix B by Marie-France
  Vign{\'e}ras. \MR{2470687 (2010j:11082)}

\bibitem[Con99]{conrad-rigid-irred-cpts}
Brian Conrad, \emph{Irreducible components of rigid spaces}, Ann. Inst. Fourier
  (Grenoble) \textbf{49} (1999), no.~2, 473--541. \MR{1697371}

\bibitem[Con14]{conrad-luminy}
\bysame, \emph{Reductive group schemes}, Autour des sch\'emas en groupes.
  {V}ol. {I}, Panor. Synth\`eses, vol. 42/43, Soc. Math. France, Paris, 2014,
  pp.~93--444. \MR{3362641}

\bibitem[dJ95]{dejong-formalrigid}
A.~J. de~Jong, \emph{Crystalline {D}ieudonn\'e module theory via formal and
  rigid geometry}, Inst. Hautes \'Etudes Sci. Publ. Math. (1995), no.~82, 5--96
  (1996). \MR{1383213}

\bibitem[DM82]{dm}
Pierre Deligne and J.S. Milne, \emph{Tannakian categories}, Hodge Cycles,
  Motives, and Shimura Varieties, LNM, vol. 900, Springer, New York, 1982,
  pp.~101--228.

\bibitem[EG17]{EGcomponents}
Matthew Emerton and Toby Gee, \emph{Dimension theory and components of
  algebraic stacks}.

\bibitem[Fon94]{fonl}
Jean-Marc Fontaine, \emph{Repr\'esentations {$l$}-adiques potentiellement
  semi-stables}, Ast\'erisque (1994), no.~223, 321--347, P{\'e}riodes
  $p$-adiques (Bures-sur-Yvette, 1988). \MR{1293977}

\bibitem[Gee11]{MR2785764}
Toby Gee, \emph{Automorphic lifts of prescribed types}, Math. Ann. \textbf{350}
  (2011), no.~1, 107--144. \MR{2785764 (2012c:11118)}

\bibitem[GG15]{MR3449190}
Toby Gee and David Geraghty, \emph{The {B}reuil-{M}\'ezard conjecture for
  quaternion algebras}, Ann. Inst. Fourier (Grenoble) \textbf{65} (2015),
  no.~4, 1557--1575. \MR{3449190}

\bibitem[GHS]{2015arXiv150902527G}
T.~{Gee}, F.~{Herzig}, and D.~{Savitt}, \emph{{General Serre weight
  conjectures}}, J. Eur. Math. Soc. (to appear).

\bibitem[GK14]{MR3292675}
Toby Gee and Mark Kisin, \emph{The {B}reuil-{M}\'ezard conjecture for
  potentially {B}arsotti-{T}ate representations}, Forum Math. Pi \textbf{2}
  (2014), e1, 56. \MR{3292675}

\bibitem[GLS14]{MR3164985}
Toby Gee, Tong Liu, and David Savitt, \emph{The {B}uzzard-{D}iamond-{J}arvis
  conjecture for unitary groups}, J. Amer. Math. Soc. \textbf{27} (2014),
  no.~2, 389--435. \MR{3164985}

\bibitem[GLS15]{MR3324938}
\bysame, \emph{The weight part of {S}erre's conjecture for {$\mathrm{GL}(2)$}},
  Forum Math. Pi \textbf{3} (2015), e2, 52. \MR{3324938}

\bibitem[GP92]{MR1186476}
Benedict~H. Gross and Dipendra Prasad, \emph{On the decomposition of a
  representation of {${\rm SO}_n$} when restricted to {${\rm SO}_{n-1}$}},
  Canad. J. Math. \textbf{44} (1992), no.~5, 974--1002. \MR{1186476}

\bibitem[GR10]{MR2730575}
Benedict~H. Gross and Mark Reeder, \emph{Arithmetic invariants of discrete
  {L}anglands parameters}, Duke Math. J. \textbf{154} (2010), no.~3, 431--508.
  \MR{2730575}

\bibitem[Gro64]{MR0173675}
A.~Grothendieck, \emph{\'{E}l\'ements de g\'eom\'etrie alg\'ebrique. {IV}.
  \'{E}tude locale des sch\'emas et des morphismes de sch\'emas. {I}}, Inst.
  Hautes \'Etudes Sci. Publ. Math. (1964), no.~20, 259. \MR{0173675 (30
  \#3885)}

\bibitem[Gro07]{grossodd}
Benedict~H. Gross, \emph{Odd {G}alois representations}, preprint, 2007.

\bibitem[Jan04]{jantzen}
Jens~Carsten Jantzen, \emph{Nilpotent orbits in representation theory}, Lie
  theory, Progr. Math., vol. 228, Birkh\"auser Boston, Boston, MA, 2004,
  pp.~1--211. \MR{2042689 (2005c:14055)}

\bibitem[Kis07]{MR2459302}
Mark Kisin, \emph{Modularity of 2-dimensional {G}alois representations},
  Current developments in mathematics, 2005, Int. Press, Somerville, MA, 2007,
  pp.~191--230. \MR{2459302 (2010a:11098)}

\bibitem[Kis08]{MR2373358}
\bysame, \emph{Potentially semi-stable deformation rings}, J. Amer. Math. Soc.
  \textbf{21} (2008), no.~2, 513--546. \MR{2373358 (2009c:11194)}

\bibitem[Kis09]{MR2600871}
\bysame, \emph{Moduli of finite flat group schemes, and modularity}, Ann. of
  Math. (2) \textbf{170} (2009), no.~3, 1085--1180. \MR{2600871 (2011g:11107)}

\bibitem[KW09]{MR2480604}
Chandrashekhar Khare and Jean-Pierre Wintenberger, \emph{On {S}erre's
  conjecture for 2-dimensional mod {$p$} representations of {${\rm
  Gal}(\overline{\Bbb Q}/\Bbb Q)$}}, Ann. of Math. (2) \textbf{169} (2009),
  no.~1, 229--253. \MR{2480604 (2009m:11077)}

\bibitem[Mat89]{matsumura}
Hideyuki Matsumura, \emph{Commutative ring theory}, second ed., Cambridge
  Studies in Advanced Mathematics, vol.~8, Cambridge University Press,
  Cambridge, 1989, Translated from the Japanese by M. Reid. \MR{1011461
  (90i:13001)}

\bibitem[Maz89]{MR1012172}
B.~Mazur, \emph{Deforming {G}alois representations}, Galois groups over {${\bf
  Q}$} ({B}erkeley, {CA}, 1987), Math. Sci. Res. Inst. Publ., vol.~16,
  Springer, New York, 1989, pp.~385--437. \MR{1012172}

\bibitem[McN04]{mcninch}
George~J. McNinch, \emph{Nilpotent orbits over ground fields of good
  characteristic}, Math. Ann. \textbf{329} (2004), no.~1, 49--85. \MR{2052869
  (2005j:17018)}

\bibitem[Pat16]{MR3529115}
Stefan Patrikis, \emph{Deformations of {G}alois representations and exceptional
  monodromy}, Invent. Math. \textbf{205} (2016), no.~2, 269--336. \MR{3529115}

\bibitem[PY02]{MR1893005}
Gopal Prasad and Jiu-Kang Yu, \emph{On finite group actions on reductive groups
  and buildings}, Invent. Math. \textbf{147} (2002), no.~3, 545--560.
  \MR{1893005}

\bibitem[Ram02]{MR1935843}
Ravi Ramakrishna, \emph{Deforming {G}alois representations and the conjectures
  of {S}erre and {F}ontaine-{M}azur}, Ann. of Math. (2) \textbf{156} (2002),
  no.~1, 115--154. \MR{1935843 (2003k:11092)}

\bibitem[SGA70]{SGA3.II}
\emph{Sch\'emas en groupes. {II}: {G}roupes de type multiplicatif, et structure
  des sch\'emas en groupes g\'en\'eraux}, S\'eminaire de G\'eom\'etrie
  Alg\'ebrique du Bois Marie 1962/64 (SGA 3). Dirig\'e par M. Demazure et A.
  Grothendieck. Lecture Notes in Mathematics, Vol. 152, Springer-Verlag,
  Berlin-New York, 1970. \MR{0274459}

\bibitem[SR72]{saavedra}
Neantro Saavedra~Rivano, \emph{Cat\'egories {T}annakiennes}, Lecture Notes in
  Mathematics, Vol. 265, Springer-Verlag, Berlin-New York, 1972. \MR{0338002}

\bibitem[{Sta}]{stacks-project}
The {Stacks Project Authors}, \emph{\itshape {S}tacks {P}roject},
  \url{http://stacks.math.columbia.edu}.

\bibitem[Til96]{MR1643682}
Jacques Tilouine, \emph{Deformations of {G}alois representations and {H}ecke
  algebras}, Published for The Mehta Research Institute of Mathematics and
  Mathematical Physics, Allahabad; by Narosa Publishing House, New Delhi, 1996.
  \MR{1643682}

\bibitem[TW95]{MR1333036}
Richard Taylor and Andrew Wiles, \emph{Ring-theoretic properties of certain
  {H}ecke algebras}, Ann. of Math. (2) \textbf{141} (1995), no.~3, 553--572.
  \MR{1333036 (96d:11072)}

\bibitem[Wil95]{MR1333035}
Andrew Wiles, \emph{Modular elliptic curves and {F}ermat's last theorem}, Ann.
  of Math. (2) \textbf{141} (1995), no.~3, 443--551. \MR{1333035 (96d:11071)}

\end{thebibliography}

 \end{document}